\documentclass[12pt, a4paper]{amsart}

\usepackage[latin1]{inputenc}

\usepackage{pdfsync}
\usepackage{amsmath,amsthm,amsfonts,amscd,amssymb,eucal,latexsym,dsfont,color}
\usepackage[all]{xy}
\usepackage{latexsym, amssymb, amscd, mathrsfs, graphics, fullpage}
\usepackage{hyperref}

\newcommand{\corr}[1]{\color{blue}#1\normalcolor}

\newcommand{\ga}{\mathfrak{a}}

\newcommand{\id}{\mathrm{id}}
\newcommand{\bfG}{\mathds{G}}
\newcommand{\DG}{D(\bfG)}
\DeclareMathOperator{\ad}{ad}

\newcommand{\rtensor}[3]{{}_{#1}\!\underset{#2}{\otimes}{}_{#3}}
\newcommand{\rast}[3]{{}_{#1}\!\underset{#2}{*}{}_{#3}}
\DeclareMathOperator{\Ad}{Ad}
\newcommand{\R}{\mathds{R}}
\newcommand{\N}{\mathds{N}}
\newcommand{\T}{\mathds{T}}
\newcommand{\SUq}{\mathrm{SU}_{q}(2)}
\newcommand{\Cu}{C_{0}^{\mathrm{u}}}
\newcommand{\Cr}{C_{0}^{\mathrm{r}}}

\swapnumbers
\newtheorem{theorem}[subsection]{Theorem}
\newtheorem{proposition}[subsection]{Proposition}
\newtheorem{corollary}[subsection]{Corollary}
\newtheorem{lemma}[subsection]{Lemma}

\newtheorem{subtheorem}[subsubsection]{Theorem}
\newtheorem{subproposition}[subsubsection]{Proposition}
\newtheorem{sublemma}[subsubsection]{Lemma}
\newtheorem{subcorollary}[subsubsection]{Corollary}

\theoremstyle{definition}
\newtheorem{subdefinition}[subsubsection]{Definition}
\newtheorem{subremark}[subsubsection]{Remark}
\newtheorem{definition}[subsection]{Definition}

\def\gG{\mathfrak{G}}
\def\gH{\mathfrak{H}}

\def\gM{\mathfrak{M}}
\def\gN{\mathfrak{N}}

\def\ga{\mathfrak{a}}
\def\gb{\mathfrak {b}}

\newcommand{\op}{\mathrm{o}}
\newcommand{\com}{\mathrm{c}}

\begin{document}

\title[Measured Quantum transformation groupoids] 
{Measured Quantum transformation groupoids}
\author{Michel Enock}
\author{Thomas Timmermann}
\address{Institut de Math\'ematiques de Jussieu\\Unit\'{e} Mixte Paris 6 / Paris 7 /
CNRS de Recherche 7586 \\Case 247, 4 place Jussieu, 75252 Paris Cedex, France}
 \email{michel.enock@math.cnrs.fr}
\address {Westf\"{a}lishes Wilhems-Universit\"{a}t M\"{u}nster\\FB 10 Mathematik
WWU M\"{u}nster
Einsteinstr. 62, 48149 M\"{u}nster, Germany}
\email{timmermt@uni-muenster.de}
\date{october 15}
\thanks{The second author was supported by the SFB 878 of the Deutsche Forschungsgemeinschaft}
\begin{abstract}
In this article, when $\bfG$ is a locally compact quantum group, we associate, to a braided-commutative $\bfG$-Yetter-Drinfel'd algebra $(N, \ga, \widehat{\ga})$ equipped  with a normal faithful semi-finite weight verifying some appropriate condition (in particular if it is invariant with respect to $\ga$, or to $\widehat{\ga}$), a structure of a measured quantum groupoid. The dual structure is then given by $(N, \widehat{\ga}, \ga)$. Examples are given, especially the situation of a quotient type co-ideal of a compact quantum group. This construction generalizes the standard construction of a transformation groupoid. Most of the results were announced by the second author in 2011, at a conference in Warsaw. 
 \end{abstract}

\maketitle

\tableofcontents
\newpage
\section{Introduction}
\label{intro}
\subsection{Locally compact quantum groups}
The theory of locally compact quantum groups, developed by J. Kustermans and S. Vaes (\cite{KV1}, \cite{KV2}), provides a comprehensive framework for the study of quantum groups in the setting of $C^*$-algebras and von Neumann algebras. It includes a far reaching generalization of the classical Pontrjagin duality of locally compact abelian groups, that covers all locally compact groups. Namely, if $G$ is a locally compact group, its dual $\widehat{G}$ will be the von Neumann algebra $\mathcal L(G)$ generated by the left regular representation $\lambda_G$ of $G$, equipped with a coproduct $\Gamma_G$ from $\mathcal L(G)$ on $\mathcal L(G)\otimes\mathcal L(G)$ defined, for all $s\in G$, by $\Gamma_G(\lambda_G(s))=\lambda_G(s)\otimes\lambda_G(s)$, and with a normal semi-finite faithful weight, called the Plancherel weight $\varphi_G$, associated via the Tomita-Takesaki construction, to the left Hilbert algebra defined by the algebra $\mathcal K(G)$ of continuous functions with compact support (with convolution as product), this weight $\varphi_G$ being left- and right-invariant with respect to $\Gamma_G$ (\cite{T2}, VII, 3).

This theory builds on many preceding works, by G. Kac, G. Kac and L. Vainerman, the first author and J.-M. Schwartz (\cite{ES1}, \cite{ES2}), S. Baaj and G. Skandalis (\cite{BS}), A. Van Daele, S. Woronowicz {\cite{W1}, \cite{W4}, \cite{W5}) and many others. See the monography written by the second author for a survey of that theory (\cite{Ti1}), and the introduction of \cite{ES2} for a sketch of the historical background. It seems to have reached now a stable situation, because it fits the needs of operator algebraists for many reasons: 

First,  the axioms of this theory are very simple and elegant: they
    can be given in both $C^*$-algebras and von Neumann algebras, and
    these two points of view are equivalent, as A. Weil had shown it
    was the fact for groups (namely any measurable group equipped with
    a left-invariant positive measure bears a topology which makes it
    locally compact, and this measure is then the Haar measure
    \cite{W}). In a von Neumann setting, a locally compact quantum
    group is just a von Neumann algebra, equipped with a
    co-associative coproduct, and two normal faithful semi-finite
    weights, one left-invariant with respect to that coproduct, and
    one right-invariant. Then, many other data are constructed, in
    particular a multiplicative unitary (as defined in \cite{BS})
    which is manageable (as defined in \cite{W5}).

Second, all
    preceeding attemps (\cite{ES2}, \cite{W4}) appear as particular
    cases of locally compact quantum groups; and many interesting
    examples were constructed (\cite{W2}, \cite{W3}, \cite{VV2}).

Third,  many constructions of harmonic analysis, or concerning
    group actions on $C^*$-algebras and von Neumann algebras, were
    generalized up to locally compact quantum groups (\cite{V}).

    Finally, many constructions made by algebraists at the level of
    Hopf $*$-algebras, or multipliers Hopf $*$-algebras can be
    generalized for locally compact quantum groups. This is the case,
    for instance, for Drinfel'd double of a quantum group (\cite{D}),
    and for Yetter-Drinfel'd algebras which were well-known in an
    algebraic approach in \cite{M}.

\subsection{Measured Quantum Groupoids}
In two articles (\cite{Val1},\cite{Val2}), J.-M. Vallin has introduced two notions (pseudo-multiplicative unitary, Hopf bimodule), in order to generalize, to the groupoid case, the classical notions of multiplicative unitary (\cite{BS}) and of a co-associative coproduct on a von Neumann algebra. Then, F. Lesieur (\cite{L}), starting from a Hopf bimodule, when there exist a left-invariant operator-valued weight and a right-invariant operator-valued weight, mimicking in that wider setting what was done in (\cite{KV1}, \cite{KV2}), obtained a pseudo-multiplicative unitary, and called ``measured quantum groupoids'' these objects. A new set of axioms had been given in an appendix of \cite{E2}. In \cite{E2} and \cite{E3}, most of the results given in \cite{V} were generalized up to measured quantum groupoids.

This theory, up to now, bears two defects:

 First, it is only a theory in a von Neumann algebra setting. The second author had made many attemps in order to provide a $C^*$-algebra version of it (see \cite{Ti1} for a survey); these attemps were fruitful, but not sufficient to complete a theory equivalent to the von Neumann one. 

 Second, there is a lack of interesting examples. For instance, the
transformation groupoid (i.e.\ the groupoid given by a locally compact group right acting on a locally compact space), which is the first non-trivial example of a groupoid (\cite{R} 1.2.a), had no quantum analog up to this article. 

\subsection{Measured Quantum Transformation Groupoid}
This article is devoted to the construction of a family of examples of measured quantum groupoids. Most of the results were announced in \cite{Ti2}. The key point, is, when looking at a transformation groupoid given by a locally compact group $G$ having a right action $\ga$ on a locally compact space $X$, to add the fact that the dual $\widehat{G}$ is trivially right acting also on $L^\infty (X)$, and that the triple $(L^\infty (X), \ga, \id)$ is a $G$-Yetter-Drinfel'd algebra, and, more precisely, a braided-commutative $G$-Yetter-Drinfel'd algebra. 

The aim of this article is to generalize the construction of transformation groupoids, using this remark which shows that this generalization is not to be found for any action of a locally compact quantum group, but for a braided-commutative $\bfG$-Yetter-Drinfel'd algebra. 

Then, for any locally compact quantum group $\bfG$, looking at any
braided-commutative Yetter-Drinfel'd algebra $(N, \ga,
\widehat{\ga})$, it is possible to put a structure of Hopf bimodule on
the crossed product $\bfG\ltimes_\ga N$, equipped with a
left-invariant operator-valued weight, and with a right-invariant
operator-valued weight. In order to get a measured quantum groupoid,
one has to choose on $N$ (which is the basis of the measured quantum
groupoid) a normal faithful semi-finite weight $\nu$ that satisfies
some condition with respect to the action $\ga$; for example, $\nu$
could be invariant with respect to $\ga$. It appears then that the dual measured quantum groupoid is the structure associated to the braided-commutative Yetter-Drinfel'd algebra $(N, \widehat{\ga}, \ga)$. 

In an algebraic framework, similar results were obtained in \cite{Lu} and \cite{BM}. It is also interesting to notice that, as for locally compact quantum groups, the framework of measured quantum groupoids appears to be a good structure in which the algebraic results can be generalized. 

The article is organized as follows:

 In chapter \ref{pre} are recalled all the necessary results needed: namely locally compact quantum groups (\ref{lcqg}), actions of locally compact quantum groups on a von Neumann algebra (\ref{action}), Drinfel'd double of a locally compact quantum group (\ref{DG}), Yetter-Drinfel'd algebras (\ref{YD}), and braided-commutative Yetter-Drinfel'd algebras (\ref{BC}). 

In chapter \ref{inv}, we study relatively invariant weights with respect to an action, and then invariant weights for a Yetter-Drinfel'd algebra, and prove that such a weight exists when the von Neumann algebra $N$ is properly infinite. 

 In chapter \ref{H}, we construct the Hopf-von Neumann structure
associated to a  braided-commutative $\bfG$-Yetter-Drinfeld algebra. The precise definition of such a structure is given in \ref{spatial} and \ref{defHopf}. We construct also a co-inverse of this Hopf-von Neumann structure. 

 In chapter \ref{MQG}, we study the conditions to put on the weight $\nu$ to construct a measured quantum groupoid
associated to a  braided-commutative $\bfG$-Yetter-Drinfeld
algebra. These conditions hold, in particular, if the weight $\nu$ is invariant with respect to $\ga$. The precise definition and properties of measured quantum groupoids are given in \ref{defMQG}, \ref{defW}, \ref{data}. 

 In chapter \ref{duality}, we obtain the dual of this measured quantum groupoid, which is the measured quantum groupoid obtained when permuting the actions $\ga$ and $\widehat{\ga}$. 

 Finally, in chapter \ref{examples}, we give several examples of measured quantum groupoids which can be constructed this way, and in chapter \ref{QM}, we study more carefully the case of a quotient type co-ideal of a compact quantum group: in that situation, one of the measured quantum groupoids constructed in \ref{NY} is Morita equivalent to the quantum subgroup.

 \section{Preliminaries}
 \label{pre}

\subsection{Locally compact quantum groups}
\label{lcqg}
A quadruplet $\bfG=(M, \Gamma, \varphi, \psi)$ is a \emph{locally
  compact quantum group} if:

(i) $M$ is a von Neumann algebra,

(ii) $\Gamma$ is an injective unital $*$-homomorphism from $M$ into
the von Neumann tensor product $M\otimes M$, called a \emph{coproduct},
satisfying $(\Gamma\otimes \id)\Gamma=(\id\otimes\Gamma)\Gamma$ (the
coproduct is called \emph{co-associative}),

(iii) $\varphi$ is a normal faithful semi-finite weight on $M^+$ which is \emph{left-invariant}, i.e.,
\[(\id\otimes\varphi)\Gamma(x)=\varphi(x)1_M \quad \text{for all } x \in\gM_\varphi^+;\]

(iv) $\psi$ is a normal faithful semi-finite weight on $M^+$ which is \emph{right-invariant}, i.e.,
\[(\psi\otimes \id)\Gamma(x)=\psi(x)1_M \quad \text{for all } x\in\gM_\psi^+.\]

In this definition (and in the following),  $\otimes$ means the von Neumann tensor product, $(\id\otimes\varphi)$ (resp. $(\psi\otimes \id)$) is an operator-valued weight from $M\otimes M$ to $M\otimes \mathds{C}$ (resp. $\mathds{C}\otimes M$). This is the definition of the von Neumann version of a locally compact quantum group (\cite{KV2}). See also, of course \cite{KV1}. 

We shall use the usual data  $H_\varphi$, $J_\varphi$, $\Delta_\varphi$ of Tomita-Takesaki theory associated to
the weight $\varphi$ (\cite{T2} chap.6 to 9, \cite{StZ}, chap.10,
\cite{St}, chap.1 and 2), which, for
simplification, we write as $H$, $J$, $\Delta$. We regard
$M$ as a von Neumann algebra on $H_\varphi$ and identify the opposite
von Neumann algebra $M^{\op}$ with the commutant $M'$. 

On the Hilbert tensor product $H\otimes H$, Kustermanns and Vaes
constructed a unitary $W$, called the \emph{fundamental unitary}, which satisfies the \emph{pentagonal equation}
\[W_{23}W_{12}=W_{12}W_{13}W_{23},\]
where, we use, as usual, the leg-numbering notation. This unitary
contains all the data of $\bfG$:  $M$ is the weak closure of the vector space (which is an algebra) $\{(\id\otimes\omega)(W): \omega\in B(H)_*\}$ and $\Gamma$ is given by (\cite{KV1} 3.17)
\[\Gamma(x)=W^*(1\otimes x)W \quad \text{for all } x\in M,\]
and
\[(\id\otimes\omega_{J_\varphi{\Lambda_{\varphi}}(y_1^*y_2),
  \Lambda_\varphi(x)})(W)=(\id\otimes\omega_{J_\varphi\Lambda_\varphi(y_2),
  J_\varphi(y_1)})\Gamma(x^*)\]
 for all $x$, $y_1$, $y_2$ in $\gN_\varphi$.
 It is then possible to construct an unital anti-$*$-automorphism $R$
 of $M$ which is involutive ($R^2=\id$), defined  by
\[R[(\id\otimes\omega_{\xi, \eta})(W)]=(\id\otimes\omega_{J\eta,
  J\xi})(W) \quad \text{for all } \xi,\eta\in H.\]
This map is a \emph{co-inverse} (often called the \emph{unitary antipode}), which means that
\[\Gamma\circ R=\varsigma\circ (R\otimes R)\circ \Gamma,\]
where $\varsigma$ is the flip of $M\otimes M$ (\cite{KV1}, 5.26).
It is straightforward to get that $\varphi\circ R$ is a right-invariant normal semi-finite faithful weight and, thanks to a unicity theorem, is therefore proportional to $\psi$. We shall always suppose that $\psi=\varphi\circ R$.

Associated to $(M,\Gamma)$ is  a \emph{dual} locally compact quantum
  group $(\widehat{M},\widehat{\Gamma})$, where $\widehat{M}$ is
the weak closure of the vector space (which is an algebra)
$\{(\omega\otimes \id)(W) : \omega\in B(H)_*\}$, and
$\widehat{\Gamma}$ is given by
\[\widehat{\Gamma}(y)=\sigma W(y\otimes 1)W^*\sigma \quad \text{for
  all } y\in \widehat{M}.\]
Here, $\sigma$ denotes the flip of $H\otimes H$.  Let
\begin{align*}
  \|\omega\|_\varphi&=\sup\{|\omega(x^*)| : x\in\gN_\varphi,
  \varphi(x^*x)\leq 1\}, &
 I_\varphi&=\{\omega\in M_* : \|\omega\|_\varphi < \infty\}.
\end{align*}
Then, it is possible to define a normal semi-finite faithful weight $\widehat{\varphi}$ on $\widehat{M}$ such that $\widehat{\varphi}((\omega\otimes \id)(W)^*(\omega\otimes \id)(W))=\|\omega\|_\varphi^2$ (\cite{KV1}8.13), and it is possible to prove that $\widehat{\varphi}$ is left-invariant with repect to $\widehat{\Gamma}$ (\cite{KV1}8.15). Moreover, the application $y\mapsto Jy^*J$ is a unital anti-$*$-automorphism $\widehat{R}$ of $\widehat{M}$, which is involutive ($\widehat{R}^2=\id$) and is a co-inverse. Therefore, $\widehat{\varphi}\circ\widehat{R}$ is right-invariant with respect to $\widehat{\Gamma}$. 

Therefore $\widehat{\bfG}=(\widehat{M}, \widehat{\Gamma},
\widehat{\varphi}, \widehat{\varphi}\circ \widehat{R})$ is a locally
compact quantum group, called the \emph{dual} of $\bfG$. Its
multiplicative unitary $\widehat{W}$ is equal to $\sigma W^*\sigma$.
The bidual locally compact quantum group $\widehat{\widehat{\bfG}}$ is
equal to $\bfG$. In particular, the construction of the dual weight,
when applied to $\widehat{\bfG}$ gives that, for any $\omega$ in
$\widehat{M}_*$, $(\id\otimes\omega)(W^*)$ belongs to $\gN_\varphi$ if
and only if $\omega$ belongs to $I_{\widehat{\varphi}}$, and we have
then
$\|\Lambda_\varphi((\id\otimes\omega)(W^*))\|=\|\omega\|_{\widehat{\varphi}}$.

The Hilbert space $H_{\widehat{\varphi}}$ is isomorphic to (and will be identified with) $H$. For simplification, we write $\widehat{J}$ for $J_{\widehat{\varphi}}$ and $\widehat{\Delta}$ for $\Delta_{\widehat{\varphi}}$; we have, for all $x\in M$, $R(x)=\widehat{J}x^*\widehat{J}$ (\cite{KV2} 2.1). The operator $W$ satisfies
\[(\widehat{\Delta}^{it}\otimes\Delta^{it})W(\widehat{\Delta}^{-it}\otimes\Delta^{-it})=W\]
and $(\widehat{J}\otimes J)W(\widehat{J}\otimes J)=W^*$. 

Associated to $(M,\Gamma)$ is \emph{a scaling group}, which is a one-parameter group $\tau_t$ of automorphisms of $M$, such that (\cite{KV2} 2.1), for all $x\in M$, $t\in\R$, we have $\tau_t(x)=\widehat{\Delta}^{it}x\widehat{\Delta}^{-it}$, satisfying $\Gamma\circ\tau_t=(\tau_t\otimes\tau_t)\Gamma$ (\cite{KV1}5.12), $R\circ \tau_t=\tau_t\circ R$ (\cite{KV1}5.21), and $\Gamma\circ\sigma_t^\varphi=(\tau_t\otimes\sigma_t^\varphi)\Gamma$ (\cite{KV1} 5.38) (and, therefore, $\Gamma\circ\sigma_t^{\varphi\circ R}=(\sigma_t^{\varphi\circ R}\otimes\tau_{-t})\Gamma$ (\cite{KV1}5.17)).

The application $S=R\circ\tau_{-i/2}$ is called the \emph{antipode} of $\bfG$. 

The modular groups of the weights $\varphi$ and $\varphi\circ R$
commute, which leads to the definition of the \emph{scaling constant} $\lambda\in\R$ and the \emph{modulus}, which is a positive self-adjoint operator $\delta$ affiliated to $M$, such that $(D\varphi\circ R:D\varphi)_t=\lambda^{it^2/2}\delta^{it}$. 

We have $\varphi\circ\tau_t=\lambda^t\varphi$, and the canonical implementation of $\tau_t$ is given by a positive non-singular operator $P$ defined by $P^{it}\Lambda_\varphi(x)=\lambda^{t/2}\Lambda_\varphi(\tau{_{t}}(x))$. Moreover, the operator $\widehat{\Delta}$ is equal to the closure of $PJ\delta^{-1}J$, and the operator $\widehat{\delta}$ is equal to the closure of $P^{-1}J\delta J\delta^{-1}\Delta^{-1}$ (\cite{KV2}, 2.1 and \cite{Y3}, 2.5).

We have $\widehat{J}J=\lambda^{i/4}J\widehat{J}$ (\cite {KV2}
2.12). The operator $\widehat{P}$ is equal to $P$, the scaling
constant $\widehat{\lambda}$ is equal to $\lambda^{-1}$. Moreover, we have (\cite{Y3}, 3.4)
\[W(\widehat{\Delta}^{it}\otimes\widehat{\Delta}^{it})W^*=\delta^{it}\widehat{\Delta}^{it}\otimes\widehat{\Delta}^{it}.\]

A \emph{representation} of $\bfG$ on a Hilbert space $K$ is a unitary $U\in
M\otimes B(K)$, satisfying $(\Gamma\otimes \id)(U)=U_{23}U_{13}$. It
is well known that such a representation satisfies that, for any $\xi$, $\eta$ in $K$, the operator $(\id\otimes\omega_{\xi, \eta})(U)$ belongs to $\mathcal D(S)$ and that
\[S[(\id\otimes\omega_{\xi, \eta})(U)]=(\id\otimes\omega_{\xi,
  \eta})(U^*)\]
 (a proof for measured quantum groupoids can be found in \cite{E2},
 5.10).

Other locally compact quantum groups are $\bfG^{\op}=(M, \varsigma\circ \Gamma, \varphi\circ R, \varphi)$ (the \emph{opposite} locally compact quantum group) and $\bfG^{\com}=(M', (j\otimes j)\circ \Gamma\circ j,\varphi\circ j, \varphi\circ R\circ j)$ (the \emph{commutant} locally compact quantum group) where $j(x)=J_\varphi x^*J_\varphi$ is the canonical anti-$*$-isomorphism between $M$ and $M'$ given by Tomita-Takesaki theory. It is easy to get that $\widehat{\bfG^{\op}}=(\widehat{\bfG})^{\com}$ and $\widehat{\bfG^{\com}}=(\widehat{\bfG})^{\op}$ (\cite{KV2}4.2). We have $M\cap \widehat{M}=M'\cap\widehat{M}=M\cap\widehat{M}'=M'\cap\widehat{M}'=\mathds{C}$. The multiplicative unitary $W^{\op}$ of $\bfG^{\op}$ is equal to $(\widehat{J}\otimes\widehat{J})W(\widehat{J}\otimes\widehat{J})$, and the multiplicative unitary $W^{\com}$ of $\bfG^{\com}$ is equal to $(J\otimes J)W(J\otimes J)$. 

Moreover, the norm closure of the space $\{(\id\otimes\omega)(W) :
\omega\in B(H)_*\}$ is a $C^*$-algebra denoted
$\Cr(\bfG)$, which is invariant under $R$, and, together with
the restrictions of $\Gamma$, $\varphi$ and $\varphi\circ R$ will give
the \emph{reduced $C^*$-algebraic locally compact quantum group}
(\cite{KV1}, \cite{KV2}). In \cite{K} was defined also a
\emph{universal} version $\Cu(\bfG)$, which is equipped with
a coproduct $\Gamma_u$. There exists a canonical
 surjective $*$-homomorphism $\pi_{\bfG}$ from $\Cu(\bfG)$ to $\Cr(\bfG)$, such that
 $
(\pi_{\bfG}\otimes\pi_{\bfG})\Gamma_u=\Gamma\circ\pi_{\bfG}$. Then, $\varphi\circ\pi_{\bfG}$ (resp. $\varphi\circ R\circ \pi_{\bfG}$) is a (non-faithful) weight on $\Cu(\bfG)$ which is left-invariant (resp. right-invariant). 

If $G$ is a locally compact group equipped with a left Haar measure $ds$, then, by duality of the Banach algebra structure of $L^1(G, ds)$, it is possible to define a co-associative coproduct $\Gamma^a_G$ on $L^\infty(G, ds)$ and to give to $(L^\infty(G, ds), \Gamma^a_G, ds, ds^{-1})$ a structure of locally compact quantum group, called $G$ again; any locally compact quantum group whose underlying von Neumann algebra is abelian is of that type.  Then, its dual locally compact quantum group $\widehat{G}$ is $(\mathcal L (G), \Gamma^s_G, \varphi_G, \varphi_G))$, where $\mathcal L(G)$ is the von Neumann algebra generated by the left regular representation $\lambda_G$ of $G$ on $L^2(G, ds)$, $\Gamma^s_G$ is defined, for all $s\in G$, by $\Gamma^s_G(\lambda_G(s))=\lambda_G(s)\otimes\lambda_G(s)$, and $\varphi_G$ is defined, for any $f$ in the algebra $\mathcal K(G)$ of continuous functions with compact support, by $\varphi_G(\int_G f(s)\lambda_G(s)ds)=f(e)$, where $e$ is the neutral element of $G$. Any locally compact quantum group which is symmetric (i.e. such that $\varsigma\circ\Gamma=\Gamma$) is of that type. 

Let $(A, \Gamma)$ be a \emph{compact quantum group}, that is, $A$ is a unital
$C^*$-algebra  and $\Gamma$ is a coassociative coproduct from $A$ to
$A\otimes_{\min}A$ satisfying the cancellation property, i.e.,
$(A\otimes_{\min} 1)\Gamma(A)$ and $(1 \otimes_{\min} A)\Gamma(A)$ are
dense in $A\otimes_{\min}A$ (\cite{W4}). Then, there exists a left- and
right-invariant state $\omega$ on $A$, and we can always restrict to
the case when $\omega$ is faithful. Moreover, $\Gamma$ extends to a
normal $*$-homomorphism from $\pi_\omega(A)''$ to the (von Neumann)
tensor product $\pi_\omega(A)''\otimes\pi_\omega(A)''$, which we shall
still denote by $\Gamma$, for simplification, and $\omega$ can be
extended to a normal faithful state on $\pi_\omega(A)''$, we shall
still denote $\omega$ for simplification. Then, $(\pi_\omega(A)'',
\Gamma, \omega, \omega)$ is a locally compact quantum group, which we
shall call the von Neumann version of $(A, \Gamma)$. Its dual is
called a discrete quantum group.
\subsection{Left actions of a locally compact quantum group}
\label{action}
A \emph{left action} of a locally compact quantum group $\bfG$ on a von
Neumann algebra $N$ is an injective unital $*$-homomorphism $\ga$ from
$N$ into the von Neumann tensor product $M\otimes N$ such that
\[(\id\otimes\ga)\ga=(\Gamma\otimes \id)\ga,\]
where $\id$ means the identity on $M$ or on $N$ as well (\cite{V}, 1.1). 

We shall denote by $N^\ga$ the sub-algebra of $N$ such that $x\in N^\ga$ if and only if $\ga(x)=1\otimes x$ (\cite{V}.2). If $N^\ga=\mathbb{C}$, the action $\ga$ is called \emph{ergodic}. The formula $T_\ga=(\varphi\circ R\otimes \id)\ga$ defines a normal faithful operator-valued weight from $N$ onto $N^\ga$. We shall say that $\ga$ is \emph{integrable} if and only if this operator-valued weight is semi-finite (\cite{V} 1.3, 1.4).

To any left action is associated (\cite{V}2.1) a \emph{crossed product}
$\bfG\ltimes_\ga N=(\ga(N)\cup \widehat{M}\otimes\mathds{C})''$ on
which  $\widehat{\bfG}^{\op}$ acts canonically by a left  action
$\tilde{\ga}$, called the \emph{dual action} (\cite{V}2.2), as follows:
\[\tilde{\ga}(X)=(\widehat{W}^{\op*}\otimes 1)(1\otimes X)(\widehat{
  W}^{\op}\otimes 1) \quad \text{for all } X\in \bfG\ltimes_\ga N;\]
 in particular, for any $x\in N$ and $y\in \widehat{M}$,
 \begin{align*}
  \tilde{\ga}(\ga(x))&=1\otimes\ga(x), &
\tilde{\ga}(y\otimes 1)&=\widehat{\Gamma}^{\op}(y)\otimes 1.
 \end{align*}
Moreover, we have $(\bfG\ltimes_\ga N)^{\tilde{\ga}}=\ga(N)$ (\cite{V} 2.7).

The operator-valued weight $T_{\tilde{\ga}}=(\widehat{\varphi}\otimes
\id)\circ \tilde{\ga}$ is semi-finite (\cite{V}2.5), which allows, for
any normal faithful semi-finite weight $\nu$ on $N$, to define a
lifted or \emph{dual} normal faithful semi-finite weight $\tilde{\nu}$ on
$\bfG\ltimes_\ga N$ by $\tilde{\nu}=\nu\circ\ga^{-1}\circ
T_{\tilde{\ga}}$ (\cite{V}3.1).  The Hilbert space $H_{\tilde{\nu}}$ is
canonically isomorphic to (and will be identified with) the Hilbert
tensor product $H\otimes H_\nu$ (\cite{V}3.4 and 3.10), and this
isomorphism identifies, for $x\in\gN_\nu$ and
$y\in\gN_{\widehat{\Phi}}$, the vector
$\Lambda_{\tilde{\nu}}((y\otimes 1)\ga(x))$ with
$\Lambda_{\widehat{\Phi}}(y)\otimes\Lambda_\nu(x)$. Moreover, for any
$X\in\gN_{\tilde{\nu}}$, there exists a family of operators $X_i$
of the form $X_i=\Sigma_j(y_{i,j}\otimes 1)\ga(x_{i,j})$, such that $X_i$ is weakly converging to $X$ and $\Lambda_{\tilde{\nu}}(X_i)$ is converging to $\Lambda_{\tilde{\nu}}(X)$ (\cite{V}, 3.4 and 3.10). 

Then
\[U_\nu^\ga=J_{\tilde{\nu}}(\widehat{J}\otimes J_\nu)\]
is a unitary which belongs to $M\otimes B(H_\nu)$, satisfies
$(\Gamma\otimes \id)(U^\ga_\nu)=(U^\ga_\nu)_{23}(U^\ga_\nu)_{13}$ and
implements $\ga$ in the sense that $\ga(x)=U^\ga_\nu(1\otimes
x)(U^\ga_\nu)^*$ for all   $x\in N$ (\cite{V}, 3.6, 3.7 and 4.4). The
operator $U^\ga_\nu$ is called \emph{the canonical implementation} of $\ga$ on $H_\nu$. 
Moreover, we have, trivially, $(U^\ga_\nu)^*=(\widehat{J}\otimes
J_\nu)J_{\tilde{\nu}}=(\widehat{J}\otimes
J_\nu)U^\ga_\nu(\widehat{J}\otimes J_\nu)$, and we get that
\[J_{\tilde{\nu}}\Lambda_{\tilde{\nu}}((y\otimes 1)\ga(x))=U^\ga_\nu(\widehat{J}\Lambda_{\widehat{\Phi}}(y)\otimes J_\nu\Lambda_\nu(x)).\]
If we take another normal faithful semi-finite weight $\psi$ on $N$, there exists a unitary $u$ from $H_\nu$ onto $H_\psi$ which intertwines the standard representations $\pi_\nu$ and $\pi_\psi$, and we have then $U^\ga_\psi=(1\otimes u)U^\ga_\nu(1\otimes u^{*})$ (\cite{V}, 4.1). 

The application $(\varsigma\otimes \id)(\id\otimes\ga)$ is a left action of $\bfG$ on $B(H)\otimes N$. Moreover, in the proof of (\cite{V}, 4.4), we find that $(\sigma\otimes \id)U^{(\varsigma\otimes \id)(\id\otimes\ga)}_{Tr\otimes\nu}(\sigma\otimes \id)=1\otimes U^\ga_\nu$. 

A \emph{right action} of a locally compact quantum $\bfG$ on a von Neumann algebra $N$ is an injective unital $*$-homomorphism $\ga$ from $N$ into the von Neumann tensor product $N\otimes M$ such that
\[(\ga\otimes \id)\ga=(\id\otimes\Gamma)\ga.\]
Then, $\varsigma\ga$ is a left action of $\bfG^{\op}$ on $N$ (where $\varsigma$ is the flip from $N\otimes M$ onto $M\otimes N$). 

In (\cite{Y3}, 2.4) and (\cite{BV}, Appendix) is defined, for any normal faithful semi-finite weight $\nu$ on $N$ and $t\in\R$, the \emph{Radon-Nykodym derivative} $(D\nu\circ\ga: D\nu)_t=\Delta_{\tilde{\nu}}^{it}(\widehat{\Delta}^{-it}\otimes \Delta_\nu^{-it})$.   This unitary $D_t$ belongs to $M\otimes N$ and 
\[(\Gamma\otimes \id)(D_t)=(\id\otimes\ga)(D_t)(1\otimes D_t),\]
(\cite{BV} 10.3 or \cite{Y2} 3.4 and \cite{Y3} 3.7)
Moreover, it is straightorward to get
\[D_{t+s}=D_t(\tau_t\otimes\sigma_t^\nu)(D_s)=D_s(\tau_s\otimes\sigma_s^\nu)(D_t).\]

\subsection{Drinfel'd double of a locally compact quantum group}
\label{DG}
Let $\bfG=(M, \Gamma, \varphi, \varphi\circ R)$ be a locally compact quantum group, $\widehat{\bfG}=(\widehat{M}, \widehat{\Gamma}, \widehat{\varphi}, \widehat{\varphi}\circ \widehat{R})$ its dual, then it is possible to construct (\cite{N}, \cite{Y1}, \cite{BV}) another locally compact quantum group 
\[
\DG=(M\otimes\widehat{M}, \Gamma_D, \varphi\otimes\widehat{\varphi}\circ \widehat{R}, \varphi\otimes\widehat{\varphi}\circ \widehat{R}),\]
called the \emph{Drinfel'd double} of $\bfG$,
where $\Gamma_D$ is defined by
\[\Gamma_D(x\otimes y)=\Ad(1\otimes\sigma W\otimes
1)(\Gamma(x)\otimes\widehat{\Gamma}(y))\] for all $x\in M$,
$y\in\widehat{M}$.  Here and throughout this paper, given a unitary
$U$ on a Hilbert space $\gH$, we denote by $\Ad(U)$ the automorphism
of $B(\gH)$ defined as usual by $x\mapsto UxU^*$ for all $x\in
B(\gH)$.

The co-inverse $R_D$ of $\DG$ is given by 
\[R_D(x\otimes y)=\Ad(W^*)(R(x)\otimes\widehat{R}(y)).\]
This locally compact quantum group is always unimodular, which means that the left-invariant weight is also right-invariant. 
In the sense of (\cite{VV1}, 2.9), $\widehat{\bfG}$ and $\bfG$ are closed quantum subgroups of $\widehat{\DG}$, which means that the injection of $\widehat{M}$ (resp. $M$) into the underlying von Neumann algebra of its dual $\widehat{\DG}$ preserve the coproduct. (See \ref{defQTCI} for more details about this definition.)
\subsection{Yetter-Drinfel'd algebras}
\label{YD}
Let $\bfG=(M, \Gamma, \varphi, \varphi\circ R)$ be a locally compact
quantum group and $\widehat{\bfG}=(\widehat{M}, \widehat{\Gamma},
\widehat{\varphi}, \widehat{\varphi}\circ \widehat{R})$ its dual. A \emph{$\bfG$-Yetter-Drinfel'd algebra} (\cite{NV}) is a von Neumann algebra $N$ with a left action $\ga$ of $\bfG$ and a left action $\widehat{\ga}$ of $\widehat{\bfG}$ such that
\[(\id\otimes\ga)\widehat{\ga} (x)=\Ad (\sigma W\otimes 1)(\id\otimes
\widehat{\ga})\ga (x) \quad \text{for all }x\in N.\]

One should remark that if $(N, \ga, \widehat{\ga})$ is a $\bfG$-Yetter-Drinfel'd algebra, then $(N, \widehat{\ga}, \ga)$ is a $\widehat{\bfG}$-Yetter-Drinfel'd algebra. 

If $B$ is a  von Neumann sub-algebra of $N$ such that $\ga(B)\subset
M\otimes B$ and $\widehat{\ga} (B)\subset \widehat{M}\otimes B$, then,
it is clear that the restriction $\ga_{|B}$
(resp. $\widehat{\ga}_{|B}$) is a left action of $\bfG$
(resp. $\widehat{\bfG}$) on $B$, and that $(B, \ga_{|B},
\widehat{\ga}_{|B})$ is a Yetter-Drinfeld algebra, which we shall call a sub-$\bfG$-Yetter-Drinfel'd algebra of $(N, \ga, \widehat{\ga})$. 


\begin{subtheorem}[\cite{NV}, 3.2]
\label{thNV1}
 Let $\bfG=(M, \Gamma, \varphi, \varphi\circ R)$ be a locally
  compact quantum group, $\widehat{\bfG}=(\widehat{M},
  \widehat{\Gamma}, \widehat{\varphi}, \widehat{\varphi}\circ
  \widehat{R})$ its dual, $\DG$ its Drinfel'd double and $N$  a von Neumann algebra equipped with a left action $\ga$ of $\bfG$ and a left action $\widehat{\ga}$ of $\widehat{\bfG}$. Then the following conditions are equivalent:
  \begin{enumerate}
  \item  $(N, \ga, \widehat{\ga})$ is a $\bfG$-Yetter-Drinfel'd
    algebra;
  \item  $(\id\otimes\widehat{\ga})\ga$ is a left action of
    $\DG$ on $N$.
  \end{enumerate}
\end{subtheorem}


\begin{subtheorem}[(\cite{NV}, 3.2)]
\label{thNV2}
Let $\bfG=(M, \Gamma, \varphi, \varphi\circ
    R)$ be a locally compact quantum group,
    $\widehat{\bfG}=(\widehat{M}, \widehat{\Gamma}, \widehat{\varphi},
    \widehat{\varphi}\circ \widehat{R})$ its dual, $\DG$ its
    Drinfeld's double and ${\ga_D}$ a left action of $\DG$
    on a von Neumann algebra $N$. Then there exist a left action $\ga$
    of $\bfG$ on $N$ and a left action $\widehat{\ga}$ of
    $\widehat{\bfG}$ on $N$ such that
    ${\ga_D}=(\id\otimes\widehat{\ga})\ga$.  These actions are
    determined by the conditions
  \begin{align*}
    (\id\otimes \id\otimes \ga){\ga_D}&=\Ad (1\otimes\sigma
    W\otimes 1)(\Gamma\otimes \id\otimes \id){\ga_D}, \\
    (\id\otimes \id\otimes\widehat{\ga}){\ga_D}&=(\id\otimes\widehat{\Gamma}\otimes \id){\ga_D},
\end{align*}
and $(N, \ga, \widehat{\ga})$ is a $\bfG$-Yetter-Drinfel'd algebra. 
\end{subtheorem}

\begin{subproposition}
  \label{propinv} With the notation of \ref{thNV2}, we have
    $N^{\ga_D}= N^{{\ga}}\cap N^{{\widehat{\ga}}}$. 
\end{subproposition}
\begin{proof}
As ${\ga_D}=(\id\otimes\widehat{\ga})\ga$, we get that $N^\ga\cap
N^{{\widehat{\ga}}}\subset N^{\ga_D}$. On the other hand, using the
formula $(\id\otimes
\id\otimes\widehat{\ga}){\ga_D}=(\id\otimes\widehat{\Gamma}\otimes
\id){\ga_D}$, we get that every $x\in N^{\ga_D}$ belongs to
$N^{\widehat{\ga}}$. Moreover, using the formula $(\id\otimes \id\otimes \ga){\ga_D}=\Ad (1\otimes\sigma
    W\otimes 1)(\Gamma\otimes \id\otimes \id){\ga_D}$, we then get that every $x \in N^{\ga_{D}}$ also belongs
to $N^{\ga}$. 
\end{proof}

\begin{subproposition}
  \label{propRN}  Let $\bfG=(M, \Gamma, \varphi, \varphi\circ R)$
    be a locally compact quantum group, $\widehat{\bfG}=(\widehat{M},
    \widehat{\Gamma}, \widehat{\varphi}, \widehat{\varphi}\circ
    \widehat{R})$ its dual, $(N, \ga, \widehat{\ga})$ a
    $\bfG$-Yetter-Drinfel'd algebra and $\nu$ a normal faithful
    semi-finite weight on $N$. Let $t\in\R$,
    $D_t=(D\nu\circ\ga:D\nu)_t$ and
    $\widehat{D}_t=(D\nu\circ\widehat{\ga}:D\nu)_t$. Then
  \[\Ad(\sigma W\otimes
  1)[(\id\otimes\widehat{\ga})(D_t)(1\otimes\widehat{D}_t)]=(\id\otimes\ga)(\widehat{D}_t)(1\otimes
  D_t),\] and if $\tilde \nu$ and $\widetilde{\hat \nu}$ denote the
  weights on $\bfG \ltimes_{\ga} N$ and $\widehat{\bfG}
  \ltimes_{\widehat{\ga}} N$, respectively, dual to $\nu$, then
  \begin{align*}
    \Ad(\sigma W\otimes
    1)[(id\otimes\widehat{\ga})(D_t)(\widehat\Delta^{it}\otimes
    \Delta_{\widetilde{\hat\nu}}^{it})]=(id\otimes\ga)(\widehat{D}_t)(\Delta^{it}\otimes
    \Delta_{\tilde{\nu}}^{it}).
  \end{align*}
\end{subproposition}
\begin{proof} As $(\tau_t\otimes\widehat{\tau}_t)(W)=W$ for all
  $t\in\R$, the first equation is a straightforward application of
  (\cite{BV}, 10.4). The second one follows easily using the relations
  \begin{align*}
    (\widehat\Delta^{it} \otimes
    \Delta_{\widetilde{\hat\nu}}^{it})(W^{*}\sigma \otimes 1) &=
(1 \otimes \widehat{D}_{t})    (\widehat\Delta^{it} \otimes
\Delta^{it} \otimes \Delta_{\nu}^{it})(W^{*}\sigma \otimes 1) \\
&=(1 \otimes \widehat{D}_{t})  (W^{*}\sigma \otimes 1)
(\Delta^{it}\otimes\widehat\Delta^{it} \otimes \Delta^{it}_{\nu})
  \end{align*}
  and $D_{t}(\widehat\Delta^{it} \otimes \Delta^{it}_{\nu}) =
  \Delta^{it}_{\tilde \nu}$.
\end{proof}

\subsubsection{\bf Basic example and De Commer's construction (\cite{DC})} 
\label{basic}
We can consider the coproduct $\Gamma_D$ of $\DG$ as a left
action of $\DG$ on $M\otimes\widehat{M}$. Using \ref{thNV1}, we
get that there exist a left action $\gb$ of $\bfG$ on
$M\otimes\widehat{M}$ and a left action $\widehat{\gb}$ of
$\widehat{\bfG}$ on $M\otimes\widehat{M}$ such that
$\Gamma_D=(\id\otimes\widehat{\gb})\gb$. We easily obtain that for all $X\in M\otimes\widehat{M}$,
\begin{align*}
  \gb(X)&=(\Gamma\otimes \id)(X), &
  \widehat{\gb}(X)&=\Ad(\sigma W\otimes
  1)[(\id\otimes\widehat{\Gamma})(X)].
\end{align*}
Therefore, $\gb$ and $\widehat{\gb}$ appear as the actions associated by (\cite{DC}, 6.5.2) to the closed quantum subgroups $\bfG$ and $\widehat{\bfG}$ of $\widehat{\DG}$.

De Commer's construction allows us to make a link between this basic example and any Yetter-Drinfel'd algebra; namely, if $(N, \ga, \widehat{\ga})$ is a Yetter-Drinfel'd algebra, let us define ${\ga_D}=(\id\otimes\widehat{\ga})\ga$ the left action of $\DG$ on $N$, and, given a normal, semi-finite faithful weight $\nu$ on $N$, let $U^{\ga_D}_\nu$, $U^\ga_\nu$, $U^{\widehat{\ga}}_\nu$ be the canonical implementation of ${\ga_D}$, $\ga$, $\widehat{\ga}$. In the sense of De Commer, $\ga$ and $\widehat{\ga}$ are ``restrictions'' (to $\bfG$ and $\widehat{\bfG}$) of ${\ga_D}$ and, using (\cite{DC}, 6.5.3 and 6.5.4), we get that
\begin{align*}
  (\gb\otimes
  \id)(U_\nu^{\ga_D})&=(U_\nu^\ga)_{14}(U_\nu^{\ga_D})_{234}, & (\widehat{\gb}\otimes
  \id)(U_\nu^{\ga_D})=(U_\nu^{\widehat{\ga}})_{14}(U_\nu^{\ga_D})_{234}.
\end{align*}
In particular, 
\begin{align*}
(U_\nu^{\ga_D})_{125}(U_\nu^{\ga_D})_{345}
&=
(\Gamma_D\otimes \id)(U_\nu^{\ga_D})\\
&= (\id\otimes\widehat{\gb}\otimes \id)(\gb\otimes \id)(U_\nu^{\ga_D})\\
&=
(\id\otimes\widehat{\gb}\otimes \id)[(U_\nu^\ga)_{14}(U_\nu^{\ga_D})_{234}]= 
(U_\nu^\ga)_{15}(U_\nu^{\widehat{\ga}})_{25}(U_\nu^{\ga_D})_{345},
\end{align*}
whence $U_\nu^{\ga_D}=(U^\ga_\nu)_{23}(U^{\widehat{\ga}}_{\nu})_{13}$. 
As this result depends on an unpublished part of \cite{DC}, we shall give a different proof of this formula  in \ref{corintegrable3}, using the techniques of invariant weights, and then give several technical corollaries of this fact which will be used throughout this paper.

\subsection{Braided-commutativity of Yetter-Drinfel'd algebras}
\label{BC}

\begin{subdefinition}
\label{BCdef1}
Let $\bfG$ be a locally compact quantum group and $\ga$ a left action
of $\bfG$ on a von Neumann algebra $N$. For any $x\in N$, let us define
\begin{align*}
  \ga^{\com}(x^{\op})&=(j\otimes .^{\op})\ga(x) = \Ad(J \otimes J_{\nu})[\ga(x)^{*}], \\
  \ga^{\op}(x^{\op})&=(R\otimes .^{\op})\ga(x) = \Ad(\widehat J \otimes J_{\nu})[\ga(x)^{*}].
\end{align*}
Then $\ga^{\com}$ is a left action of $\bfG^{\com}$ on $N^{\op}$, and $\ga^{\op}$
is a left action of $\bfG^{\op}$ on $N^{\op}$.  
\end{subdefinition}

Let $\nu$ be a normal semi-finite faithful weight on $N$ and $\nu^{\op}$
the normal semi-finite faithful weight on $N^{\op}$ defined by
$\nu^{\op}(x^{\op})=\nu(x)$ for any $x\in N^+$. Let
$D_t=(D\nu\circ\ga:D\nu)_t$, $D_t^{\op}=(D\nu^{\op}\circ\ga^{\op}:D\nu^{\op})_t$,
which belongs to $M\otimes N^{\op}$, and
$D_t^{\com}=D(\nu^{\op}\circ\ga^{\com}:D\nu^{\op})_t$, which belongs to $M'\otimes N^{\op}$.
Then for all $t\in\R$,
\begin{align*}
D_{-t}^{\op}&=\Ad(\widehat{J}\otimes J_\nu)[D_t],  & D_{-t}^{\com} &=\Ad(J\otimes J_\nu)[D_t].
\end{align*}
\begin{sublemma}
\label{lemBC}
 Let $\bfG$ be a locally compact quantum group, $\ga$ a
    left action of $\bfG$ on a von Neumann algebra $N$, $\nu$ a
    normal faithful semi-finite weight on $N$, and $U^{\ga}_{\nu}$ the
    standard implementation of $\ga$. Then:

    \begin{enumerate}
    \item  $({\bfG}\ltimes_{\ga} N)' = U^{\ga}_{\nu}({\bfG}^{\op}
      \ltimes_{\ga^{\op}} N^{\op})(U^{\ga}_{\nu})^{*}$;
    \item  $(U^{\ga}_{\nu})^{*}$ is the standard implementation of the
      left action $\ga^{\op}$ on $N^{\op}$ with respect to the
      opposite weight $\nu^{\op}$. In particular,
      $(U^{\ga}_{\nu})^{*}$ is a representation of ${\bfG}^{\op}$ and
      $\ga^{\op}(x^{\op})=(U^{\ga}_{\nu})^{*}(1 \otimes
      x^{\op})U^{\ga}_{\nu}$ for all $x\in N$.
    \item  $\Delta_{\tilde{\nu}}^{it}U^\ga_\nu=U^\ga_\nu
      D^{\op}_{-t}(\widehat{\Delta}^{it}\otimes\Delta_\nu^{it})$ and
      $\Ad(\widehat{\Delta}^{it}\otimes\Delta_\nu^{it})[(U^\ga_\nu)^*]=(D^{\op}_{-t})^*(U^\ga_\nu)^*D_t$
      for all $t\in\R$.
    \end{enumerate}
\end{sublemma}
\begin{proof}
  (i) The relation $U^{\ga}_{\nu}=J_{\tilde \nu}(\widehat J \otimes
  J_{\nu})$ and the definition of the crossed products imply
  \begin{align*}
    U^{\ga}_{\nu}({\bfG}^{\op} \ltimes_{\ga^{\op}}
    N^{\op})(U^{\ga}_{\nu})^{*} &=
    J_{\tilde \nu}(\widehat J \otimes J_{\nu})((\widehat J\widehat M\widehat J \otimes 1_{H_{\nu}}) \cup \ga^{\op}(N^{\op}))''(\widehat J \otimes J_{\nu}) J_{\tilde \nu} \\
    &= J_{\tilde \nu}({\bfG} \ltimes_{\ga} N)J_{\tilde \nu} \\ & = ({\bfG} \ltimes_{\ga} N)'.
  \end{align*}

(ii) Denote by $\mu$ the weight on ${\bfG}^{\op}
\ltimes_{\ga^{\op}} N^{\op}$ dual to $\nu^{\op}$. By \S3 in \cite{V}, there
exists a GNS-map $\Lambda_{\mu} \colon
\gN_{\mu} \to H \otimes H_{\nu}$ determined by
  \begin{align} \label{eq:lambda-not}
    \Lambda_{\mu} ( (\widehat J y\widehat J \otimes
    1_{H_{\nu}})\ga^{\op}(x^{\op})^{*}) = \widehat J\widehat \Lambda(y) \otimes
    J_{\nu}\Lambda_{\nu}(x)
  \end{align}
  for all $y\in \gN_{\widehat \phi}$ and $x \in \gN_{\nu}$, and
  the standard implementation $U^{\ga^{\op}}_{\nu^{\op}}$ of $\ga^{\op}$
  with respect to $\nu^{\op}$ is given by
  $U^{\ga^{\op}}_{\nu^{\op}}=J_{\mu}(\widehat J \otimes
  J_{\nu})$.

  On the other hand, the GNS-map $\Lambda_{\tilde \nu}$ for the dual
  weight $\tilde \nu$ yields a GNS-map $\Lambda_{\tilde \nu^{\op}}$ for
  the opposite $\tilde\nu^{\op}$ on the commutant
  $J_{\tilde\nu}(M\ltimes_{\ga} N)J_{\tilde\nu}$, determined by
  \begin{align} \label{eq:lambda-nto}
    \Lambda_{\tilde \nu^{\op}}(J_{\tilde \nu}(y \otimes 1)\ga(x)J_{\tilde \nu}) = J_{\tilde \nu}\Lambda_{\tilde \nu}((y \otimes 1)\ga(x))= J_{\tilde \nu}(\widehat\Lambda(y) \otimes \Lambda_{\nu}(x))
  \end{align}
for $y\in \gN_{\widehat \phi}$ and $x\in \gN_{\nu}$.

Comparing \eqref{eq:lambda-not} with \eqref{eq:lambda-nto} and using
the relation $U^{\ga}_{\nu}=J_{\tilde \nu}(\widehat J \otimes J_{\nu})$,
we can conclude that
  \begin{align*}
    \Lambda_{\mu}((U^{\ga}_{\nu})^{*}a U^{\ga}_{\nu})
     &= (U^{\ga}_{\nu})^{*} \Lambda_{\tilde \nu^{\op}}(a)
  \end{align*}
for all $a \in \gN_{\tilde\nu^{\op}}$. Consequently, $J_{\mu} = (U^{\ga}_{\nu})^{*}J_{\tilde{\nu}} U^{\ga}_{\nu}$ and
$U^{\ga^{\op}}_{\nu^{\op}}=J_{\mu}(\widehat J \otimes
  J_{\nu})  = (U^{\ga}_{\nu})^{*}$.
  
  (iii) Using \ref{action}, we have:
  \begin{align*}
  \Delta_{\tilde{\nu}}^{it}U^\ga_\nu(\widehat{\Delta}^{-it}\otimes\Delta_\nu^{-it})
  &=
  \Delta_{\tilde{\nu}}^{it}J_{\tilde{\nu}}(\widehat{J}\otimes J_\nu)(\widehat{\Delta}^{-it}\otimes\Delta_\nu^{-it})\\
  &=
 J_{\tilde{\nu}}\Delta_{\tilde{\nu}}^{it}(\widehat{J}\otimes J_\nu)(\widehat{\Delta}^{-it}\otimes\Delta_\nu^{-it})\\
&=
J_{\tilde{\nu}}D_t(\widehat{\Delta}^{it}\otimes\Delta_\nu^{it})(\widehat{J}\otimes J_\nu)(\widehat{\Delta}^{-it}\otimes\Delta_\nu^{-it})\\
&=
J_{\tilde{\nu}}(\widehat{J}\otimes J_\nu)D^{\op}_{-t}\\
&=
U^\ga_\nu D^{\op}_{-t}
\end{align*}
from which we get the first formula, and then the second one by taking the adjoints. 
\end{proof}

\begin{subdefinition}
\label{defBC}
Let  $\bfG$ be a locally compact quantum group and $(N, \ga,
\widehat{\ga})$ a $\bfG$-Yetter-Drinfel'd algebra. Since
$\Ad(\widehat{J}J)=\Ad(J\widehat{J})$, the following  two properties are equivalent:

(i) $\ga^{\com}(N^{\op})$ and $\widehat{\ga}^{\com}(N^{\op})$ commute;

(ii) $\ga^{\op}(N^{\op})$ and $\widehat{\ga}^{\op}(N^{\op})$ commute;

We shall say that $(N, \ga, \widehat{\ga})$ is \emph{braided-commutative} if these conditions are fulfilled. 
\end{subdefinition}
It is clear that any sub-$\bfG$-Yetter-Drinfel'd algebra of a braided-commutative $\bfG$-Yetter-Drinfel'd algebra is also braided-commutative.

\begin{subtheorem}[(\cite{Ti2})]
\label{ThBC}
 Let $\bfG$ be a locally compact quantum group, $(N, \ga, \widehat{\ga})$ a $\bfG$-Yetter-Drinfel'd algebra, $\nu$  a normal faithful semi-finite weight on $N$, and $U^\ga_\nu$ the standard implementation of $\ga$. Define an injective anti-$*$-homomorphism $\beta$ by 
\[\beta(x)=U^\ga_\nu
\widehat{\ga}^{\op}(x^{\op})(U^\ga_\nu)^*=\Ad(U^\ga_\nu(U^{\widehat{\ga}}_\nu)^*)[1\otimes J_\nu x^*J_\nu] \quad
\text{for all } x\in N.\]
Then:

(i) $\beta(N)$ commutes with $\ga(N)$.

(ii) $(N, \ga, \widehat{\ga})$ is braided-commutative if and only if
$\beta (N)\subset {\bfG}\ltimes_\ga N$.
\end{subtheorem}
\begin{proof}
(i) The two formulas for $\beta(x)$ coincide by Lemma \ref{lemBC} (ii), and clearly, 
$\beta(N)\subseteq U^{\ga}_{\nu}(\widehat M\otimes N^{\op})(U^{\ga}_{\nu})^{*}$ commutes with $\ga(N)=U^{\ga}_\nu(1 \otimes N)(U^{\ga}_{\nu})^{*}$.

(ii) Using Lemma \ref{lemBC} (i), we see that
$\beta(N)=U^{\ga}_{\nu}\widehat{\ga}^{\op}(N^{\op})(U^{\ga}_{\nu})^{*}$
lies in ${\bfG} \ltimes_{\ga} N$ if and only if it commutes with
$({\bfG} \ltimes_{\ga} N)'=U^{\ga}_{\nu}({\bfG}^{\op}
\ltimes_{\ga^{\op}} N^{\op})(U^{\ga}_{\nu})^{*}$, that is, if and only
if $\widehat{\ga}^{\op}(N^{\op})$ commutes with $\widehat J \widehat M
\widehat J\otimes 1_{H_{\nu}}$ and with $\ga^{\op}(N^{\op})$. But
since $\widehat{\ga}^{\op}(N^{\op}) \subseteq \widehat M \otimes
N^{\op}$, the first condition is always satisfied.  \end{proof}

\begin{subproposition}
\label{PropBC}
 Let $\bfG$ be a locally compact quantum group and $(N, \ga,
  \widehat{\ga})$ a braided-commutative $\bfG$-Yetter-Drinfel'd
  algebra. Then $N^\ga\subseteq Z(N)$ and $N^{\widehat{\ga}}\subseteq
  Z(N)$. 
\end{subproposition}
\begin{proof}
Using \ref{BCdef1}, we get that the algebra $1\otimes (N^\ga)^{\op}$ commutes with $\widehat{\ga}^{\op}(N^{\op})$, and, therefore, that $1\otimes N^\ga$ commutes with $\widehat{\ga} (N)$. As it commutes with $B(H)\otimes 1$, it will commute with $B(H)\otimes N$, by (\cite{V}, th. 2.6). This is the first result. Applying it to the braided-commutative $\widehat{\bfG}$-Yetter-Drinfel'd algebra $(N, \widehat{\ga}, \ga)$, we get the second result. \end{proof}

\section{Invariant weights on Yetter-Drinfel'd algebras}
\label{inv}

In this chapter, we recall the definition (\ref{definv}) and basic properties (\ref{prop1inv}, \ref{prop2inv}) of a normal semi-finite faithful weight on a von Neumann algebra $N$, relatively invariant with respect to a left action $\ga$ of a locally compact quantum group $\bfG$ on $N$. Then, we study the case of an invariant weight on a Yetter-Drinfel'd algebra $(N, \ga, \widehat{\ga})$ (\ref{thinv1}, \ref{defYDinv}), and we prove that if $N$ is properly infinite, there exists such a weight (\ref{corint}). 

\begin{definition}
\label{definv}
Let $\bfG$ be a locally compact quantum group and $\ga$ a left action
of $\bfG$ on a von Neumann algebra $N$. Let $k$ be a positive
invertible operator affiliated to $M$. A normal faithful semi-finite weight $\nu$ on $N$ is
said to be \emph{$k$-invariant} under $\ga$ if for all $x\in N^+$,
\[(\id\otimes\nu)\ga(x)=\nu(x)k.\] 
\end{definition}

Applying $\Gamma$ to this formula,
one gets $\Gamma(k)=k\otimes k$, whence $k^{it}$ is a (one-dimensional)
representation of $\bfG$ for all $t\in\R$. So, $k^{it}$ belongs to the von Neumann subalgebra $I(M)$ of $M$ generated by all unitaries $u$ of $M$ such that $\Gamma(u)=u\otimes u$. As $I(M)$ is globally invariant by $\tau_t$ and $R$, using (\cite{BV}, 10.5), we get that it is a locally compact quantum group, whose scaling group will be the restriction of $\tau_t$ to $I(M)$. Since this locally compact quantum group is cocommutative, we therefore get that the restriction of $\tau_t$ to $I(M)$ is trivial, from which we get that $\tau_t(k)=k$ for all $t\in\R$.

This property implies that $P$ and $k$ (resp.\ $\widehat{\Delta}$ and $k$) strongly commute. Therefore their product $kP$ (resp.\ $k\widehat{\Delta}$) is closable, and its closure will be denoted again $kP$ (resp.\ $k\widehat{\Delta}$).

It is proved in (\cite{Y3}, 4.1) that $\nu$ is $k$-invariant if and only if, for all $t\in\R$, we have $(D\nu\circ\ga:D\nu)_t=k^{-it}\otimes 1$ (or, equivalently, $\Delta_{\tilde{\nu}}^{it}=k^{-it}\widehat{\Delta}^{it}\otimes\Delta_\nu^{it}$). 

If $k=1$, we shall say that $\nu$ is \emph{invariant} under $\ga$. 

\begin{proposition}
\label{prop1inv}
 Let $\bfG$ be a locally compact quantum group, $\ga$ a left
  action of $\bfG$ on a von Neumann algebra $N$, and $\nu_1$ and
  $\nu_2$ two $k$-invariant normal faithful semi-finite weights on
  $N$. Then $(D\nu_1:D\nu_2)_t$ belongs to $N^\ga$  for all $t\in\R$. 
\end{proposition}
\begin{proof}
For $k=1$, this result had been proved in (\cite{E3}7.8) for right actions of measured quantum groupoids. To get it for left actions of locally compact quantum groups is just a translation. The generalization for any $k$ is left to the reader (see \cite{V}, 3.9). \end{proof}

\begin{proposition}
\label{prop2inv}
 Let $\bfG$ be a locally compact quantum group, $\ga$ a left
  action of $\bfG$ on a von Neumann algebra $N$, and $\nu$ a  $k$-invariant faithful
  normal semi-finite weight on $N$. Then:

(i)   $\ga(\sigma_t^\nu(x))=(\Ad
k^{-it}\circ\tau_t\otimes\sigma_t^\nu)\ga(x)$ for all $x\in N$ and $t\in\R$;

(ii) for all $x\in\gN_\nu$, $\xi\in \mathcal D(k^{-1/2})$ and $\eta\in H$, $(\omega_{k^{-1/2}\xi, \eta}\otimes \id)\ga(x)$ belongs to $\gN_\nu$, and the canonical implementation $U^{\ga}_\nu$ is given by
\[(\omega_{\xi, \eta}\otimes
\id)(U_\nu^\ga)\Lambda_\nu(x)=\Lambda_\nu[(\omega_{k^{-1/2}\xi,
  \eta}\otimes \id)\ga(x)].\]
\end{proposition}

\begin{proof}
  (i) Since
  $\Delta_{\tilde{\nu}}^{it}=k^{-it}\widehat{\Delta}^{it}\otimes\Delta_\nu^{it}$,
\[\ga(\sigma_t^{\nu}(x))=\sigma_t^{\tilde{\nu}}(\ga(x))=(k^{-it}\widehat{\Delta}^{it}\otimes\Delta_\nu^{it})\ga(x)(\widehat{\Delta}^{-it}k^{it}\otimes\Delta_\nu^{-it})\]
 for all $t\in\R$.

(ii)  The first result of (ii) is proved (for $k=\delta^{-1}$) in
(\cite{V}, 2.4), and the general case can be proved the same
way.
\end{proof}

\begin{theorem}
\label{thinv1}
 Let $\bfG$ be a locally compact quantum group, $(N, \ga, \widehat{\ga})$ a $\bfG$-Yetter-Drinfel'd algebra, ${\ga_D}=(\id\otimes\widehat{\ga})\ga$ the action of $\DG$ introduced in \ref{thNV1}, and $\nu$  a faithful normal semi-finite weight on $N$. Then the following conditions are equivalent:

(i) the weight $\nu$ is invariant under $\ga$ and invariant under $\widehat{\ga}$. 

(ii) the weight $\nu$ is invariant under ${\ga_D}$. 
\end{theorem}
\begin{proof}
The fact that (i) implies (ii) is trivial.  Suppose that (ii) holds.  Choose a state  $\omega$ in  $\widehat{M}_*$ and define $\nu'=(\omega\otimes\nu)\widehat{\ga}$. As  $(\id\otimes \id\otimes\nu){\ga_D}=\nu$, we get that $(\id\otimes\nu')\ga=\nu$. 

But 
\begin{align*}
(\id\otimes \id\otimes \nu'){\ga_D}
&=
(\id\otimes \id\otimes(\omega\otimes\nu)\widehat{\ga})(\id\otimes\widehat{\ga})\ga\\
&=
(\id\otimes \id\otimes\omega\otimes\nu)(\id\otimes\widehat{\Gamma}\otimes \id)(\id\otimes\widehat{\ga})\ga,
\end{align*}
and, for any  state $\omega'$ in $\widehat{M}_*$, 
\begin{align*}
(\id\otimes\omega'\otimes\nu'){\ga_D}
=
(\id\otimes(\omega'\otimes\omega)\circ \widehat{\Gamma}\otimes\nu){\ga_D}
=
\nu.
\end{align*}
 Therefore, by linearity, we get that $(\id\otimes \id\otimes\nu'){\ga_D}=\nu$. 
On the other hand, 
\begin{align*}
(\id\otimes \id\otimes \nu'){\ga_D}
&=
\Ad(W^*\sigma )(\id\otimes \id\otimes \nu')(\id\otimes\ga)\widehat{\ga}\\
&=
\Ad(W^*\sigma )(\id\otimes(\id\otimes\nu')\ga)\widehat{\ga}\\
&=
\Ad(W^*\sigma )(\id\otimes\nu)\widehat{\ga}
\end{align*}
But, as $(\id\otimes \id\otimes\nu'){\ga_D}=\nu$, we get that $\nu=(\id\otimes\nu)\widehat{\ga}$, and, therefore, $\nu$ is invariant under $\widehat{\ga}$. So, we get that $\nu'=\nu$, and $\nu$ is invariant under $\ga$. \end{proof}

\subsection{Definition}
\label{defYDinv}
Let $\bfG$ be a locally compact quantum group and $(N, \ga, \widehat{\ga})$ a $\bfG$-Yetter-Drinfel'd algebra. A normal faithful semi-finite weight on $N$ will be called \emph{Yetter-Drinfel'd invariant} if it satisfies one of the equivalent conditions of \ref{thinv1}.

\begin{theorem}
  \label{Thintegrable} \it Let $\bfG$ be a locally compact quantum
    group, $(N, \ga, \widehat{\ga})$ a $\bfG$-Yetter-Drinfel'd algebra
    and ${\ga_D}=(\id\otimes\widehat{\ga})\ga$ the action of
    $\DG$ introduced in \ref{thNV1}. If ${\ga_D}$ is
    integrable, then there exists a Yetter-Drinfel'd invariant normal
    faithful semi-finite weight on $N$.
\end{theorem}
\begin{proof} Clear by \cite{V}, 2.5, using the fact that the locally compact quantum group $\DG$ is unimodular. \end{proof}

\begin{corollary}
  \label{corintegrable2}
  \label{corintegrable} Let $\bfG=(M, \Gamma, \varphi, \psi)$ be
    a locally compact quantum group and $(N, \ga, \widehat{\ga})$ a
    $\bfG$-Yetter-Drinfel'd algebra. Denote by $H$ the Hilbert space
    $L^2(M)=L^2(\widehat{M})$. Then $(B(H)\otimes N, (\varsigma\otimes
    \id)(\id\otimes\alpha), (\varsigma\otimes
    \id)(\id\otimes\widehat{\alpha}))$ is a $\bfG$-Yetter-Drinfel'd
    algebra which has a normal semi-finite faithful Yetter-Drinfel'd
    invariant weight. 
\end{corollary}
\begin{proof} Let ${\ga_D}=(\id\otimes\widehat{\ga})\ga$
be the action of $\DG$ introduced in \ref{thNV1}. Using (\cite{V}, 2.6), we know that the action $(\varsigma\otimes \id)(\id\otimes{\ga_D})$ is a left action of $\DG$ which is cocycle-equivalent to the bidual action of ${\ga_D}$. As this bidual action is integrable (\cite{V}, 2.5), it has a Yetter-Drinfel'd invariant semi-finite faithful weight by \ref{Thintegrable}. Using (\cite{V}, 2.6.3), one gets that this weight is invariant as well under $(\varsigma\otimes \id)(\id\otimes{\ga_D})$. \end{proof}

\begin{corollary}
\label{corintegrable3}
 Let $\bfG$ be a locally compact quantum group, $(N, \ga, \widehat{\ga})$ a $\bfG$-Yetter-Drinfel'd algebra, $\nu$  a normal semi-finite faithful weight on $N$, $U^\ga_\nu$ and $U^{\widehat{\ga}}_\nu$ the canonical implementations of the actions $\ga$ and $\widehat{\ga}$, and $\beta$ the anti-$*$-homomorphism introduced in \ref{ThBC}. Then: 
 \begin{enumerate}
 \item the unitary implementations of the actions $\ga$,
   $\widehat{\ga}$ and $\ga_{D}$ are linked by the relation
   \begin{align*}
     U^{\ga_D}_\nu=(U^\ga_\nu)_{23}(U^{\widehat{\ga}}_\nu)_{13};
   \end{align*}
 \item 
   $(U^\ga_\nu)_{13}(U^{\widehat{\ga}}_\nu)_{23}=W_{12}(U^{\widehat{\ga}}_\nu)_{23}(U^\ga_\nu)_{13}W^*_{12}$;
 \item  $\Ad(1\otimes U^{\widehat{\ga}}_\nu(U^{\ga}_\nu)^*)[W\otimes
   1]=(U_\nu^\ga)^*_{13}W_{12}=(U^{\ga}_{\nu})_{23}^{*}W_{12}^{*}(U^{\ga}_{\nu})_{23}$;
 \item  writing $\beta^{\dag}$ for the map $x^{\op} \mapsto \beta(x)$,
   we have
   \[ \Ad(W\otimes 1)[1 \otimes \beta(x)] =(\id \otimes
   \beta^{\dag})(\ga^{\op}(x^{\op})) \quad \text{for all } x\in N.
   \]
 \end{enumerate}
\end{corollary}
\begin{proof}
(i) Suppose first that there is a faithful semi-finite Yetter-Drinfel'd invariant weight $\nu'$ for $(N, \ga, \widehat{\ga})$. Then, for $\xi_1$, $\xi_2$, $\eta_1$, $\eta_2$ in $H$, $x\in\gN_\nu$, we get, using \ref{prop2inv},
\begin{align*}
(\omega_{\xi_1\otimes\xi_2, \eta_1\otimes\eta_2}\otimes \id)(U^{\ga_D}_{\nu'})\Lambda_{\nu'}(x)
&=
\Lambda_{\nu'}[(\omega_{\xi_1\otimes\xi_2, \eta_1\otimes\eta_2}\otimes \id){\ga_D}(x)]\\
&=
\Lambda_{\nu'}[(\omega_{\xi_2, \eta_2}\otimes \id)\widehat{\ga}(\omega_{\xi_1, \eta_1}\otimes \id)\ga(x)]\\
&=(\omega_{\xi_2, \eta_2}\otimes \id)(U^{\widehat{\ga}}_{\nu'})(\omega_{\xi_1, \eta_1}\otimes \id)(U^\ga_{\nu'})\Lambda_{\nu'}(x)m,
\end{align*}
from which we get (i) for such a weight $\nu'$. Applying that result to \ref{corintegrable2}, we get that there exists a normal semi-finite faithful weight $\psi$ on $B(H)\otimes N$ such that
\[U^{(\varsigma\otimes \id)(\id\otimes{\ga_D})}_\psi=(U^{(\varsigma\otimes \id)(\id\otimes\widehat{\ga})}_\psi)_{234}(U^{(\varsigma\otimes \id)(\id\otimes\ga)}_\psi)_{134}.\]
 Using now (\cite{V}4.1), we get that for every normal semi-finite faithful weight on $N$, 
\[U^{(\varsigma\otimes \id)(\id\otimes{\ga_D})}_{Tr\otimes\nu}=(U^{(\varsigma\otimes \id)(\id\otimes\widehat{\ga})}_{Tr\otimes\nu})_{234}(U^{(\varsigma\otimes \id)(\id\otimes\ga)}_{Tr\otimes\nu})_{134}\]
which by (\cite{V}4.4) implies (i). 

(ii) From (i) we get that
$(U^\ga_\nu)_{23}(U^{\widehat{\ga}}_\nu)_{13}$ is a representation of
$\DG$. Therefore,
 \begin{align*}
 (U^{\widehat{\ga}}_\nu)_{45}(U^\ga_\nu)_{35}(U^{\widehat{\ga}}_\nu)_{25}(U^\ga_\nu)_{15}
 &=
 \Ad(1\otimes\sigma W\otimes 1\otimes 1)[(\Gamma\otimes\widehat{\Gamma}\otimes \id)((U^{\widehat{\ga}}_\nu)_{23}(U^\ga_\nu)_{13})]\\
 &=
  \Ad(1\otimes\sigma W\otimes 1\otimes 1)[(U^{\widehat{\ga}}_\nu)_{45}(U^{\widehat{\ga}}_\nu)_{35}(U^\ga_\nu)_{25}(U^\ga_\nu)_{15}]\\
  &=
(U^{\widehat{\ga}}_\nu)_{45} \Ad(1\otimes\sigma W\otimes 1\otimes 1)[(U^{\widehat{\ga}}_\nu)_{35}(U^\ga_\nu)_{25}](U^\ga_\nu)_{15},
\end{align*}
from which we infer that
\[(U^\ga_\nu)_{35}(U^{\widehat{\ga}}_\nu)_{25}=\Ad(1\otimes\sigma W\otimes 1\otimes 1)[(U^{\widehat{\ga}}_\nu)_{35}(U^\ga_\nu)_{25}].\]
 After renumbering the legs,  we obtain (ii). 

(iii)
The relation $W_{12}^*(U^\ga_\nu)_{23}^*W_{12}=(\Gamma\otimes
\id)(U^\ga_\nu)^*=(U^\ga_\nu)_{13}^*(U^\ga_\nu)^*_{23}$ implies
\[(U^\ga_\nu)^*_{23}W_{12}(U^\ga_\nu)_{23}=W_{12}(U^\ga_\nu)^*_{13}.\]
Using (ii), we get
\[(U^\ga_\nu)_{13}(U^{\widehat{\ga}}_\nu)_{23}=W_{12}(U^{\widehat{\ga}}_\nu)_{23}(U^\ga_\nu)^*_{23}W_{12}^*(U^\ga_\nu)_{23}\]
and, therefore, 
\[W_{12}^*(U^\ga_\nu)_{13}=(U^{\widehat{\ga}}_\nu)_{23}(U^\ga_\nu)^*_{23}W_{12}^*(U^\ga_\nu)_{23}(U^{\widehat{\ga}}_\nu)_{23}^*\]
which implies (iii). 

(iv) Relation (iii) and \ref{lemBC} imply
\begin{align*}
\Ad(W_{12})[\beta(x)_{23}]
&=\Ad(W_{12}(U^{\ga}_{\nu})_{23}(U^{\widehat\ga}_{\nu})_{23})[1 \otimes
1 \otimes x^{\op}] \\
&=\Ad((U^{\ga}_{\nu})_{23}(U^{\widehat\ga}_{\nu})^{*}_{23}(U^{\ga}_{\nu})_{13}^{*}W_{12})[1\otimes
1\otimes x^{\op}] \\
&=
\Ad((U^{\ga}_{\nu})_{23}(U^{\widehat\ga}_{\nu})^{*}_{23})[\ga^{\op}(x^{\op})_{13}]
\\
&=(\id \otimes \beta^{\dag})(\ga^{\op}(x^{\op})). \qedhere
\end{align*}
 \end{proof}

 \begin{subremark}
   We have quickly shown in \ref{basic} that (i) can also be deduced
   from a particular case of (\cite{DC} 6,5), which remains
   unpublished.
 \end{subremark}
 \begin{lemma}
\label{leminf}
 Let $N$ be a properly infinite von Neumann algebra.
 \begin{enumerate}
 \item  Let $(e_n)_{n\in\N}$ be a sequence of pairwise orthogonal
   projections in $N$, equivalent to $1$ and whose sum is $1$, and let
   $(v_n)_{n\in\N}$ be a sequence of isometries in $N$ such that
   $v_n^*v_n=1$ and $v_nv_n^*=e_n$ for all $n\in\N$, (and, therefore
   $v_i^*v_j=0$ if $i\not= j$).  Let $H$ be a separable Hilbert space
   and $u_{i,j}$ a set of matrix units of $B(H)$ acting on an
   orthonormal basis $(\xi_{i})_{i}$. For any $x\in N$, let
   \[\Phi(x)=\sum_{i,j}u_{i,j}\otimes v^*_ixv_j\]
   Then $\Phi$ is an isomorphism of $N$ onto $B(H)\otimes N$, and
   $\Phi^{-1}(1\otimes x)=\sum_i v_ixv_i^*$.
 \item  Let $\ga$ be a left action of a locally compact quantum group
   $\bfG=(M, \Gamma, \varphi, \psi)$ with separable predual $M_*$ on
   $N$. Then the operator $V=\sum_n(1\otimes v_n)\ga(v_n^*)$ exists,
   is a unitary in $M\otimes N$ and a cocycle for $\ga$, that is,
   $(\Gamma\otimes \id)(V)=(1\otimes V)(\id\otimes \ga)(V)$.
   Moreover, the actions $(\varsigma\otimes \id)(\id\otimes \ga)$ and
   $(\id\otimes\Phi)\ga\Phi^{-1}$ are linked by the relation \[(\varsigma\otimes
   \id)(\id\otimes \ga)(X)=
   \Ad((\id\otimes\Phi)(V))[(\id\otimes\Phi)\ga\Phi^{-1}(X)].
   \]
 \item  Let $\phi$ be a normal semi-finite faithful weight on $N$.
   Then for each $n\in\N$, the weight $\phi_n$ on $N$ defined by
   $\phi_n(x)=\phi(v_n xv_n^*)$ for all $x\in N^+$ is faithful, normal
   and semi-finite, and
   $\phi\circ\Phi^{-1}=\sum_n(\omega_{\xi_n}\otimes\phi_n)$.
 \item  Let $\psi$ be a normal semi-finite faithful weight on
   $B(H)\otimes N$. Then, with the notations of (iii)
   $(\psi\circ\Phi)_n(x)=\psi(u_{n,n}\otimes x)$ for all $x\in
   N^+$. If $\psi$ is invariant under $(\varsigma\otimes
   \id)(\id\otimes\ga)$, then each $(\psi\circ\Phi)_n$ is a normal
   semi-finite faithful weight on $N$, invariant under $\ga$.
 \end{enumerate}
\end{lemma}
\begin{proof}
(i) This result is taken from (\cite{T1}Th 4.6).

(ii) This assertion is proved in (\cite{E1}Th. IV.3) for right actions of Kac algebras, but remains true for left actions of any locally compact quantum group. 

(iii) Let $(\xi_i)_{i\in\N}$ be the orthonormal basis of $H$ defined
by the matrix units $u_{i,j}$. Then we can define an isometry $I$ from
$L^2(N)$ into $H\otimes L^2(N)$ by $I\eta=\sum_n \xi_{n} \otimes
v_n^*\eta$ for all $\eta\in L^2(N)$. It is then straightforward to get
that, for all sequences $(\eta_n)_{n\in \N}$ such that
$\sum_n\|\eta_n\|^2<\infty$, we have $I^*(\sum_n \xi_{n} \otimes
\eta_n)=\sum_nv_n\eta_n$. Therefore, $I$ is unitary and
$\Phi(x)=IxI^*$ and for all $x\in N$. So, for any $\zeta\in L^2(N)$,
$\omega_\zeta\circ \Phi^{-1}$ is equal to the normal weight
$\sum_n\omega_{\xi_n}\otimes\omega_{v^{*}_{n}\zeta}$. Hence, $\phi\circ\Phi^{-1}$ is the weight
$\sum_n \omega_{\xi_n}\otimes \phi_n$.

Let now $x\in N$ such that $\phi_n(x^*x)=0$. By definition, we get
that $xv_n^*=0$ and therefore $x=0$. So, the weight $\phi_n$ is
faithful. As $\phi$ is semi-finite, there exists in $\gM_\phi^+$ an
increasing family $x_k\uparrow 1$. For all $n\in\N$, we get $y_k=(\omega_{\xi_n}\otimes \id)\Phi(x_k)\uparrow 1$ and   $\phi_n(y_k)=(\omega_{\xi_n}\otimes\phi_n)\Phi(x_k)\leq \phi(x_k) <\infty$, which gives that $\phi_n$ is semi-finite. 

(iv) First,
\[(\psi\circ\Phi)_n(x)=(\psi\circ\Phi)(v_n^*xv_n)=\psi(\sum_{i,j}u_{i,j}\otimes v_i^*v_nxv_n^*v_j)=\psi(u_{n,n}\otimes x).\]
 If $\psi$ is invariant under $(\varsigma\otimes \id)(\id\otimes\ga)$, then it is clear that all $(\psi\circ\Phi)_n$ are normal semi-finite faithful weights on $N$, invariant under $\ga$.
\end{proof}

\begin{corollary}
  \label{corint}
  Let $\bfG=(M, \Gamma, \varphi, \psi)$ be a locally compact quantum
  group such that the predual $M_*$ is separable, and $(N, \ga,
  \widehat{\ga})$ a $\bfG$-Yetter-Drinfel'd algebra, where $N$ is a
  properly infinite von Neumann algebra. Then this
  $\bfG$-Yetter-Drinfel'd algebra has a normal faithful semi-finite
  invariant weight.
\end{corollary}
\begin{proof}
Use the left action ${\ga_D}=(\id\otimes\widehat{\ga})\ga$ of $\DG$ on $N$ and apply \ref{corintegrable} and \ref{leminf}(iv). \end{proof}


\section{The Hopf bimodule associated to a braided-commutative Yetter-Drinfel'd algebra}
\label{H}
In this chapter, we recall the definition of the relative tensor
product of Hilbert spaces, and of the fiber product of von Neumann
algebra (\ref{spatial}). Then, we recall the definition of a Hopf
bimodule (\ref{defHopf}) and a co-inverse. Starting then from a
braided-commutative Yetter-Drinfel'd algebra $(N, \ga,
\widehat{\ga})$, and any normal semi-finite faithful weight $\nu$ on
$N$, we first construct an isomorphism of the Hilbert spaces $H\otimes
H\otimes H_\nu$ and $(H\otimes
H_\nu)\underset{\nu}{_\beta\otimes_\ga}(H\otimes H_\nu)$ (\ref{propV})
and then show that the dual action $\tilde{\ga}$ of
$\widehat{\bfG}^{\op}$ on the crossed product $\bfG\ltimes_\ga N$,
modulo this isomorphism, can be interpreted as a coproduct on
$\bfG\ltimes_\ga N$ (\ref{th1}). Finally, we construct an involutive
anti-$*$-automorphism of $\bfG\ltimes_\ga N$ which turns out to be a
co-inverse (\ref{th2}).

\subsection{Relative tensor products of Hilbert spaces and fiber products of von Neumann algebras (\cite{C}, \cite{S}, \cite{T2}, \cite{EVal})}
\label{spatial}
 Let $N$ be a von Neumann algebra, $\psi$ a normal semi-finite
 faithful weight on $N$; we shall denote by $H_\psi$, $\gN_\psi$,
 \ldots the canonical objects of the Tomita-Takesaki theory associated
 to the weight $\psi$. 

Let $\alpha$ be a non-degenerate faithful representation of $N$ on a Hilbert space $\mathcal H$. The set of $\psi$-bounded elements of the left module $_\alpha\mathcal H$ is
\[D(_\alpha\mathcal{H}, \psi)= \lbrace \xi \in \mathcal{H} : \exists C < \infty ,\| \alpha (y) \xi\|
\leq C \| \Lambda_{\psi}(y)\|,\forall y\in \gN_{\psi}\rbrace.\]
 For any $\xi$ in $D(_\alpha\mathcal{H}, \psi)$, there exists a bounded operator
$R^{\alpha,\psi}(\xi)$ from $H_\psi$ to $\mathcal{H}$ such that
\[R^{\alpha,\psi}(\xi)\Lambda_\psi (y) = \alpha (y)\xi \quad\text{for
  all } y \in \gN_{\psi},\]
and this operator the actions of $N$. 
If $\xi$ and $\eta$ are bounded vectors, we define the operator product 
\[\langle \xi|\eta\rangle _{\alpha,\psi} = R^{\alpha,\psi}(\eta)^* R^{\alpha,\psi}(\xi),\]
which belongs to $\pi_{\psi}(N)'$. This last algebra will be
identified with the opposite von Neumann algebra $N^{\op}$ using Tomita-Takesaki theory. 

If now $\beta$ is a non-degenerate faithful anti-representation of $N$ on a Hilbert space $\mathcal K$, the relative tensor product $\mathcal K\underset{\psi}{_\beta\otimes_\alpha}\mathcal H$ is the completion of the algebraic tensor product $K\odot D(_\alpha\mathcal{H}, \psi)$ by the scalar product defined by 
\[(\xi_1\odot\eta_1 |\xi_2\odot\eta_2 )= (\beta(\langle \eta_1|
\eta_2\rangle _{\alpha,\psi})\xi_1 |\xi_2)\]
for all $\xi_1, \xi_2\in \mathcal{K}$ and $\eta_1, \eta_2 \in D(_\alpha\mathcal{H},\psi)$.
If $\xi\in \mathcal{K}$ and $\eta\in D(_\alpha\mathcal{H},\psi)$, we
 denote by $\xi\underset{\psi}{_\beta\otimes_\alpha}\eta$ the image of $\xi\odot\eta$ into $\mathcal K\underset{\psi}{_\beta\otimes_\alpha}\mathcal H$. Writing $\rho^{\beta, \alpha}_\eta(\xi)=\xi\underset{\psi}{_\beta\otimes_\alpha}\eta$, we get a bounded linear operator from $\mathcal H$ into $\mathcal K\underset{\nu}{_\beta\otimes_\alpha}\mathcal H$, which is equal to $1_\mathcal K\otimes_\psi R^{\alpha, \psi}(\eta)$. 

 Changing the weight $\psi$ will give an isomorphic Hilbert space, but
 the isomorphism will not exchange elementary tensors!

We shall denote by $\sigma_\psi$ the relative flip, which is a unitary sending $\mathcal{K}\underset{\psi}{_\beta\otimes_\alpha}\mathcal{H}$ onto $\mathcal{H}\underset{\psi^{\op}}{_\alpha\otimes _\beta}\mathcal{K}$, defined by
\[\sigma_\psi
(\xi\underset{\psi}{_\beta\otimes_\alpha}\eta)=\eta\underset{\psi^{\op}}{_\alpha\otimes_\beta}\xi\]
 for all $\xi \in D(\mathcal {K}_\beta ,\psi^{\op} )$ and $\eta \in
 D(_\alpha \mathcal {H},\psi)$.

If $x\in \beta(N)'$ and $y\in \alpha(N)'$, it is possible to define an
operator $x\underset{\psi}{_\beta\otimes_\alpha}y$ on $\mathcal
K\underset{\psi}{_\beta\otimes_\alpha}\mathcal H$, with natural values
on the elementary tensors. As this operator does not depend upon the
weight $\psi$, it will be denoted by $x\underset{N}{_\beta\otimes_\alpha}y$. 

If $P$ is a von Neumann algebra on $\mathcal H$ with
$\alpha(N)\subset P$, and $Q$ a von Neumann algebra on $\mathcal K$
with $\beta(N)\subset Q$, then we define the fiber product
$Q\underset{N}{_\beta*_\alpha}P$ as
$\{x\underset{N}{_\beta\otimes_\alpha}y : x\in Q', y\in P'\}'$. 
 This von Neumann algebra can be defined independently of the Hilbert
 spaces on which $P$ and $Q$ are represented. If for $i=1,2$,
 $\alpha_i$ is a faithful non-degenerate homomorphism from $N$ into
 $P_i$, and $\beta_i$ is a faithful non-degenerate anti-homomorphism from $N$ into $Q_i$, and $\Phi$ (resp. $\Psi$) a homomorphism from $P_1$ to $P_2$ (resp. from $Q_1$ to $Q_2$) such that $\Phi\circ\alpha_1=\alpha_2$ (resp. $\Psi\circ\beta_1=\beta_2$), then, it is possible to define a homomorphism $\Psi\underset{N}{_{\beta_1}*_{\alpha_1}}\Phi$ from $Q_1\underset{N}{_{\beta_1}*_{\alpha_1}}P_1$ into $Q_2\underset{N}{_{\beta_2}*_{\alpha_2}}P_2$. 

We define a relative flip $\varsigma_N$ from $\mathcal L(\mathcal K)\underset{N}{_\beta*_\alpha}\mathcal L(\mathcal H)$ onto $\mathcal L(\mathcal H)\underset{N^{\op}}{_\alpha*_\beta}\mathcal L(\mathcal K)$ by $\varsigma_N(X)=\sigma_\psi X(\sigma_{\psi})^*$ for any $X\in \mathcal L(\mathcal K)\underset{N}{_\beta*_\alpha}\mathcal L(\mathcal H)$ and any normal semi-finite faithful weight $\psi$ on $N$.

Let now $U$ be an isometry from a Hilbert space $\mathcal K_1$ in a Hilbert space $\mathcal K_2$, which intertwines two anti-representations $\beta_1$ and $\beta_2$ of $N$, and let $V$ be an isometry from a Hilbert space $\mathcal H_1$ in a Hilbert space $\mathcal H_2$, which intertwines two representations $\alpha_1$ and $\alpha_2$ of $N$. Then, it is possible to define, on linear combinations of elementary tensors, an isometry $U\underset{\psi}{_{\beta_1}\otimes_
{\alpha_1}}V$ which can be extended to the whole Hilbert space
$\mathcal K_1\underset{\psi}{_{\beta_1}\otimes_{\alpha_1}}\mathcal
H_1$ with values in $\mathcal
K_2\underset{\psi}{_{\beta_2}\otimes_{\alpha_2}}\mathcal H_2$. One can
show that this isometry does not depend upon the weight $\psi$. It
will be denoted by $U\underset{N}{_{\beta_1}\otimes_{\alpha_1}}V$. If $U$ and $V$ are unitaries, then $U\underset{N}{_{\beta_1}\otimes_{\alpha_1}}V$ is an unitary and $(U\underset{N}{_{\beta_1}\otimes_{\alpha_1}}V)^*=U^*\underset{N}{_{\beta_2}\otimes_{\alpha_2}}V^*$. 

In (\cite{DC}, chap. 11), De Commer had shown that, if $N$ is finite-dimensional, the Hilbert space $\mathcal K\underset{\nu}{_\beta\otimes_\alpha}\mathcal H$ can be isometrically imbedded into the usual Hilbert tensor product $\mathcal K\otimes\mathcal H$. 

\subsection{Definitions}
\label{defHopf}
 A quintuple $(N, M, \alpha, \beta, \Gamma)$ will be called a \emph{Hopf bimodule}, following (\cite{Val2}, \cite{EVal} 6.5), if
$N$,
$M$ are von Neumann algebras, $\alpha$ is a faithful non-degenerate
representation of $N$ into $M$, $\beta$ is a
faithful non-degenerate anti-representation of
$N$ into $M$, with commuting ranges, and $\Gamma$ is an injective $*$-homomorphism from $M$
into
$M\underset{N}{_\beta *_\alpha}M$ such that, for all $X$ in $N$,
\begin{enumerate}
\item $\Gamma
  (\beta(X))=1\underset{N}{_\beta\otimes_\alpha}\beta(X)$,
\item  $\Gamma
  (\alpha(X))=\alpha(X)\underset{N}{_\beta\otimes_\alpha}1$,
\item  $\Gamma$ satisfies the co-associativity relation
  \[(\Gamma \underset{N}{_\beta *_\alpha}\id)\Gamma =(\id
  \underset{N}{_\beta *_\alpha}\Gamma)\Gamma\]
\end{enumerate}
This last formula makes sense, thanks to the two preceeding ones and
\ref{spatial}. The von Neumann algebra $N$ will be called the \emph{basis} of $(N, M, \alpha, \beta, \Gamma)$. 

In (\cite{DC}, chap. 11), De Commer had shown that, if $N$ is
finite-dimensional, the Hilbert space
$L^2(M)\underset{\nu}{_\beta\otimes_\alpha}L^2(M)$ can be
isometrically imbedded into the usual Hilbert tensor product
$L^2(M)\otimes L^2(M)$ and the projection $p$ on this closed subspace
belongs to $M\otimes M$. Moreover, the fiber product
$M\underset{N}{_\beta*_\alpha}M$ can be then identified with the
reduced von Neumann algebra $p(M\otimes M)p$ and we can consider $\Gamma$ as a usual coproduct $M\mapsto M\otimes M$, but with the condition $\Gamma(1)=p$. 

A \emph{co-inverse} $R$ for a Hopf bimodule $(N, M, \alpha, \beta,
\Gamma)$ is an involutive ($R^2=\id$) anti-$*$-isomorphism of $M$
satisfying $R\circ\alpha=\beta$ (and therefore $R\circ\beta=\alpha$)
and $\Gamma\circ R=\varsigma_{N^{\op}}\circ
(R\underset{N}{_\beta*_\alpha}R)\circ \Gamma$, where
$\varsigma_{N^{\op}}$ is the flip from
$M\underset{N^{\op}}{_\alpha*_\beta}M$ onto
$M\underset{N}{_\beta*_\alpha}M$. A Hopf bimodule is called \emph{co-commutative} if $N$ is abelian, $\beta=\alpha$, and $\Gamma=\varsigma\circ\Gamma$. 

For an example, suppose that $\mathcal G$ is a measured groupoid, with
$\mathcal G^{(0)}$ as its set of units. We denote
by $r$ and $s$ the range and source applications from $\mathcal G$ to $\mathcal G^{(0)}$, given by
$xx^{-1}=r(x)$ and $x^{-1}x=s(x)$, and  by $\mathcal G^{(2)}$ the set of composable elements, i.e.\ 
\[\mathcal G^{(2)}=\{(x,y)\in \mathcal G^2 : s(x)=r(y)\}.\]
Let $(\lambda^u)_{u\in \mathcal G^{(0)}}$ be a Haar system on
$\mathcal G$ and $\nu$ a measure on $\mathcal G^{(0)}$. Let us denote by $\mu$ the measure on $\mathcal G$ given by integrating $\lambda^u$ by $\nu$,
\[\mu=\int_{{\mathcal G}^{(0)}}\lambda^ud\nu.\]
By definition,  $\nu$ is called \emph{quasi-invariant} if $\mu$ is equivalent to its image under the inversion $x\mapsto x^{-1}$ of $\mathcal G$ (see \cite{R},
 \cite{C2} II.5, \cite{Pa} and \cite{AR} for more details, precise definitions and examples of groupoids). 

In \cite{Y1}, \cite{Y2}, \cite{Y3} and \cite{Val2} was associated to a measured groupoid $\mathcal G$, equipped with a Haar system $(\lambda^u)_{u\in \mathcal G ^{(0)}}$ and a quasi-invariant measure $\nu$ on $\mathcal G ^{(0)}$, a
Hopf bimodule with an abelian underlying von Neumann algebra $(L^\infty (\mathcal G^{(0)}, \nu), L^\infty (\mathcal G, \mu), r_{\mathcal G}, s_{\mathcal G}, \Gamma_{\mathcal
G})$, where 
$r_{\mathcal G}(g)=g\circ r$ and $s_{\mathcal G}(g)=g\circ s$ for all
$g$ in $L^\infty (\mathcal G^{(0)})$
 and where
$\Gamma_{\mathcal G}(f)$, for $f$ in $L^\infty (\mathcal G)$, is the
function defined on $\mathcal G^{(2)}$ by $(s,t)\mapsto f(st)$. Thus,
$\Gamma_{\mathcal G}$ is an involutive homomorphism from $L^\infty (\mathcal G)$ into $L^\infty
(\mathcal G^{(2)})$, which can be identified with
$L^\infty (\mathcal G){_s*_r}L^\infty (\mathcal G)$.

It is straightforward to get that the inversion of the groupoid gives a co-inverse for this Hopf bimodule structure. 

\begin{proposition}[\cite{Ti2}]
\label{propV}
 Let $\bfG$ be a locally compact quantum group,  $(N, \ga,
  \widehat{\ga})$ a braided-commutative $\bfG$-Yetter-Drinfel'd
  algebra,  $\beta$ the injective anti-$*$-homomorphism from $N$ into
  $\bfG\ltimes_\ga N$ introduced in \ref{ThBC}, and $\nu$ a normal
  semi-finite faithful weight $\nu$ on $N$. Then the relative tensor
  product $(H\otimes H_\nu)\underset{\nu}{_\beta\otimes_\ga}(H\otimes
  H_\nu)$ can be canonically identified with $H\otimes H\otimes H_\nu$
  as follows:
  \begin{enumerate}
  \item  For any $\eta\in H$, $p\in\gN_\nu$, the vector $U^\ga_\nu(
    \eta\otimes J_\nu\Lambda_\nu(p))$ belongs to $D(_\ga (H\otimes
    H_\nu), \nu)$ and
    \[R^{\ga, \nu}(U^\ga_\nu( \eta\otimes
    J_\nu\Lambda_\nu(p)))=U^\ga_\nu l_\eta J_\nu pJ_\nu,\] where
    $l_\eta$ is the application $\zeta\rightarrow \eta\otimes\zeta$
    from $H_\nu$ into $H\otimes H_\nu$.
There exists a unitary $V_1$ from $(H\otimes
    H_\nu)\underset{\nu}{_\beta\otimes_\ga}(H\otimes H_\nu)$ onto
    $H\otimes H\otimes H_\nu$ such that
    \[V_1(\Xi\underset{\nu}{_\beta\otimes_\ga}U^\ga_\nu( \eta\otimes
    J_\nu\Lambda_\nu(p)))=\eta\otimes\beta(p^*)\Xi \quad\text{for all
    } \Xi\in H\otimes H_\nu,\] and
    $V_1(X\underset{N}{_\beta\otimes_\ga}(1_H\otimes
    1_{H_\nu}))=(1_H\otimes X)V_1$ for all $X\in\beta(N)'$, in
    particular, for $X \in \ga(N)$.  Morover, writing $\beta^{\dag}$
    for the map $x^{\op} \mapsto \beta(x)$, we have for all $x \in N$,
    \begin{align*}
      V_1[(1_H\otimes 1_{H_\nu})\underset{N}{_\beta\otimes_\ga}(1_H\otimes x^{\op})]&= (\id \otimes \beta^{\dag})(\ga^{\op}(x^{\op}))V_1, \\
      V_1[(1_H\otimes 1_{H_\nu})\underset{N}{_\beta\otimes_\ga}
      \beta(x)]&= (\id \otimes
      \beta^{\dag})(\widehat{\ga}^{\op}(x^{\op}))V_1.
    \end{align*}
  \item  For any $\xi\in H$, $q\in\gN_\nu$, the vector $U^\ga_\nu
    (U^{\widehat{\ga}}_\nu)^*(\xi\otimes\Lambda_\nu(q))$ belongs to
    $D(_\beta(H\otimes H_\nu), \nu^{\op})$ and
    \[R^{\beta, \nu^{\op}}(U^\ga_\nu
    (U^{\widehat{\ga}}_\nu)^*(\xi\otimes\Lambda_\nu(q)))=U^\ga_\nu(U^{\widehat{\ga}}_\nu)^*l_\xi
    q.\]
There exists a unitary $V_2$ from $(H\otimes
    H_\nu)\underset{\nu}{_\beta\otimes_\ga}(H\otimes H_\nu)$ onto
    $H\otimes H\otimes H_\nu$ such that
    \[V_2[U^\ga_\nu
    (U^{\widehat{\ga}}_\nu)^*(\xi\otimes\Lambda_\nu(q))\underset{\nu}{_\beta\otimes_\ga}\Xi]=\xi\otimes\ga(q)\Xi
    \quad\text{for all } \Xi\in H\otimes H_\nu,
    \]
    and $V_2((1_H\otimes
    1_{H_\nu})\underset{N}{_\beta\otimes_\ga}X)=(1_H\otimes X)V_2$ for
    all $X\in\ga(N)'$, in particular, for $X\in \beta(N)$.
  \item  $V_{2}V_{1}^{*}
    =\sigma_{12}(U^{\ga}_{\nu})_{13}(U^{\widehat{\ga}}_{\nu})_{23}(U^{\ga}_{\nu})_{23}^{*}=
    \sigma_{12}W_{12}(U^{\widehat{\ga}}_{\nu})_{23}(U^{\ga}_{\nu})^{*}_{23}W_{12}^{*}$.
  \end{enumerate}
\end{proposition}
\begin{proof}
(i) For all $n\in\gN_\nu$, 
\[U^\ga_\nu l_\eta J_\nu pJ_\nu\Lambda_\nu (n)=
U^\ga_\nu(\eta\otimes J_\nu pJ_\nu\Lambda_\nu (n))=
U^\ga_\nu(\eta\otimes nJ_\nu\Lambda_\nu (p))=
\ga(n)U^\ga_\nu(\eta\otimes J_\nu\Lambda_\nu (p)),\]
which gives the proof of the first part of (i). Let now $\eta'\in H$, $p'\in\gN_\nu$, $\Xi'\in H\otimes H_\nu$. Then
\[
\langle U_{\nu}^\ga( \eta'\otimes J_\nu\Lambda_\nu(p'))| U_{\nu}^\ga( \eta\otimes J_\nu\Lambda_\nu(p))\rangle _{\ga, \nu}^{\op} = J_{\nu}p^{*}J_{\nu} l_{\eta}^{*} l_{\eta'} J_{\nu}p'J_{\nu} = (\eta|\eta') J_{\nu}p^{*}p'J_{\nu}
\]
and hence
\begin{multline*}
(\Xi\underset{\nu}{_\beta\otimes_\ga}U^\ga( \eta\otimes J_\nu\Lambda_\nu(p))|\Xi'\underset{\nu}{_\beta\otimes_\ga}U^\ga( \eta'\otimes J_\nu\Lambda_\nu(p')))
=\\
=(\beta(\langle U^\ga( \eta'\otimes J_\nu\Lambda_\nu(p')), U^\ga( \eta\otimes J_\nu\Lambda_\nu(p))\rangle _{\ga, \nu}^{\op})\Xi|\Xi')=\\
=
(\eta|\eta')(\beta(p^*p')\Xi|\Xi')
\end{multline*}
which proves the existence of an isometry $V_1$ satisfying the above
formula. As the image of $V_1$ is dense in $H\otimes H\otimes H_\nu$,
we get that $V_{1}$ is unitary.

Next, let $z\in B(H)$, $x\in N$. Then
\begin{align*}
(z \otimes \beta(x))
V_1[\Xi\underset{\nu}{_\beta\otimes_\ga}U^\ga_\nu(\eta\otimes J_\nu\Lambda_\nu(p))]  &=
{z}\eta\otimes \beta(x)\beta(p^*)\Xi 
\\ &=  {z}\eta\otimes \beta((x^*p)^*)\Xi 
\\ &=
V_1[\Xi\underset{\nu}{_\beta\otimes_\ga}U^\ga_\nu(z\eta\otimes J_{\nu}\Lambda_\nu(x^{*}p))]
\\ &=
V_1[\Xi\underset{\nu}{_\beta\otimes_\ga}U^\ga_\nu(z\otimes x^{o})(\eta\otimes J_{\nu}\Lambda_\nu(p))],
\end{align*}
that is,
$(z \otimes \beta(x))
V_1 = V_{1}(1 \underset{\nu}{_\beta\otimes_\ga} U^\ga_\nu(z\otimes x^{o}) (U^{\ga}_{\nu})^{*})$.
In particular, \begin{align*}
  (\id \otimes \beta^{\dag})(\ga^{o}(x^{o}))V_1 &= 
  V_{1}(1 \underset{\nu}{_\beta\otimes_\ga} U^{\ga}_{\nu}\ga^{o}(x^{o}) (U^{\ga}_{\nu})^{*} ) = 
  V_{1}(1 \underset{\nu}{_\beta\otimes_\ga} (1_{H} \otimes x^{o})), \\
  (\id \otimes \beta^{\dag})(\widehat{\ga}^{o}(x^{o}))V_1 &= 
  V_{1}(1 \underset{\nu}{_\beta\otimes_\ga} U^{\ga}_{\nu}\widehat{\ga}^{o}(x^{o}) (U^{\ga}_{\nu})^{*} ) = 
  V_{1}(1 \underset{\nu}{_\beta\otimes_\ga} \beta(x)).
\end{align*}

(ii) We proceed as above. First,
 we have
\[U^\ga_\nu(U^{\widehat{\ga}}_\nu)^*l_\xi qJ_\nu\Lambda_\nu(n)=U^\ga_\nu(U^{\widehat{\ga}}_\nu)^*(\xi\otimes J_\nu nJ_\nu\Lambda_\nu (q))=
\beta(n^*)U^\ga_\nu(U^{\widehat{\ga}}_\nu)^*(\xi\otimes\Lambda_\nu (q)),\]
which gives the proof of the first part of (ii). Let now $\xi'\in H$, $q'\in\gN_\nu$. Then
\begin{multline*}
(U^\ga_\nu (U^{\widehat{\ga}}_\nu)^*(\xi\otimes\Lambda_\nu(q))\underset{\nu}{_\beta\otimes_\ga}\Xi|U^\ga_\nu (U^{\widehat{\ga}}_\nu)^*(\xi'\otimes\Lambda_\nu(q'))\underset{\nu}{_\beta\otimes_\ga}\Xi')=\\
=(\ga(\langle U^\ga_\nu (U^{\widehat{\ga}}_\nu)^*(\xi\otimes\Lambda_\nu(q)), U^\ga_\nu (U^{\widehat{\ga}}_\nu)^*(\xi'\otimes\Lambda_\nu(q'))\rangle _{\beta, \nu^{\op}})\Xi|\Xi')=\\
=(\xi|\xi')(\ga(q'^*q)\Xi|\Xi')
\end{multline*}
which proves the existence of an isometry $V_2$ satisfying the above
formula. Again, as the image of $V_2$ is dense in $H\otimes H\otimes H_\nu$, we get (ii).

(iii) Applying (i), we get
\begin{multline*}
V_1[U^\ga_\nu (U^{\widehat{\ga}}_\nu)^*(\xi\otimes\Lambda_\nu(q))\underset{\nu}{_\beta\otimes_\ga}U^\ga_{\corr{\nu}}( \eta\otimes J_\nu\Lambda_\nu(p))]=\\
=\eta\otimes\beta(p^*)U^\ga_\nu (U^{\widehat{\ga}}_\nu)^*(\xi\otimes\Lambda_\nu(q))
=\eta\otimes U^\ga_\nu (U^{\widehat{\ga}}_\nu)^*(\xi\otimes J_\nu pJ_\nu\Lambda_\nu(q))=\\
=\eta\otimes U^\ga_\nu (U^{\widehat{\ga}}_\nu)^*(\xi\otimes qJ_\nu\Lambda_\nu(p)),
\end{multline*}
and, applying (ii), we get
\begin{multline*}\ V_2[U^\ga_\nu (U^{\widehat{\ga}}_\nu)^*(\xi\otimes\Lambda_\nu(q))\underset{\nu}{_\beta\otimes_\ga}U^\ga_{\corr{\nu}}( \eta\otimes J_\nu\Lambda_\nu(p))]=\\
=\xi\otimes\ga(q)U^\ga_{\corr{\nu}}( \eta\otimes J_\nu\Lambda_\nu(p))
=\xi\otimes U^\ga_\nu(\eta\otimes qJ_\nu\Lambda_\nu(p)),
\end{multline*}
from which we easily get $(1_H\otimes U^{\widehat{\ga}}_\nu
(U^\ga_\nu)^*)V_1=(\sigma\otimes 1_{H\nu})(1_H\otimes
(U^\ga_\nu)^*)V_2$. Using Corollary \ref{corintegrable3}(iii), we conclude
\begin{align*}
V_{2}V_{1}^{*}=\sigma_{12}(U^{\ga}_{\nu})_{13}(U^{\widehat{\ga}}_{\nu})_{23}(U^{\ga}_{\nu})_{23}^{*}
&=
\sigma_{12}W_{12}(U^{\widehat{\ga}}_{\nu})_{23}(U^{\ga}_{\nu})_{23}^{*}W_{12}^{*}.  \qedhere
\end{align*}
 \end{proof}
\begin{theorem}[\cite{Ti2}]
\label{th1}
 Let $\bfG$ be a locally compact quantum group and  $(N, \ga, \widehat{\ga})$  a braided-commutative $\bfG$-Yetter-Drinfel'd algebra. We use the notations of \ref{propV}. 
 \begin{enumerate}
 \item  For $X\in\bfG\ltimes_\ga N$, let
   $\widetilde{\Gamma}(X)=V_1^*\tilde{\ga}(X)V_1$. Then this defines a
   normal $*$-homomorphism $\widetilde{\Gamma}$ from $\bfG\ltimes_\ga
   N$ into $(\bfG\ltimes_\ga
   N)\underset{N}{_\beta*_\ga}(\bfG\ltimes_\ga N)$. For all $x\in N$,
   \begin{align*}
     \widetilde{\Gamma}(\ga(x))&=\ga(x)\underset{N}{_\beta\otimes_\ga}(1_H\otimes
     1_{H_\nu}), \\ \widetilde{\Gamma}(\beta(x))&=(1_H\otimes
     1_{H_\nu})\underset{N}{_\beta\otimes_\ga}\beta(x),
   \end{align*}
   and for all $y\in \widehat M$,
   \[\widetilde{\Gamma}(y\otimes
   1_{H_\nu})=V_1^*(\widehat{\Gamma}^{\op}(y) \otimes
   1)V_1=V_2^*(\widehat{\Gamma}(y)\otimes 1_{H_\nu})V_2.\]
   
 \item  $(N, \bfG\ltimes_\ga N, \ga, \beta, \widetilde{\Gamma})$ is a
   Hopf bimodule.
   \item We have $\tilde\ga(\beta^\dag (x^o))=(id\otimes\beta^\dag)\widehat{\ga}^o(x^o)$, where $\beta^\dag$ has been defined in \ref{propV}. 
 \end{enumerate}
\end{theorem}
\begin{proof}
  (i) and (iii) Let $x\in N$. Then \ref{propV}(iii) implies
\[\widetilde{\Gamma}(\ga(x))=V_1^*\tilde{\ga}(\ga(x))V_1=V_1^*(1_H\otimes \ga(x))V_1=\ga(x)\underset{N}{_\beta\otimes_\alpha}(1_H\otimes 1_{H_\nu}),\]
in particular, $\widetilde{\Gamma}(\ga(x))$ lies in  $(\bfG\ltimes_\ga N)\underset{N}{_\beta*_\ga}(\bfG\ltimes_\ga
N)$. 

Next, by definition,
\begin{align*}
  \widetilde{\Gamma}(\beta(x))&=
\Ad(V_1^*\widehat{W}^{\op*}_{12})[1
  \otimes \beta(x)] =\Ad(V_1^*\widehat{W}^{\op*}_{12}(U^{\ga}_{\nu})_{23})[1
  \otimes \widehat\ga^{\op}(x^{\op})].
\end{align*}
Since
$U^\ga_\nu \in M\otimes B(H_\nu)$ commutes with $\widehat{W}^{\op} \in
\widehat{M}\otimes M'$,  this is equal to
\begin{align*}
  \Ad(V_1^*(U^{\ga}_{\nu})_{23})\widehat{W}^{o*}_{12})[1
  \otimes \widehat\ga^{\op}(x^{\op})] 
  = \Ad(V_1^*)[(\id \otimes
\beta^{\dag})(\widehat\ga^{\op}(x^{\op}))]= (1_H\otimes 1_{H_\nu})\underset{N}{_\beta\otimes_\alpha}\beta(x),
\end{align*}
where we used \ref{propV}(i). From this calculation, one gets (iii) as well. 

For $y \in \widehat M$, we get by definition of $\tilde \Gamma$
\begin{align*}
  \tilde \Gamma(y\otimes 1) = \Ad(V_{1}^{*})[\tilde{\ga}(y\otimes
  1)] 
  = \Ad(V_{1}^{*})[\widehat{\Gamma}^{\op}(y) \otimes 1] 
=\Ad(V_{1}^{*}W_{12})[y\otimes 1\otimes 1].
\end{align*}
By \ref{propV} (iii), $V_{1}^{*}W_{12} =
V_{2}^{*}\sigma_{12}W_{12}(U^{\widehat\ga}_{\nu})_{23}
(U^{\ga}_{\nu})_{23}^{*}$ and hence
\begin{align*}
  \tilde \Gamma(y\otimes 1) &= \Ad(  V_{2}^{*}\sigma_{12}W_{12}(U^{\widehat\ga}_{\nu})_{23}
  (U^{\ga}_{\nu})_{23}^{*})[y\otimes 1\otimes 1] =
  \Ad(V_{2}^{*})(\widehat{\Gamma}(y) \otimes 1).
\end{align*}
To see that $\tilde \Gamma(y\otimes 1)$ lies in   $(\bfG\ltimes_\ga N)\underset{N}{_\beta*_\ga}(\bfG\ltimes_\ga
N)$, note that for any $Y$ in $(\bfG\ltimes_{\ga} N)'$, 
\begin{align*}
  \Ad(V_{1})[Y\underset{N}{_\beta\otimes_\alpha}(1_H\otimes
  1_{H_\nu})] &= 1_{H} \otimes Y
  =\Ad(V_{2})[(1_H\otimes
  1_{H_\nu}) \underset{N}{_\beta\otimes_\alpha} Y]
\end{align*}
 by \ref{propV}, and $1_{H} \otimes Y$ commutes with
 \begin{align*}
 \Ad(V_{1})(\widetilde{\Gamma}(y\otimes 1))&=\widehat{\Gamma}^{\op}(y)
\otimes 1   &&\text{and} &\Ad(V_{2})(\widetilde{\Gamma}(y\otimes 1))&=\widehat{\Gamma}(y)
\otimes 1.
 \end{align*}

(ii)
To get (ii), we must verify that $\widetilde \Gamma$ is co-associative.
It is trivial to get that
\begin{align*}
  (\widetilde{\Gamma}\underset{N}{_\beta*_\ga}\id)\widetilde{\Gamma}(\ga(x))=\ga(x)\underset{N}{_\beta\otimes_\ga}(1_H\otimes 1_{H_\nu})\underset{N}{_\beta\otimes_\ga}(1_H\otimes 1_{H_\nu})=(\id\underset{N}{_\beta*_\ga}\widetilde{\Gamma})\widetilde{\Gamma}(\ga(x))
\end{align*}
for all $x\in N$.

  Next, let $y\in \widehat{M}$ and consider the following
diagrams,
\begin{align*}
  \xymatrix@C=20pt@R=20pt{
    \widehat M \otimes 1_{H_{\nu}} \ar[r]^{\widehat\Gamma \otimes \id}
    \ar[rd]_(0.4){\widetilde \Gamma} & \widehat M \otimes \widehat M \otimes
    1_{H_{\nu}} \ar[d]^{\ad_{(V_{2}^{*})} \otimes \id} \ar[r]^(0.4){\id
      \otimes \widetilde \Gamma} & \widehat M \otimes  (\bfG \ltimes_{\ga} N)\underset{N}{_{\beta}*_{\ga}}
(\bfG \ltimes_{\ga} N)  \ar[d]_{\ad_{(V_{2}^{*})}
 \underset{N}{_{\beta}*_{\ga}} \id}\\
& (\bfG \ltimes_{\ga} N) \underset{N}{_{\beta}*_{\ga}}
(\bfG \ltimes_{\ga} N)  \ar[r]_(0.4){\id
  \underset{N}{_{\beta}*_{\ga}}\widetilde \Gamma}& 
(\bfG \ltimes_{\ga} N) \underset{N}{_{\beta}*_{\ga}}
(\bfG \ltimes_{\ga} N) \underset{N}{_{\beta}*_{\ga}}
(\bfG \ltimes_{\ga} N) \\
\widehat M \otimes 1_{H_{\nu}} \ar[r]^(0.4){\varsigma_{23}(\widehat \Gamma \otimes
  \id)} \ar[rd]_(0.4){\tilde \Gamma} & \widehat M \otimes 1_{H_{\nu}} \otimes
\widehat M \ar[r]^(0.4){\widetilde \Gamma \otimes \id} \ar[d]^{\id \otimes
  \ad_{(\tilde V_{1}^{*})}} & 
(\bfG \ltimes_{\ga} N) \underset{N}{_{\beta}*_{\ga}}
(\bfG \ltimes_{\ga} N)  \otimes \widehat M
\ar[d]_{\id  \underset{N}{_{\beta}*_{\ga}}  \ad_{(\tilde V_{1}^{*})}}
\\
& (\bfG \ltimes_{\ga} N) \underset{N}{_{\beta}*_{\ga}}
(\bfG \ltimes_{\ga} N)  \ar[r]_(0.4){\widetilde \Gamma
  \underset{N}{_{\beta}*_{\ga}}\id} &(\bfG \ltimes_{\ga} N) \underset{N}{_{\beta}*_{\ga}}
(\bfG \ltimes_{\ga} N) \underset{N}{_{\beta}*_{\ga}}
(\bfG \ltimes_{\ga} N)
} 
\end{align*}
where $\tilde V_{1}$ denotes the composition of  the unitary $V_1$
with the flip $\eta \otimes \xi \otimes \zeta \mapsto \xi \otimes \zeta
\otimes \eta$ (for $\xi$, $\eta$ in $H$ and $\zeta$ in $H_\nu$). 
The triangles commute by (i) and the squares
commutes by definition of $V_{1}$ and $V_{2}$. Next, consider the following diagram:
\begin{align*} 
  \xymatrix@C=-30pt@R=15pt{
    & \widehat M \otimes 1_{H_{\nu}} \ar[rd]^{\varsigma_{23}(\widehat\Gamma
      \otimes \id)}   \ar[ld]_{\widehat\Gamma \otimes \id} & \\
    \widehat M \otimes \widehat M \otimes 1_{H_{\nu}} \ar[dd]_{\id \otimes
      \widetilde \Gamma}  \ar[rd]^{\id \otimes \varsigma_{23}(\widehat\Gamma
      \otimes \id)} & &  \widehat
    M \otimes 1_{H_{\nu}} \otimes \widehat M \ar[dd]^{\widetilde\Gamma \otimes
    \id}  \ar[ld]_{\widehat\Gamma \otimes \id \otimes \id}\\
&\widehat M \otimes \widehat M \otimes 1_{H_{\nu}} \otimes \widehat M
\ar[ld]_{\id \otimes \ad_{(\tilde V_{1}^{*})}} \ar[rd]^{\ad_{(V_{2}^{*})}
\otimes \id} & \\
\widehat M \otimes (\bfG \ltimes_{\ga} N) \underset{N}{_{\beta}*_{\ga}}
(\bfG \ltimes_{\ga} N)
 \ar[rd]^(0.6){\ad_{(V_{2}^{*})}
 \underset{N}{_{\beta}*_{\ga}} \id}
&& (\bfG \ltimes_{\ga} N) \underset{N}{_{\beta}*_{\ga}}
(\bfG \ltimes_{\ga} N) \otimes \widehat M \ar[ld]_(0.6){\id  \underset{N}{_{\beta}*_{\ga}} \ad_{(\tilde V_{1}^{*})}}\\
&(\bfG \ltimes_{\ga} N) \underset{N}{_{\beta}*_{\ga}}
(\bfG \ltimes_{\ga} N) \underset{N}{_{\beta}*_{\ga}}
(\bfG \ltimes_{\ga} N) &
  }
\end{align*}
The upper middle cell commutes by co-associativity of $\widehat \Gamma$, the
left and the right triangle commute by (i), and the
lower middle cell commutes because the following diagram does,
\begin{align*}
\xymatrix@C=60pt{
  (H \otimes H_{\nu}) \underset{\nu}{_{\beta}\otimes_{\ga}} (H \otimes
  H_{\nu}) \underset{\nu}{_{\beta}\otimes_{\ga}} (H \otimes H_{\nu})
  \ar[r]^{V_{2} \underset{N}{_{\beta}\otimes_{\ga}} \id}
\ar[d]_{\id \underset{N}{_{\beta}\otimes_{\ga}} \tilde V_{1}} &
  H \otimes ((H \otimes H_{\nu})\underset{\nu}{_{\beta}\otimes_{\ga}} (H
  \otimes H_{\nu})) \ar[d]^{\id \otimes \tilde V_{1}} \\
   ((H \otimes H_{\nu})\underset{\nu}{_{\beta}\otimes_{\ga}} (H
  \otimes H_{\nu})) \otimes H \ar[r]_{V_{2} \otimes \id} &
  H \otimes (H \otimes H_{\nu}) \otimes H
}
\end{align*}
where both compositions are given by
\begin{align*}
  U^{\ga}_{\nu}(U^{\widehat\ga}_{\nu})^{*}(\xi \otimes \Lambda_{\nu}(q))
  \underset{\nu} {_{\beta}\otimes_{\ga}} \Xi \underset{\nu}{_{\beta}\otimes_{\ga}}
  U^{\ga}_{\nu}(\eta \otimes J_{\nu}\Lambda_{\nu}(p)) \mapsto
  \xi \otimes \ga(q)\beta(p^{*})\Xi \otimes \eta.
\end{align*}
Combining everything,  we can conclude that
$(\widetilde \Gamma \underset{N}{_{\beta}*_{\ga}} \id) \circ \widetilde
  \Gamma  (y\otimes 1)=  (\id \underset{N}{_{\beta}*_{\ga}} \widetilde
  \Gamma) \circ\widetilde \Gamma (y\otimes 1)$.
 \end{proof}


\begin{proposition}[\cite{Ti2}]
\label{propI}
 Consider on the Hilbert space $H\otimes H_\nu$ the anti-linear operator:
\[I=U^\ga_\nu (J\otimes J_\nu)U^{\widehat{\ga}}_{\nu}(U^\ga_\nu)^*=U^\ga_\nu(U^{\widehat{\ga}}_{\nu})^*(J\otimes J_\nu)(U^\ga_\nu)^*=U^\ga_\nu(U^{\widehat{\ga}}_{\nu})^*J_{\widetilde{\hat{\nu}}}U^{\widehat{\ga}}_{\nu}(U^\ga_\nu)^*,\]
where $\widetilde{\hat{\nu}}$ denotes the dual weight of $\nu$ on the
crossed product $\widehat{\bfG}\ltimes_{\widehat{\ga}}N$. 
\begin{enumerate}
\item  $I$ is a bijective isometry and $I^2=1$.
\item  $I\ga(x)^{*}I=\beta(x)$ and $I\beta(x)^*I=\ga(x)$ for all $x\in
  N$.
\item  $I(y^*\otimes 1)I=\widehat{R}(y)\otimes 1$ for all $y\in
  \widehat{M}$.
\item  If $\sigma_\nu$ denotes the flip from $(H\otimes
  H_\nu)\underset{\nu}{_\beta\otimes_\ga}(H\otimes H_\nu)$ to
  $(H\otimes H_\nu)\underset{\nu^{\op}}{_\ga\otimes_\beta}(H\otimes
  H_\nu)$, then
  \[V_2=(J\otimes
  I)V_1(I\underset{N^{\op}}{_\ga\otimes_\beta}I)\sigma_\nu.\]
\end{enumerate}
\end{proposition}
\begin{proof}
  (i) The relation $U^{\widehat{\ga}}_{{\nu}}=J_{\widetilde{\hat
      \nu}}(J\otimes J_\nu)$ (\ref{action}) shows that the three
  expressions given for $I$ coincide and that $I$ is isometric,
  bijective, anti-linear, and equal to $I^*$. Moreover, the formula
  $I=U^\ga_\nu(U^{\widehat{\ga}}_{\nu})^*J_{\widetilde{\hat{\nu}}}U^{\widehat{\ga}}_{\nu}(U^\ga_\nu)^*$
  shows that $I^{2}=1_{H} \otimes 1_{H_{\nu}}$.

  (ii) We only need to prove the first equation. But by \ref{ThBC},
  \begin{align*}
    I\ga(x)^{*}I^* &= \Ad(
    U^{\ga}_{\nu}(U^{\widehat{\ga}}_{\nu})^{*}(J \otimes
    J_{\nu})(U^{\ga}_{\nu})^{*})[\ga(x)^{*}] =
    \Ad(U^{\ga}_{\nu}(U^{\widehat{\ga}}_{\nu})^{*})[1 \otimes x^{\op}]
    = \beta(x).
  \end{align*}

  (iii)  Using \ref{corintegrable3}(iii) and the fact that
  $U^{\ga}_{\nu}$ is a representation, we find that
\begin{align*}
(\widehat{J}\otimes I)W_{12}(\widehat{J}\otimes I)
&=
\Ad((U^\ga_\nu)_{23}(\widehat{J}\otimes J\otimes J_\nu)(U^{\widehat{\ga}}_\nu)_{23}(U^\ga_\nu)_{23}^*)[W_{12}]\\
&=
\Ad((U^\ga_\nu)_{23}(\widehat{J}\otimes J\otimes J_\nu))[(U^\ga_\nu)_{13}^*W_{12}]\\
&=
\Ad((U^\ga_\nu)_{23})[(U^\ga_\nu)_{13}W^{*}_{12}] \\
&=
W_{12}^*.
\end{align*}
 For any $\xi$, $\eta$ in $H$, we can conclude that
\[I(J(\omega_{\xi, \eta}\otimes \id)(W)^*J\otimes 1)I=I((\omega_{\widehat{J}\eta, \widehat{J}\xi}\otimes \id)(W)\otimes 1)I=
(\omega_{\xi, \eta}\otimes \id)(W)^*\otimes 1,\]
from which (iii) follows by continuity.

(iv) By (ii), 
\begin{align*}
V_{1}(I\underset{\nu^{\op}}{_\ga\otimes_\beta}I)\sigma_\nu [U^\ga_\nu (U^{\widehat{\ga}}_\nu)^*(\xi\otimes\Lambda_\nu (q))\underset{\nu}{_\beta\otimes_\ga} \Xi] &=
V_{1}
[I\Xi \underset{\nu}{_\beta\otimes_\ga}U^\ga_\nu(J\xi\otimes J_\nu\Lambda_\nu(q))]  \\ &=
J\xi\otimes \beta(q^*)I \Xi \\ & =  J\xi \otimes I\ga(q)\Xi \\ &=
(J\otimes I)V_2[U^\ga_\nu
(U^{\widehat{\ga}}_\nu)^*(\xi\otimes\Lambda_\nu
(q))\underset{\nu}{_\beta\otimes_\ga} \Xi]. \qedhere
\end{align*}
\end{proof}

\begin{theorem}[\cite{Ti2}]
\label{th2}
 Let $\bfG$ be a locally compact quantum group,  $(N, \ga, \widehat{\ga})$ a braided-commutative $\bfG$-Yetter-Drinfel'd algebra and $I$  the anti-linear surjective isometry constructed in \ref{propI}. Then:
 \begin{enumerate}
 \item  For all $z\in \bfG\ltimes_\ga N$, let
   $\widetilde{R}(z)=Iz^*I$. Then $\widetilde{R}$ is an involutive
   anti-$*$-isomorphism of $\bfG\ltimes_\ga N$, and
   $\widetilde{R}(\ga(x))=\beta(x)$, $\widetilde{R}(\beta(x))=\ga(x)$
   and $\widetilde{R}(y\otimes 1_{H_\nu})=\widehat{R}(y)\otimes
   1_{H_\nu}$ for all $x\in N$ and $y\in \widehat M$.
 \item  $\widetilde{R}$ is a co-inverse for the Hopf bimodule $(N,
   \bfG\ltimes_\ga N, \ga, \beta, \widetilde{\Gamma})$ constructed in
   \ref{th1}.
 \end{enumerate}
\end{theorem}
\begin{proof}
 (i)  This is just a straightforward corollary of \ref{propI}(ii) and (iii).

(ii) We need to prove that
\[\widetilde{\Gamma}=\varsigma_{N^{\op}} (\widetilde{R}\underset{N}{_\beta\otimes_\alpha}\widetilde{R})\widetilde{\Gamma}\widetilde{R}.\]
 Using  (i), we find that for $x\in N$,
\begin{align*}
\varsigma_{N^{\op}}(\widetilde{R}\underset{N}{_\beta\otimes_\alpha}\widetilde{R})\widetilde{\Gamma}\widetilde{R}(\ga(x))
&=
\varsigma_{N^{\op}}(\widetilde{R}\underset{N}{_\beta\otimes_\alpha}\widetilde{R})\widetilde{\Gamma}(\beta(x))\\
&=
\varsigma_{N^{\op}}(\widetilde{R}\underset{N}{_\beta\otimes_\alpha}\widetilde{R})((1_H\otimes 1_{H_\nu})\underset{N}{_\beta\otimes_\alpha}\beta(x))
=
\ga(x)\underset{N}{_\beta\otimes_\ga}(1_H\otimes 1_{H_\nu})
\end{align*}
coincides with $\widetilde{\Gamma}(\ga(x))$. For $y\in \widehat{M}$,
we conclude from \ref{th1} and \ref{propI}(iv) that
\begin{align*}
\varsigma_{N^{\op}}(\widetilde{R}\underset{N}{_\beta\otimes_\alpha}\widetilde{R})\widetilde{\Gamma}\widetilde{R}(y\otimes 1_{H_\nu})
&=
\varsigma_{N^{\op}}(\widetilde{R}\underset{N}{_\beta\otimes_\alpha}\widetilde{R})\widetilde{\Gamma}(\widehat{R}(y)\otimes 1_{H_\nu})\\
&=
\varsigma_{N^{\op}}(\widetilde{R}\underset{N}{_\beta\otimes_\alpha}\widetilde{R})[V_2^*(\widehat{\Gamma}(\widehat{R}(y)\otimes 1_{H_\nu})V_2]\\
&=
V_1^*((\widehat{R}\otimes\widehat{R})\widehat{\Gamma}(\widehat{R}(y)) \otimes 1_{H_{\nu}})V_1\\
&=
\widetilde{\Gamma}(y\otimes 1_{H_\nu})
\end{align*}
As $\bfG\ltimes_\ga N$ is the von Neumann algebra generated by $\ga(N)$ and $\widehat{M}\otimes 1_{H_\nu}$, this finishes the proof of (ii). 
\end{proof}

\begin{lemma}
\label{lemmaga} Let $\bfG$ be a locally compact quantum group,
  $(N, \ga, \widehat{\ga})$ a braided-com\-mutative
  $\bfG$-Yetter-Drinfel'd algebra, $\widetilde{\Gamma}$ the injective
  $*$-homomorphism from $\bfG\ltimes_\ga N$ into $(\bfG\ltimes_\ga
  N)\underset{N}{_\beta*_\ga}(\bfG\ltimes_\ga N)$ defined in
  \ref{th1}, $\tilde{\ga}$ the dual action of $\widehat{\bfG}$ on
  $\bfG\ltimes_\ga N$, and $V_1$ as in \ref{propV}.  Denote by $\tau$
  the flip from $(H\otimes
  H_\nu){}_\beta\underset{\nu}{\otimes}{}_{(1\otimes\ga)}(H\otimes
  H\otimes H_\nu)$ onto $H\otimes [(H\otimes
  H_\nu)\underset{\nu}{_\beta\otimes_\ga}(H\otimes H_\nu)]$ given by
\[\tau(\Xi_\beta\!\underset{\nu}{\otimes}\!{}_{(1\otimes\ga)}(\xi\otimes\Xi'))=\xi\otimes\Xi\underset{\nu}{_\beta\otimes_\alpha}\Xi'\]
  for all $\xi\in H$, $\Xi\in D(_\beta(H\otimes H_\nu), \nu^{\op})$,
  $\Xi'\in D(_\ga(H\otimes H_\nu), \nu)$. Then:
  \begin{enumerate}
  \item 
    $(\id\underset{N}{_\beta*_\ga}\tilde{\ga})\widetilde{\Gamma}(X)=\tau^*(\id\otimes\widetilde{\Gamma})\tilde{\ga}(X)\tau$
    for all $X\in \bfG\ltimes_\ga N$.
  \item 
    $V_2\widetilde{\Gamma}(X)V_2^*=(\widehat{R}\otimes\widetilde{R})\tilde{\ga}(\widetilde{R}(X))$.
  \end{enumerate}
\end{lemma}
\begin{proof}
  (i) For any $x'\in M'$, we have $V_1[(1_H\otimes
  1_{H_\nu})\underset{N}{_\beta\otimes_\ga}(x'\otimes
  1_{H_\nu})]=(x'\otimes 1_H\otimes 1_{H_\nu})V_1$. As
  $\widehat{W}^{\op}$ belongs to $\widehat{M}\otimes M'$, we infer
\begin{align} \label{eq:tautau}
   (1_H\otimes V_1)\tau [(1_H\otimes 1_{H_\nu})_\beta\!\underset{N}{\otimes}\!{}_{1\otimes\ga}(\widehat{W}^{\op}\otimes 1_{H_\nu})]=
(\widehat{W}^{\op}\otimes 1_H\otimes 1_{H_\nu})(1_H\otimes V_1)\tau.
 \end{align}
Therefore, we can conclude that for all $X \in \bfG\ltimes_{\ga} N$,
\begin{align*}
(\id\underset{N}{_\beta*_\ga}\tilde{\ga})\widetilde{\Gamma}(X)&=
\Ad([(1_H\otimes
1_{H_\nu})_\beta\!\underset{N}{\otimes}\!{}_{1\otimes\ga}(\widehat{W}^{o*}\otimes
1_{H_\nu})]\tau^*(1_H\otimes  V_1^*))[1_{H} \otimes
\tilde{\ga}(X)]
\\ &=
\Ad(\tau^*(1_H\otimes V_1^*)(\widehat{W}^{o*}\otimes 1_H\otimes
1_{H_\nu}))[1_{H} \otimes \tilde{\ga}(X)] \\
&= \Ad(\tau^*(1_H\otimes V_1^*))[(\widehat{\Gamma}^{\op}\otimes
\id)\tilde{\ga}(X)] \\
&=
\Ad(\tau^*(1_H\otimes V_1^*))[(\id\otimes\tilde{\ga})\tilde{\ga}(X)]\\
&=
\tau^*(\id\otimes\widetilde{\Gamma})\tilde{\ga}(X)\tau.
\end{align*}

(ii) By \ref{propI}(iii),
\begin{align*}
  \Ad(V_2)[\widetilde{\Gamma}(X)] &= \Ad((J\otimes
  I)V_1\sigma_{\nu^{\op}}(I\underset{N}{_\beta\otimes_\ga}I))[\widetilde{\Gamma}(X)] \\
  &= \Ad((J\otimes I)V_1)[\widetilde{\Gamma}\widetilde{R}(X^*)]\\
  &= (\widehat{R}\otimes
  \widetilde{R})\tilde{\ga}(\widetilde{R}(X)). \qedhere
\end{align*}
\end{proof}

\section{Measured quantum groupoid structure associated to a braided-commutative Yetter-Drinfel'd algebra equipped with an appropriate weight}
\label{MQG}

In this chapter, after recalling the definition of a measured quantum
groupoid (\ref{defMQG}) and describing the major data associated to a
measured quantum groupoid (\ref{defW}, \ref{data}), we try to
construct,  given a braided-commutative $\bfG$-Yetter-Drinfel'd
algebra $(A, \ga, \widehat{\ga})$ and a normal semi-finite faithful
weight on $N$, a structure of a measured quantum groupoid, denoted
$\mathfrak G(N, \ga, \widehat{\ga}, \nu)$, on the crossed product
$\bfG\ltimes_\ga N$ or, more precisely, on the Hopf bimodule
constructed in \ref{th2}.  Without any hypothesis on the normal
faithful semi-finite weight $\nu$ on $N$, we construct a
left-invariant operator-valued weight (\ref{th3}) and a
right-invariant one (\ref{th3}), and we give a necessary and
sufficient condition for a weight $\nu$ on $N$ to be relatively invariant with respect to these two operator-valued weights (\ref{thmqg}). This condition is clearly satisfied (\ref{cormqg}) if $\nu$ is $k$-invariant with respect to $\ga$ (for $k$ affiliated to $Z(M)$, or $k=\delta^{-1}$). 

\subsection{Definition of measured quantum groupoids (\cite{L}, \cite{E2})}
\label{defMQG}

A \emph{measured quantum groupoid} is an octuple $\mathfrak {G}=(N, M, \alpha, \beta, \Gamma, T, T', \nu)$ such that (\cite{E2}, 3.8):

(i) $(N, M, \alpha, \beta, \Gamma)$ is a Hopf bimodule, 

(ii) $T$ is a left-invariant normal, semi-finite, faithful operator-valued weight from $M$ to $\alpha (N)$ (to be more precise, from $M^{+}$ to the extended positive elements of $\alpha(N)$ (cf. \cite{T2} IX.4.12)), which means that, for any $x\in\gM_T^+$, we have $(\id\underset{\nu}{_\beta*_\alpha}T)\Gamma(x)=T(x)\underset{N}{_\beta\otimes_\alpha}1$. 

(iii) $T'$ is a right-invariant normal, semi-finite, faithful operator-valued weight from $M$ to $\beta (N)$, which means that, for any $x\in\gM_{T'}^+$, we have $(T'\underset{\nu}{_\beta*_\alpha}\id)\Gamma(x)=1\underset{N}{_\beta\otimes_\alpha}T'(x)$. 

(iv) $\nu$ is normal semi-finite faithful weight on $N$, which is relatively invariant with respect to $T$ and $T'$, which means that the modular automorphisms groups of the weights $\Phi=\nu\circ\alpha^{-1}\circ T$ and $\Psi=\nu\circ\beta^{-1}\circ T'$ commute. The weight $\Phi$ will be called left-invariant, and $\Psi$ right-invariant. 

For example, 
let $\mathcal G$ be a measured groupoid equipped with a left Haar
system $(\lambda^u)_{u\in\mathcal G^{(0)}}$ and a quasi-invariant
measure $\nu$ on $\mathcal G^{(0)}$. Let us use the notations
introduced in \ref{defHopf}. If $f\in L^\infty(\mathcal G, \mu)^+$,
consider the function on $\mathcal G^{(0)}$, $u\mapsto \int_{\mathcal
  G}fd\lambda^u$, which belongs to $L^\infty (\mathcal G^{(0)},
\nu)$. The image of this function by the homomorphism $r_\mathcal G$
is the function on $\mathcal G$, $\gamma\mapsto \int_{\mathcal
  G}fd\lambda^{r(\gamma)}$, and the application which sends $f$ to
this function can be considered as an operator-valued weight from
$L^\infty(\mathcal G, \mu)$ to $r_{\mathcal G}(L^\infty (\mathcal
G^{(0)}, \nu))$ which is normal, semi-finite and faithful.  By
definition of the Haar system $(\lambda^u)_{u\in\mathcal G^{(0)}}$, it
is left-invariant in the sense of (ii). We shall denote this operator-valued weight from $L^\infty(\mathcal G, \mu)$ to $r_{\mathcal G}(L^\infty (\mathcal G^{(0)}, \nu))$ by $T_\mathcal G$.
 If we write $\lambda_u$ for the image of $\lambda^u$ under the inversion $x\mapsto x^{-1}$ of the groupoid $\mathcal G$, starting from the application which sends $f$ to the function on $\mathcal G^{(0)}$ defined by $u\mapsto \int_{\mathcal G}fd\lambda_u$, we define a normal semifinite faithful operator-valued weight from $L^\infty(\mathcal G, \mu)$ to $s_{\mathcal G}(L^\infty (\mathcal G^{(0)}, \nu))$, which is right-invariant in the sense of (ii),  and which we shall denote by $T_{\mathcal G}^{(-1)}$.
 
  We then get that
\[(L^\infty(\mathcal G^{(0)}, \nu), L^\infty(\mathcal G, \mu), r_\mathcal G, s_\mathcal G, \Gamma_\mathcal G, T_\mathcal G, T_{\mathcal G}^{(-1)}, \nu)\] is a measured quantum groupoid, which we shall denote again $\mathcal G$. 

It can be proved (\cite{E4}) that any measured quantum groupoid, whose underlying von Neumann algebra is abelian, is of that type.

\subsection{Pseudo-multiplicative unitary.} 
\label{defW}
Let $\mathfrak{G}=(N, M, \alpha, \beta, \Gamma, T, T', \nu)$ be an
octuple satisfying the axioms (i), (ii) (iii) of \ref{defMQG}. We
shall write $H=H_\Phi$, $J=J_\Phi$ and $\gamma(n)=J\alpha(n^*)J$ for
all $n\in N$.

Then (\cite{L}, 3.7.3 and 3.7.4), $\mathfrak {G}$ can be equipped with
a pseudo-multiplicative unitary $W$ which is a unitary from
$H\underset{\nu}{_\beta\otimes_\alpha}H$ onto
$H\underset{\nu^{\op}}{_\alpha\otimes_{\gamma}}H$ (\cite{E2}, 3.6)
that intertwines $\alpha$, $\gamma$, $\beta$  in the following way:
for all $X\in N$, 
\begin{align*}
W(\alpha
(X)\underset{N}{_\beta\otimes_\alpha}1)&=
(1\underset{N^{\op}}{_\alpha\otimes_{\gamma}}\alpha(X))W, \\
W(1\underset{N}{_\beta\otimes_\alpha}\beta
(X)) &=(1\underset{N^{\op}}{_\alpha\otimes_{\gamma}}\beta (X))W, \\
W(\gamma(X) \underset{N}{_\beta\otimes_\alpha}1) &=
(\gamma(X)\underset{N^{\op}}{_\alpha\otimes_{\gamma}}1)W, \\
W(1\underset{N}{_\beta\otimes_\alpha}\gamma(X))&=
(\beta(X)\underset{N^{\op}}{_\alpha\otimes_{\gamma}}1)W.
\end{align*}
Moreover, the operator $W$ satisfies the \emph{pentagonal relation}
\[(1\underset{N^{\op}}{_\alpha\otimes_{\gamma}}W)
(W\underset{N}{_\beta\otimes_\alpha}1_{H})
=(W\underset{N^{\op}}{_\alpha\otimes_{\gamma}}1)
\sigma^{23}_{\alpha, \beta}(W\underset{N}{_{\gamma}\otimes_\alpha}1)
(1\underset{N}{_\beta\otimes_\alpha}\sigma_{\nu^{\op}})
(1\underset{N}{_\beta\otimes_\alpha}W),\]
where $\sigma^{23}_{\alpha, \beta}$
goes from $(H\underset{\nu^{\op}}{_\alpha\otimes_{\gamma}}H)\underset{\nu}{_\beta\otimes_\alpha}H$ to $(H\underset{\nu}{_\beta\otimes_\alpha}H)\underset{\nu^{\op}}{_\alpha\otimes_{\gamma}}H$, 
and $1\underset{N}{_\beta\otimes_\alpha}\sigma_{\nu^{\op}}$ goes from $H\underset{\nu}{_\beta\otimes_\alpha}(H\underset{\nu^{\op}}{_\alpha\otimes_{\gamma}}H)$ to $H\underset{\nu}{_\beta\otimes_\alpha}H\underset{\nu}{_{\gamma}\otimes_\alpha}H$. 
The operators in this formula are well-defined because of the
intertwining relations listed above.

 Moreover, $W$, $M$ and $\Gamma$ are related by the following results:

(i) $M$ is the weakly closed linear space generated by all operators of the form $(\id*\omega_{\xi, \eta})(W)$, where $\xi\in D(_\alpha H, \nu)$ and $\eta\in D(H_{\gamma}, \nu^{\op})$ (see \cite{E2}, 3.8(vii)).

(ii)  $\Gamma(x)=W^*(1\underset{N^{\op}}{_\alpha\otimes_{\gamma}}x)W$
for all $x\in M$(\cite{E2}, 3.6). 

(iii) For any $x$, $y_1$, $y_2$ in $\gN_T\cap\gN_\Phi$, we have (\cite{E2}, 3.6)
\[(\id*\omega_{J_\Phi\Lambda_\Phi (y_1^*y_2), \Lambda_\Phi (x)})(W)=
(\id\underset{N}{_\beta*_\alpha}\omega_{J_\Phi\Lambda_\Phi(y_2), J_\Phi\Lambda_\Phi(y_1)})\Gamma (x^*).\]
If $N$ is finite-dimensional, using the fact that the relative tensor products can be identified with closed subspaces of the usual Hilbert tensor product (\ref{spatial}), we get that $W$ can be considered as a partial isometry, which is multiplicative in the usual sense (i.e.\ such that $W_{23}W_{12}=W_{12}W_{13}W_{23}$.)
\subsection{Other data associated to a measured quantum groupoid} (\cite{L}, \cite{E2})
\label{data}
 Suppose that $\gG=(N, M, \alpha, \beta, \Gamma, T, T', \nu)$ is a measured quantum groupoid in the sense of \ref{defMQG}. Let us write $\Phi=\nu\circ\alpha^{-1}\circ T$, which is a normal semi-finite faithful left-invariant weight on $M$. Then:

(i) There exists an anti-$*$-automorphism $R$ on $M$ such that
\begin{align*}
  R^2&=\id, & R(\alpha(n))&=\beta(n) \text{ for all } n\in N, &
  \Gamma\circ
  R&=\varsigma_{N^{\op}}(R\underset{N}{_\beta*_\alpha}R)\Gamma\end{align*}
and
\begin{align*}
  R((\id*\omega_{\xi, \eta})(W))=(\id*\omega_{J\eta, J\xi})(W) \quad \text{for all } \xi\in D(_\alpha H, \nu), \eta\in D(H_\gamma, \nu^{\op}).
\end{align*}
This map $R$ will be called the \emph{co-inverse}.

(ii) There exists a one-parameter group $\tau_t$ of automorphisms of $M$ such that
\begin{align*}
R\circ\tau_t&=\tau_t\circ R, &
\tau_t(\alpha(n))&=\alpha(\sigma^\nu_t(n)), & \tau_t(\beta(n))
&=\beta(\sigma^\nu_t(n)), &
  \Gamma\circ\sigma_t^\Phi &=(\tau_t\underset{N}{_\beta*_\alpha}\sigma_t^\Phi)\Gamma
\end{align*}
for all $t\in\R$ and and $n\in N$.
This one-parameter group- will be called the \emph{scaling group}.

(iii) The weight $\nu$ is relatively invariant with respect to $T$ and
$RTR$. Moreover, $R$ and $\tau_t$ are still the co-inverse and the
scaling group of this new measured quantum groupoid, which we shall
denote  by
\[\underline{\gG}=(N, M, \alpha, \beta, \Gamma, T, RTR, \nu), \]
and  for simplification we shall assume now that $T'=RTR$ and $\Psi=\Phi\circ R$. 

(iv) There exists a one-parameter group $\gamma_t$ of automorphisms of $N$ such that
\[\sigma_t^{T}(\beta(n))=\beta(\gamma_t(n))\]
 for all $t\in\R$ and $n\in N$.
Moreover, we get that $\nu\circ\gamma_t=\nu$. 

(vi) There exist a positive non-singular operator $\lambda$ affiliated to $Z(M)$ and a positive non-singular operator $\delta$ affiliated with $M$ such that
\[(D\Phi\circ R: D\Phi)_t=\lambda^{it^2/2}\delta^{it},\]
and therefore
\[(D\Phi\circ\sigma_s^{\Phi\circ R}:D\Phi)_t=\lambda^{ist}.\]
The operator $\lambda$ will be called the \emph{scaling operator}, and there exists a positive non-singular operator $q$ affiliated to $N$ such that $\lambda=\alpha(q)=\beta(q)$. We have $R(\lambda)=\lambda$. 

The operator $\delta$ will be called the \emph{modulus}. We have $R(\delta)=\delta^{-1}$ and $\tau_t(\delta)=\delta$ for all $t\in\R$, and we can define a one-parameter group of unitaries $\delta^{it}\underset{N}{_\beta\otimes_\alpha}\delta^{it}$ which acts naturally on elementary tensor products and satisfies for all $t\in\R$
\[\Gamma(\delta^{it})=\delta^{it}\underset{N}{_\beta\otimes_\alpha}\delta^{it}.\]

(vii) We have $(D\Phi\circ\tau_t: D\Phi)_s=\lambda^{-ist}$, which
proves that $\tau_t\circ\sigma_s^\Phi=\sigma_s^\Phi\circ\tau_t$  for all $s$, $t$ in $\R$ and allows to define a one-parameter group of unitaries by
\[P^{it}\Lambda_\Phi(x)=\lambda^{t/2}\Lambda_\Phi(\tau_t(x)) \quad
\text{for all }  x\in\gN_\Phi.\]
Moreover, for any $y$ in $M$, we get that
\[\tau_t(y)=P^{it}yP^{-it}.\]
 As for the multiplicative unitary associated to a locally compact quantum group, one can prove, using this operator $P$, a ``managing property'' for $W$, and we shall say that the pseudo-multiplicative unitary $W$ is \emph{manageable}, with ``managing operator'' $P$.

As $\tau_t\circ\sigma_t^\Phi=\sigma_t^\Phi\circ\tau_t$, we get that $J_\Phi PJ_\Phi=P$. 

(viii) It is possible to construct a \emph{dual} measured quantum groupoid 
\[\widehat{\gG}=(N, \widehat{M}, \alpha, \gamma, \widehat{\Gamma}, \widehat{T}, \widehat{T'}, \nu)\]
where $\widehat{M}$ is equal to the weakly closed linear space
generated by all operators of the form $(\omega_{\xi, \eta}*\id)(W)$,
for $\xi\in D( H_\beta, \nu^{\op})$ and $\eta\in D(_\alpha H, \nu)$,
 $\widehat{\Gamma}(y)=\sigma_{\nu^{\op}}
W(y\underset{N}{_\beta\otimes_\alpha}1)W^*\sigma_\nu$ for all
$y\in\widehat{M}$, and the dual left operator-valued weight $\widehat{T}$ is constructed in a similar way as the dual left-invariant weight of a locally compact quantum group. Namely, it is possible to construct a normal semi-finite faithful weight $\widehat{\Phi}$ on $\widehat{M}$ such that, for all $\xi\in D(H_\beta, \nu^{\op})$ and $\eta\in D(_\alpha H, \nu)$ such that $\omega_{\xi, \eta}$ belongs to $I_\Phi$, 
\[\widehat{\Phi}((\omega_{\xi, \eta}*\id)(W)^*(\omega_{\xi, \eta}*\id)(W))=\|\omega_{\xi, \eta}\|_\Phi^2.\]
We can prove  that
$\sigma_t^{\widehat{\Phi}}\circ\alpha=\alpha\circ\sigma_t^\nu$  for all $t\in\R$, which gives the existence of an operator-valued weight $\widehat{T}$, which appears then to be left-invariant. 

As the formula $y\mapsto J y^*J$ ($y\in \widehat{M}$) gives a
co-inverse for the coproduct $\widehat{\Gamma}$, we get also a
right-invariant operator-valued weight.  Moreover, the
pseudo-multiplicative unitary $\widehat{W}$ associated to
$\widehat{\gG}$ is $\widehat{W}=\sigma_\nu W^*\sigma_\nu$, its
managing operator $\widehat{P}$ is equal to $P$,  its scaling group
is given by $\widehat{\tau}_t(y)=P^{it}yP^{-it}$, its scaling operator
$\widehat{\lambda}$ is equal to $\lambda^{-1}$, and its one-parameter
group of unitaries $\widehat{\gamma}_t$ of $N$ is equal to
$\gamma_{-t}$.

We write $\widehat{\Phi}$ for $\nu\circ\alpha^{-1}\circ\widehat{T}$,
identify $H_{\widehat{\Phi}}$ with $H$, and write
$\widehat{J}=J_{\widehat{\Phi}}$. Then
$R(x)=\widehat{J}x^*\widehat{J}$  for all $x\in M$ and $W^*=(\widehat{J}\underset{N^{\op}}{_\alpha\otimes_\gamma}J)W(\widehat{J}\underset{N^{\op}}{_\alpha\otimes_\gamma}J)$. 

Moreover, we have $\widehat{\widehat{\gG}}=\gG$. 

For example, let $\mathcal G$ be a measured groupoid as in \ref{defMQG}. The dual $\widehat{\mathcal G}$ of the measured quantum groupoid constructed in \ref{defMQG} (and denoted again by $\mathcal G$) is
\[\widehat{\mathcal G}=(L^\infty(\mathcal G^{(0)}, \nu), \mathcal L(\mathcal G), r_{\mathcal G}, r_{\mathcal G}, \widehat{\Gamma}_{\mathcal G}, \widehat{T}_{\mathcal G}, \widehat{T}_{\mathcal G}),\]
where $\mathcal L(\mathcal G)$ is the von Neumann algebra generated by
the convolution algebra associated to the groupoid $\mathcal G$, the
coproduct $\widehat{\Gamma}_{\mathcal G}$ had been defined in
(\cite{Val1} 3.3.2), and the operator-valued weight $\widehat{T}_{\mathcal G}$ had been defined in (\cite{Val1}, 3.3.4). The underlying Hopf bimodule is co-commutative. 

\begin{theorem}[\cite{Ti2}]
\label{th3}
 Let $\bfG$ be a locally compact quantum group and  $(N, \ga,
  \widehat{\ga})$ a braided-commutative $\bfG$-Yetter-Drinfel'd
  algebra. Then the normal faithful semi-finite operator-valued weight
  $T_{\tilde{\ga}}$ from $\bfG\ltimes_\ga A$ onto $\ga(N)$
  (\cite{V}1.3 and 2.5) is left-invariant with respect to the Hopf
  bimodule structure constructed in \ref{th2}, and $\tilde{R}\circ
  T_{\tilde{\ga}}\circ\tilde{R}$ is right-invariant. 
\end{theorem}
\begin{proof}
For all positive $X$ in $\bfG\ltimes_\ga N$, we find, using \ref{lemmaga}(i) and \ref{th2},
\begin{align*}
(\id\underset{N}{_\beta*_\ga}T_{\tilde{\ga}})\widetilde{\Gamma}(X)
&=
(\id\underset{N}{_\beta*_\ga}(\widehat{\varphi}\circ \widehat{R}\otimes \id)\tilde{\ga})\widetilde{\Gamma}(X)\\
&=
(\widehat{\varphi}\circ \widehat{R}\otimes \id)(\id\otimes\widetilde{\Gamma})\tilde{\ga}(X)\\
&=
\widetilde{\Gamma}(T_{\tilde{\ga}}(X))\\
&=
T_{\tilde{\ga}}(X)\underset{N}{_\beta\otimes _\ga}(1_H\otimes 1_{H_\nu})
\end{align*}
which proves that $T_{\tilde{\ga}}$ is left-invariant. Using \ref{th2}, we get trivially that $\widetilde{R}\circ T_{\tilde{\ga}}\circ\widetilde{R}$ is a normal faithful semi-finite operator valued weight from $\bfG\ltimes_\ga N$ onto $\beta(N)$, which is right-invariant with respect to the coproduct $\widetilde{\Gamma}$. 
\end{proof}

In the situation above, we shall denote by $\mathfrak{G}(N, \ga, \widehat{\ga}, \nu)$ the Hopf-bimodule $(N, \bfG\ltimes_\ga N, \ga, \beta, \widetilde{\Gamma})$ constructed in \ref{th1}(ii), equipped with its co-inverse $\widetilde{R}$ constructed in \ref{th2}(ii), with the left-invariant operator-valued weight $T_{\tilde{\ga}}$ and the right-invariant operator-vlaued weight $\widetilde{R}\circ T_{\tilde{\ga}}\circ\widetilde{R}$, and with the normal semi-finite faithful weight $\nu$ on $N$. 
\begin{proposition}
\label{scaling}
 Let $\bfG$ be a locally compact quantum group, $(N, \ga, \widehat{\ga})$ a braided-commu-tative $\bfG$-Yetter-Drinfel'd algebra, $\nu$ a normal semi-finite faithful weight on $N$, $D_t$ its Radon-Nikodym derivative with respect to $\ga$ (\ref{action}) and $D_t^{\op}$ the Radon-Nikodym derivative of the weight $\nu^{\op}$ on $N^{\op}$ with respect to the action $\ga^{\op}$ (\ref{BCdef1}). For all $t\in\R$,  denote by $\widetilde{\tau}_t$ the map $\Ad [U^\ga_\nu(U^{\widehat{\ga}}_\nu)^*\Delta_{\widetilde{\hat{\nu}}}^{it}U^{\widehat\ga}_\nu(U^\ga_\nu)^*]$ defined on $B(H\otimes H_\nu)$, where $\widetilde{\hat{\nu}}$ is the dual weight of $\nu$ on the crossed product $\widehat{\bfG}\ltimes_{\widehat{\ga}}N$. Then:
 \begin{enumerate}
 \item  $\widetilde{\tau}_t\circ\beta(x)=\beta(\sigma_t^\nu(x))$
   for all $x\in N$ and $t\in \R$.
 \item  for all $t\in\R$, $\widetilde{\tau}_t$ commutes with $\Ad I$,
   where $I$ had been defined in \ref{propI}, and, therefore
   $\widetilde{\tau}_t(\ga(x))=\ga(\sigma_t^\nu(x))$ for all $x\in N$
   and $t\in \R$.
 \item  Denote by $\beta^\dag$ the application $x^{\op}\mapsto
   \beta(x)$ from $N^{\op}$ into $\bfG\ltimes_\ga N$.  Then
   \begin{align*}
     (\id\otimes\widetilde{\tau}_t)(W_{12})&=\widehat{\Delta}^{-it}_1(\id\otimes\beta^\dag)(D^{\op}_{-t})W_{12}(\id\otimes\ga)(D_t)\widehat{\Delta}^{it}_1
     \\ &= (\tau_{-t} \otimes \beta^{\dag})(D^{\op}_{-t})(\id \otimes
     \widehat\tau)(W)_{12} (\tau_{-t} \otimes \ga)(D_{t}).
   \end{align*}
 \item  $\widetilde{\tau}_t(\bfG\ltimes_\ga N)=\bfG\ltimes_\ga N$ and
   $\widetilde{\tau}_t\circ\widetilde{R}=\widetilde{R}\circ\widetilde{\tau}_t$.
 \end{enumerate}
\end{proposition}
\begin{proof}
(i) For any $x\in N$, 
\begin{align*}
\widetilde{\tau}_t(\beta(x))
&=
\Ad
(U^\ga_\nu(U^{\widehat{\ga}}_\nu)^*)[\Delta_{\widetilde{\hat{\nu}}}^{it}] \cdot
\Ad(U^\ga_\nu(U^{\widehat{\ga}}_\nu)^*)[1\otimes x^{\op}] \cdot \Ad(U^\ga_\nu(U^{\widehat{\ga}}_\nu)^*)[\Delta_{\widetilde{\hat{\nu}}}^{-it}]\\
&=
\Ad (U^\ga_\nu(U^{\widehat{\ga}}_\nu)^*\Delta_{\widetilde{\hat{\nu}}}^{it})[1\otimes x^{\op}]\\
&=
\Ad (U^\ga_\nu(U^{\widehat{\ga}}_\nu)^*)[D_t(1\otimes\sigma_t^\nu(x)^{\op})D_t^*]\\
&=
\Ad (U^\ga_\nu(U^{\widehat{\ga}}_\nu)^*)[1\otimes\sigma_t^\nu(x)^{\op}]\\
&=
\beta(\sigma_t^\nu(x))
\end{align*}

(ii) The first assertion follows from the fact that
$J_{\widetilde{\hat{\nu}}}$ and $\Delta_{\widetilde{\hat{\nu}}}^{it}$
commute. To conclude that
$\widetilde{\tau}_t(\ga(x))=\ga(\sigma_t^\nu(x))$,  use (i) and \ref{propI}(ii).

(iii)
Let $t\in\R$. Then \ref{lemBC}(iii) and
\ref{lcqg} imply
\begin{align*}
\Ad((\widehat{\Delta} \otimes \Delta_{\widetilde{\hat\nu}})^{it}) [(U^{\ga}_{\nu})_{13}^{*}W_{12}] &=
\Ad((\widehat{D}_{t})_{23}(\widehat{\Delta} \otimes \Delta \otimes
  \Delta_{\nu})^{it}) [(U^{\ga}_{\nu})^{*}_{13}W_{12}] \\
 &=(\widehat{D}_{t})_{23}(D^{\op}_{-t})_{13}^*(U^\ga_\nu)_{13}^*(D_t)_{13}W_{12} (\widehat{D}_{t})^{*}_{23}
 \\
&= (D^{\op}_{-t})_{13}^*(\widehat{D}_{t})_{23}(U^\ga_\nu)_{13}^*(D_t)_{13}W_{12}(\widehat{D}_{t})^{*}_{23}.
\end{align*}
But 
 \ref{propRN} gives that
 $(\id \otimes \widehat{\ga})(D_{t})( \widehat{D}_t)_{23}=
W^*_{12}(U^\ga_\nu)_{13}(\widehat{D}_t)_{23}(U^\ga_\nu)^*_{13}(D_t)_{13}W_{12}$, whence
\[(\widehat{D}_t)_{23}(U^\ga_\nu)^*_{13}(D_t)_{13}W_{12}(\widehat{D}_{t})_{23}^{*}=(U^\ga_\nu)^*_{13}W_{12}(1
\otimes \widehat{\ga})(D_{t}).\]
We insert this relation above and find
\begin{align*}
\Ad((\widehat{\Delta} \otimes \Delta_{\widetilde{\hat\nu}})^{it})
[(U^{\ga}_{\nu})_{13}^{*}W_{12}] &=
(D^{\op}_{-t})_{13}^*  \cdot (U^{\ga}_{\nu})_{13}^{*}W_{12}  \cdot
(\id \otimes \widehat\ga)(D_{t}).
\end{align*}
We use
this relation and $\Ad(1 \otimes
U^{\widehat\ga}_\nu(U^\ga_\nu)^*)[W_{12}]=(U^{\ga}_{\nu})_{13}^{*}W_{12}$
(\ref{corintegrable3}), and find
\begin{align*}
  (\id \otimes \widetilde\tau_{t})(W_{12}) &=
  \Ad(1_{H} \otimes
  U^\ga_\nu(U^{\widehat{\ga}}_\nu)^*\Delta_{\widetilde{\hat{\nu}}}^{it}U^{\widehat\ga}_\nu(U^\ga_\nu)^*)[W_{12}]
  \\
&= \Ad(\widehat{\Delta}^{-it} \otimes U^\ga_\nu(U^{\widehat{\ga}}_\nu)^*)[\Ad((\widehat{\Delta} \otimes \Delta_{\widetilde{\hat\nu}})^{it})
((U^{\ga}_{\nu})_{13}^{*}W_{12})] \\
&=\Ad(\widehat{\Delta}^{-it} \otimes
U^\ga_\nu(U^{\widehat{\ga}}_\nu)^*) 
[(D^{\op}_{-t})_{13}^*  \cdot (U^{\ga}_{\nu})_{13}^{*}W_{12}  \cdot
(\id \otimes \widehat\ga)(D_{t})] \\
&= \widehat{\Delta}^{-it}_1(\id\otimes\beta^\dag)(D^{\op}_{-t})W_{12}(\id\otimes\ga)(D_t)\widehat{\Delta}^{it}_1.
\end{align*}

(iv) For any $\omega\in M_*$, the element  $\widetilde{\tau}_t [(\omega\otimes \id)(W)\otimes
1]$ belongs to $\bfG\ltimes_\ga N$ because
\[\widetilde{\tau}_t [(\omega\otimes \id)(W)\otimes 1]=(\omega\circ\tau_{-t})[(\id\otimes\beta^\dag)(D^{\op}_{-t})W_{12}(\id\otimes\ga)(D_t)].\]
By continuity, we get that
$\widetilde{\tau}_t(y\otimes 1)$ belongs to $\bfG\ltimes_\ga N$ for any $y\in\widehat{M}$.  Together
with (ii), we obtain that $\widetilde{\tau}_t(\bfG\ltimes_\ga
N)\subseteq \bfG\ltimes_\ga N$, and, as $\widetilde\tau$ is a one-parameter group of
automorphisms, we have  $\widetilde{\tau}_t(\bfG\ltimes_\ga
N)=\bfG\ltimes_\ga N$. By (ii), $\widetilde \tau_{t}$ commutes with
$\widetilde R$. \end{proof}

\begin{lemma}
\label{sigmaW}
 Let $\bfG$ be a locally compact quantum group, $(N, \ga,
  \widehat{\ga})$ a braided-commutative $\bfG$-Yetter-Drinfel'd
  algebra, $\nu$ a normal faithful semi-finite weight on $N$, $D_t$
  its Radon-Nikodym derivative with respect to $\ga$ (\ref{action}) and $\tilde{\nu}$  the dual weight of $\nu$ on the crossed product $\bfG\ltimes_{\ga}N$. Then for all $t\in\R$,
\[(\id\otimes
\sigma_t^{\tilde{\nu}})(W_{12})=\delta_1^{-it}\widehat{\Delta}_1^{-it}W_{12}(\id\otimes\ga)(D_t)\widehat{\Delta}^{it}_1
= (\id \otimes \widehat\sigma_{t})(W)_{12}(\tau_{-t}\otimes \ga)(D_{t}).\]
\end{lemma}
\begin{proof}
By (\cite{Y3}, 3.4) and \ref{action},
\begin{align*}
(\id\otimes \sigma_t^{\tilde{\nu}})(W_{12})
&=
[D_t(\widehat{\Delta}^{it}\otimes\Delta_\nu^{it})]_{23}W_{12}[(\widehat{\Delta}^{-it}\otimes\Delta_\nu^{-it})D_t^*]_{23}\\
&=
\delta_1^{-it}\widehat{\Delta}^{-it}_1(D_t)_{23}W_{12}\widehat{\Delta}^{it}_1(D_t^*)_{23}\\
&=
\delta_1^{-it}\widehat{\Delta}^{-it}_1W_{12}(\Gamma\otimes \id)(D_t)(D_t^*)_{23}\widehat{\Delta}^{it}_1\\
&=
\delta_1^{-it}\widehat{\Delta}^{-it}_1W_{12}(\id\otimes \ga)(D_t)\widehat{\Delta}^{it}_1.\qedhere
\end{align*}
\end{proof}

\begin{proposition}
\label{scalingGamma}
 Let $\bfG$ be a locally compact quantum group, $(N, \ga,
  \widehat{\ga})$ a braided-commu-tative $\bfG$-Yetter-Drinfel'd
  algebra, $\nu$ a normal faithful semi-finite weight on $N$, and
  $\tilde{\nu}$  the dual weight of $\nu$ on the crossed product
  $\bfG\ltimes_{\ga}N$. Then the one-parameter group $\widetilde{\tau}_t$
  of $\bfG\ltimes_\ga N$ constructed in \ref{scaling} satisfies, for
  all $t\in\R$,
  \begin{align*}
    \widetilde{\Gamma}\circ\sigma_t^{\tilde{\nu}}&=(\widetilde{\tau}_t\underset{N}{_\beta*_\ga}\sigma_t^{\tilde{\nu}})\circ\widetilde{\Gamma}, &
\widetilde{\Gamma}\circ\sigma_t^{\tilde{\nu}\circ\widetilde{R}}&=(\sigma_t^{\tilde{\nu}\circ\widetilde{R}}\underset{N}{_\beta*_\ga}\widetilde{\tau}_{-t})\circ\widetilde{\Gamma}.
  \end{align*}
\end{proposition}
\begin{proof}
Let $x\in N$ and $t\in\R$. Then  \ref{scaling}(ii) and \ref{th1} imply
\begin{align*}
  \widetilde{\Gamma}\circ\sigma_t^{\tilde{\nu}}(\ga(x))=\widetilde{\Gamma}(\ga(\sigma_t^\nu(x)))&=\ga(\sigma_t^\nu(x))\underset{N}{_\beta\otimes_\ga}1\\
  &=(\widetilde{\tau}_t\underset{N}{_\beta*_\ga}\sigma_t^{\tilde{\nu}})(\ga(x)\underset{N}{_\beta\otimes_\ga}1)=
  (\widetilde{\tau}_t\underset{N}{_\beta*_\ga}\sigma_t^{\tilde{\nu}})\widetilde{\Gamma}(\ga(x)).
\end{align*}
Next, let $V_2$ be the unitary from $(H\otimes
H_\nu)\underset{\nu}{_\beta\otimes_\ga}(H\otimes H_\nu)$ onto
$H\otimes H\otimes H_\nu$ introduced in \ref{propV}, and denote by
$\widetilde{\hat\nu}$ the weight on $\widehat\bfG\ltimes_{\ga} N$ dual
to $\nu$ as before. Then
\begin{align*}
  V_{2}[U^\ga_\nu(U^{\widehat{\ga}}_\nu)^*\Delta_{\widetilde{\hat{\nu}}}^{it}U^{\widehat{\ga}}(U^\ga_\nu)^*\underset{N}{_\beta\otimes_\ga}\Delta_{\tilde{\nu}}^{it}]V_2^*
  (\xi \otimes \ga(q)\Xi) &=
  V_{2}[U^\ga_\nu(U^{\widehat{\ga}}_\nu)^*\Delta_{\widetilde{\hat{\nu}}}^{it}(\xi
  \otimes \Lambda_{\nu}(q)) 
\underset{N}{_\beta\otimes_\ga}\Delta_{\tilde{\nu}}^{it}\Xi] \\
&=
  V_{2}[U^\ga_\nu(U^{\widehat{\ga}}_\nu)^*\widehat D_{t}(\Delta^{it}\xi
  \otimes \Lambda_{\nu}(\sigma^{\nu}_{t}(q))) 
\underset{N}{_\beta\otimes_\ga}\Delta_{\tilde{\nu}}^{it}\Xi] \\
&=  (\id \otimes \ga)(\widehat D_{t})(\Delta^{it}\xi \otimes
\ga(\sigma^{\nu}_{t}(q))\Delta^{it}_{\tilde \nu}\Xi) \\
&=  (\id \otimes \ga)(\widehat D_{t})(\Delta^{it} \otimes
\Delta^{it}_{\tilde \nu})(\xi \otimes \ga(q)\Xi).
\end{align*}
 Let now $y \in \widehat M$. Then by \ref{th1},
\begin{align*}
  \Ad(V_{2})[\tilde
\Gamma(y\otimes 1)]=\widehat\Gamma(y) \otimes 1 = \Ad(\sigma_{12}W_{12})[y
\otimes 1].
\end{align*}
Using these two relations and \ref{propRN}, we find
\begin{align*}
  \Ad(V_{2})[(\tilde \tau \underset{N}{_{\beta}*_{\alpha}} 
\sigma^{\tilde\nu}_{t})(\tilde \Gamma(y\otimes 1))] &= \Ad((\id
\otimes \ga)(\widehat D_{t})(\Delta^{it} \otimes
\Delta^{it}_{\tilde{\nu}}) \sigma_{12}W_{12}
 )[y\otimes 1\otimes 1] \\
 &= \Ad(\sigma_{12}W_{12}(\id \otimes \widehat\ga)(D_{t})(\widehat
 \Delta^{it} \otimes \Delta^{it}_{\widetilde{\hat \nu}}))[y\otimes 1
 \otimes 1] \\
 &=
 \Ad(\sigma_{12}W_{12}(U^{\widehat{\ga}}_{\nu})_{23}(D_{t})_{13})[\widehat
 \sigma_{t}(y) \otimes 1 \otimes 1].
\end{align*}
By \ref{propV}(iii),
$\sigma_{12}W_{12}(U^{\widehat{\ga}}_{\nu})_{23}
=V_{2}V_{1}^{*}W_{12}(U^{\ga}_{\nu})_{23}$ and hence
\begin{align*}
  \Ad(V_{1})[(\tilde \tau \underset{N}{_{\beta}*_{\alpha}} 
\sigma^{\tilde\nu}_{t})(\tilde \Gamma(y\otimes 1))] &=
\Ad(W_{12}(U^{\ga}_{\nu})_{23}  (D_{t})_{13})[\widehat
 \sigma_{t}(y) \otimes 1 \otimes 1] \\
 &=\Ad(W_{12}(\id \otimes \ga)(D_{t})(D_{t})_{23})[\widehat
 \sigma_{t}(y) \otimes 1 \otimes 1] \\
 &=\Ad((D_{t})_{23}W_{12})[\widehat
 \sigma_{t}(y) \otimes 1 \otimes 1] \\
&= \Ad((D_{t})_{23})[\widehat \Gamma^{\op}(\widehat \sigma_{t}(y))\otimes 1].
\end{align*}
On the other hand, 
\begin{align*}
  \Ad(V_{1})[\tilde \Gamma(\sigma^{\tilde\nu}_{t}(y \otimes 1))] 
  &= \tilde\ga(\sigma^{\tilde \nu}_{t}(y\otimes 1)) \\
  &= \Ad((\widehat W^{\op}_{12})^{*})[\sigma^{\tilde \nu}_{t}(y\otimes 1)] \\
  &=\Ad((\widehat W^{\op}_{12})^{*}(D_{t})_{23})[\widehat\sigma_{t}(y)\otimes 1]
  \\
  &=\Ad((D_{t})_{23}(\widehat W^{\op})_{12}^{*})[\widehat\sigma_{t}(y)\otimes 1] \\
  &=\Ad((D_{t})_{23})[(\widehat \Gamma^{\op}(\widehat \sigma_{t}(y))\otimes 1)],
\end{align*}
showing that
$(\tilde \tau \underset{N}{_{\beta}*_{\alpha}} 
\sigma^{\tilde\nu}_{t})(\tilde \Gamma(y\otimes 1))=  \tilde \Gamma(\sigma^{\tilde\nu}_{t}(y \otimes 1))$.

Since $\bfG\ltimes_{\ga} N$ is generated by $\ga(N)$ and $\widehat{M}
\otimes 1$, the first of the two formulas follows.
Using \ref{scaling}(iv),  the second one is easy to prove from the first one. \end{proof}

\begin{corollary}
\label{corgamma}
 Let $\bfG$ be a locally compact quantum group, $(N, \ga,
  \widehat{\ga})$ a braided-commu-tative $\bfG$-Yetter-Drinfel'd
  algebra, $\nu$ a normal faithful semi-finite weight on $N$, and
  $\tilde{\nu}$  the dual weight of $\nu$ on the crossed product
  $\bfG\ltimes_{\ga}N$. Then there exists a one-parameter group
  $\gamma_t$ of automorphisms of $N$ such that
  $\sigma_t^{\tilde{\nu}}(\beta(x))=\beta(\gamma_t(x))$.
\end{corollary}
\begin{proof}
Using \ref{scalingGamma}, we get that for all $x\in N$ and $t\in \R$,
\[\widetilde{\Gamma}(\sigma_t^{\tilde{\nu}}(\beta(x)))=(\widetilde{\tau}_t\underset{N}{_\beta*_\ga}\sigma_t^{\tilde{\nu}})(\widetilde{\Gamma}(\beta(x)))=(\widetilde{\tau}_t\underset{N}{_\beta*_\ga}\sigma_t^{\tilde{\nu}})(
1\underset{N}{_\beta\otimes_\ga}\beta(x))=1\underset{N}{_\beta\otimes_\ga}\sigma_t^{\tilde{\nu}}(\beta(x))\]
from which we get the result by (\cite{L}, 4.0.9). \end{proof}

\begin{theorem}
\label{thmqg}
 Let $\bfG$ be a locally compact quantum group, $(N, \ga,
  \widehat{\ga})$ a braided-commu-tative $\bfG$-Yetter-Drinfel'd
  algebra, $\nu$ a normal faithful semi-finite weight on $N$, $D_t$
  the Radon-Nikodym derivative of $\nu$ with respect to the action
  $\ga$, $\tilde{\nu}$ the dual weight of $\nu$ on the
  crossed product $\bfG\ltimes_\ga N$, $\widetilde{\tau}_t$ the one
  parameter group of automorphisms of $\bfG\ltimes_\ga N$ constructed
  in \ref{scaling}, and $\gamma_t$ the one parameter group of
  automorphisms of $N$ constructed in \ref{corgamma}. Let $\Phi_t$ be
  the automorphism of $M$ defined by $\Phi_t(x)=\tau_t\circ\Ad
  \delta^{-it}$ (let us remark that $\Phi_t$ is an automorphism of
  $\bfG$). Then the following conditions are equivalent:
  \begin{enumerate}
  \item  $(\Phi_t\otimes\gamma_t)(D_s)=D_s$ for all $s$, $t$ in
    $\R$.
  \item  $\sigma_t^{\tilde{\nu}}$ and $\widetilde{\tau}_s$ commute for
    all $s$, $t$ in $\R$.
  \item  $\sigma_t^{\tilde{\nu}}$ and
    $\sigma_s^{\tilde{\nu}\circ\widetilde{R}}$ commute for all $s$,
    $t$ in $\R$.
  \item  $\mathfrak{G}(N, \ga, \widehat{\ga}, \nu)$ is a measured
    quantum groupoid.
  \end{enumerate}
If these conditions hold, then $\widetilde{\tau}_t$ is the  scaling
group of  $\mathfrak{G}(N, \ga, \widehat{\ga}, \nu)$, and $\gamma_t$
is the one parameter group of automorphisms of $N$ defined in
\ref{data}(iv).
\end{theorem}
\begin{proof}
  The restrictions of $\sigma_t^{\tilde{\nu}}$ and
  $\widetilde{\tau}_s$ on $\ga(N)$ always commute because
  $\sigma_t^{\tilde{\nu}}\circ\widetilde{\tau}_s(\ga(x))=\ga(\sigma_t^\nu\circ\sigma_s^\nu(x))$
  and
  $\widetilde{\tau}_s\circ\sigma_t^{\tilde{\nu}}(\ga(x))=\ga(\sigma_s^\nu\circ\sigma_t^\nu(x))$
  for all $x\in N$ by \ref{scaling}(ii).

Using now \ref{sigmaW}, \ref{scaling}(iii) and \ref{action}, we get that
\begin{align*}
(\id\otimes\widetilde{\tau}_s\sigma_t^{\tilde{\nu}})(W_{12})
&=
\delta_1^{-it}\widehat{\Delta}^{-it}_{1}(\id\otimes\widetilde{\tau}_s)(W_{12})(\id\otimes\widetilde{\tau}_s\ga)(D_t)\widehat{\Delta}_{1}^{it}\\
&=
\delta_1^{-it}\widehat{\Delta}_{1}^{-it}\widehat{\Delta}_{1}^{-is}(\id\otimes\beta^\dag)(D^{\op}_{-s})W_{12}(\id\otimes\ga)(D_s)\widehat{\Delta}_{1}^{is}(\id\otimes\ga\sigma_s^\nu)(D_t)\widehat{\Delta}_{1}^{it}\\
&=
\delta_1^{-it}\widehat{\Delta}_{1}^{-i(s+t)}(\id\otimes\beta^\dag)(D^{\op}_{-s})W_{12}(\id\otimes\ga)(D_s(\tau_s\otimes\sigma_s^\nu)(D_t))\widehat{\Delta}_1^{i(s+t)}\\
&=
\delta_1^{-it}\widehat{\Delta}^{-i(s+t)}(\id\otimes\beta^\dag)(D^{\op}_{-s})W_{12}(\id\otimes\ga)(D_{s+t})\widehat{\Delta}_1^{i(s+t)}
\end{align*}
and, on the other hand,
\begin{align*}
(\id\otimes\sigma_t^{\tilde{\nu}}\widetilde{\tau}_s)(W_{12})
&=
\widehat{\Delta}_1^{-is}(\id\otimes\sigma_t^{\widetilde\nu}\beta^\dag)(D^{\op}_{-s})(\id\otimes\sigma_t^{\tilde{\nu}})(W_{12})(\id\otimes\sigma_t^{\widetilde\nu}\ga)(D_s)\widehat{\Delta}_1^{is}\\
&=
\widehat{\Delta}_1^{-is}(\id\otimes\beta^\dag\gamma_t^{\op})(D^{\op}_{-s})\delta_1^{-it}\widehat{\Delta}_{1}^{-it}W_{12}(
\id \otimes \ga)(D_{t})\widehat{\Delta}_{1}^{it}(\id\otimes\sigma_t^{\tilde{\nu}}\ga)(D_s)\widehat{\Delta}_{1}^{is}\\
&=
\widehat{\Delta}_{1}^{-i(s+t)}\delta_1^{-it}(\Phi_t\otimes\beta^\dag\gamma_t^{\op})(D^{\op}_{-s})W_{12}(\id\otimes\ga)(D_t(\tau_t\otimes\sigma_t^\nu)(D_s))\widehat{\Delta}_{1}^{i(s+t)}\\
&=
\widehat{\Delta}_{1}^{-i(s+t)}\delta_1^{-it}(\Phi_t\otimes\beta^\dag\gamma_t^{\op})(D^{\op}_{-s})W_{12}(\id\otimes\ga)(D_{s+t})\widehat{\Delta}_{1}^{i(s+t)}.
\end{align*}
Consequently,
$(\id\otimes\sigma_t^{\tilde{\nu}}\widetilde{\tau}_s)(W_{12})=(\id\otimes\widetilde{\tau}_s\sigma_t^{\tilde{\nu}})(W_{12})$
if and only if $(\Phi_t\otimes\gamma_t)(D_s)=D_s$, which gives the equivalence of (i) and (ii). 

Let us suppose (ii).  Using \ref{scalingGamma}, we get 
\begin{align*}
  \widetilde{\Gamma}(\sigma_t^{\tilde{\nu}}\sigma_s^{\tilde{\nu}\circ\widetilde{R}})&=(\widetilde{\tau}_t\sigma_s^{\tilde{\nu}\circ\widetilde{R}}\underset{N}{_\beta*_\ga}\sigma_t^{\tilde{\nu}}\widetilde{\tau}_{-s})\circ\widetilde{\Gamma}
  &&\text{and} &
  \widetilde{\Gamma}(\sigma_s^{\tilde{\nu}\circ\widetilde{R}}\sigma_t^{\tilde{\nu}})&=(\sigma_s^{\tilde{\nu}\circ\widetilde{R}}\widetilde{\tau}_t\underset{N}{_\beta*_\ga}\widetilde{\tau}_{-s}\sigma_t^{\tilde{\nu}})\circ\widetilde{\Gamma},
\end{align*}
and by the commutation of $\widetilde{\tau}$ with $\sigma^{\tilde{\nu}}$ and with $\sigma^{\tilde{\nu}\circ\widetilde{R}}$, we get (iii). 

By definition of a measured quantum groupoid, we have the equivalence of (iii) and (iv).  The fact that (iv) implies (ii) is given by \ref{data}(vi). \end{proof}

\begin{corollary}
\label{cormqg}
 Let $\bfG$ be a locally compact quantum group and $(N, \ga,
  \widehat{\ga})$ a braided-commutative $\bfG$-Yetter-Drinfel'd
  algebra such that one of the following conditions holds: 
  \begin{enumerate}
  \item  $N$ is properly infinite, or
  \item  $\ga$ is integrable, or
  \item  $\bfG$ is (the von Neumann version of) a compact quantum
    group.
  \end{enumerate}
Then there exists a normal semi-finite faithful weight $\nu$ on $N$ such that $\mathfrak{G}(N, \ga, \widehat{\ga}, \nu)$ is a measured quantum groupoid.\end{corollary}
\begin{proof}
We consider the individual cases:

(i) By \ref{corint}, there exists a normal semi-finite faithful weight $\nu$ on $N$, invariant under $\ga$; therefore its Radon-Nikodym derivative $D_t=1$, and we get the result by \ref{thmqg}. 

(ii) In that case, there exists a weight $\nu$ on $N$ which is $\delta^{-1}$-invariant with respect to $\ga$; so we can apply again \ref{thmqg} to get the result. 

(iii) We are here in a particular case of (ii), but with $\delta=1$. \end{proof}

\begin{proposition}
\label{tau}
 Let $\bfG$ be a locally compact quantum group,  $(N, \ga,
  \widehat{\ga})$ a braided-commu\-tative $\bfG$-Yetter-Drinfel'd
  algebra, $\nu$ a normal faithful semi-finite weight on $N$,
  $k$-invariant with respect to $\ga$ (with $k$ affiliated to
  $Z(M)$). Then:
  \begin{enumerate}
  \item  the scaling group $\widetilde{\tau}_t$ of $\mathfrak{G}(N,
    \ga, \widehat{\ga}, \nu)$ is given by
    $\widetilde{\tau}_t(X)=(P^{it}\otimes
    \Delta_\nu^{it})X(P^{-it}\otimes\Delta_\nu^{-it})$ for all $X\in
    \bfG\ltimes_\ga N$;
  \item  the scaling operator $\widetilde{\lambda}$ is equal to
    $\lambda^{-1}$, where $\lambda$ is the scaling constant of $\bfG$,
    and the managing operator $\widetilde{P}$ is equal to
    $P\otimes\Delta_\nu$.
  \end{enumerate}
\end{proposition}
\begin{proof}
(i)
The scaling group $\widetilde{\tau}_t$ satisfies
$\widetilde{\tau}_t(\ga(x))=\ga(\sigma_t^\nu(x))$
 for all $x\in N$ (\ref{scaling}(ii)). Using now \ref{prop2inv} (i), we get that $\widetilde{\tau}_t(\ga(x))=(\tau_t\otimes\sigma_t^\nu)(\ga(x))$. 

On the other hand, using \ref{scaling}(iii) and \ref{definv}, we get that 
\[(\id\otimes\widetilde{\tau}_t)(W_{12})=\widehat{\Delta}^{-it}_1R(k^{-it})_1W_{12}k^{it}\widehat{\Delta}^{it}_1=(\tau_{-t}\otimes \id)(W)\otimes 1=(\id\otimes\widehat{\tau}_t)(W)\otimes 1.\]
So, for all $y\in\widehat{M}$, we have $\widetilde{\tau}_t(y\otimes 1)=\widehat{\tau}_t(y)\otimes 1$,  from which we get (i). 

(ii)
 The scaling operator is equal to $\lambda^{-1}$ because
\begin{align*}
\tilde{\nu}(\widetilde{\tau}_t(\ga(x^*)(y^*y\otimes 1_{H_\nu})\ga(x)))
&=
\tilde{\nu}[\ga(\sigma_t^\nu(x^*))(\widehat{\tau}_t(y^*y)\otimes 1_{H_\nu})\ga(\sigma_t^\nu(x))]\\
&=
\nu(\sigma_t^\nu(x^*x))\widehat{\varphi}(\widehat{\tau}_t(y^*y))\\
&=
\lambda^{-t}\nu(x^*x)\widehat{\varphi}(y^*y)\\
&=
\lambda^{-t}\tilde{\nu}(\ga(x^*)(y^*y\otimes 1_{H_\nu})\ga(x)),
\end{align*}
and  $\widetilde{P}$ is equal to $P\otimes\Delta_\nu$ because
\begin{align*}
\Lambda_{\tilde{\nu}}(\widetilde{\tau}_t((y\otimes 1_{H_\nu})\ga(x)))
&=
\Lambda_{\tilde{\nu}}[(\widehat{\tau}_t(y)\otimes 1_{H_\nu})\ga(\sigma_t^\nu(x))]\\
&=
\Lambda_{\widehat{\varphi}}(\widehat{\tau}_t(y))\otimes\Lambda_\nu(\sigma_t^\nu(x))\\
&=
\lambda^{t/2}(P^{it}\otimes\Delta_\nu^{it})(\Lambda_{\widehat{\varphi}}(y)\otimes\Lambda_\nu(x)). \qedhere
\end{align*}
 \end{proof}

\section{Duality}
\label{duality}
In this chapter, we prove (\ref{ovw}) that, if  $\mathfrak G(N, \ga, \widehat{\ga}, \nu)$ is a measured quantum groupoid, its dual is isomorphic to $\mathfrak G(N, \widehat{\ga},\ga, \nu)$, which is therefore also a  measured quantum groupoid.

\begin{lemma}
\label{V3}
 Let $\bfG$ be a locally compact quantum group,  $(N, \ga,
  \widehat{\ga})$ a braided-commutative $\bfG$-Yetter-Drinfel'd
  algebra and $\nu$ a normal faithful semi-finite weight on
  $N$, and let $\mathfrak{G}(N, \ga, \widehat{\ga}, \nu)$ be the associated Hopf-bimodule, equipped with a co-inverse, a left-invariant operator-valued weight and a right-invariant valued weight by \ref{th1}(ii), \ref{th2} and \ref{th3}. Then:
  \begin{enumerate}
  \item  The anti-representation $\gamma$ of $N$ is given by
    $\gamma(x^*)=1_H\otimes J_\nu xJ_\nu$ for all $x\in N$, .
  \item  For any $\xi\in H$, $p\in\gN_\nu$, the vector
    $\xi\otimes\Lambda_\nu(p)$ belongs to $D((H\otimes
    H_\nu)_{\gamma}, \nu^{\op})$, and $R^{\gamma,
      \nu^{\op}}(\xi\otimes\Lambda_\nu(p))=l_\xi p$, where $l_\xi$ is
    the linear application from $H_\nu$ to $H\otimes H_\nu$ given by
    $l_\xi\zeta=\xi\otimes\zeta$ for all $\zeta\in H_\nu$.
  \item  There exists a unitary $V_3$ from $(H\otimes
    H_\nu)\underset{\nu^{\op}}{_\ga\otimes_{\gamma}}(H\otimes H_\nu)$
    onto $H\otimes H\otimes H_\nu$ such that
    \[V_3[\Xi\underset{\nu^{\op}}{_\ga\otimes_{\gamma}}(\xi\otimes\Lambda_\nu(p))]=\xi\otimes\ga(p)\Xi
    \quad \text{for all } \Xi \in H \otimes H_{\nu}.\] Moreover,
    $(1\otimes X)V_3=V_3(X\underset{N^{\op}}{_\ga\otimes_{\gamma}}1)$
    for all $X\in\ga(N)'$.

    (iv)
    $V_3(I\underset{N}{_\beta\otimes_\ga}J_{\tilde{\nu}})=(\widehat{J}\otimes
    I)V_1$.
  \end{enumerate}
\end{lemma}
\begin{proof}
(i) By definition (\ref{defW}), the left-invariant weight of $\mathfrak{G}(N, \ga, \widehat{\ga}, \nu)$ is the dual weight $\tilde{\nu}$. Therefore, by definition (\ref{defW}), and using \ref{action},
 \[\gamma(x^*)=J_{\tilde{\nu}}\ga(x)J_{\tilde{\nu}}=(\widehat{J}\otimes J_\nu)(U_\nu^\ga)^*\ga(x)U^\ga_\nu(\widehat{J}\otimes J_\nu)=1_H\otimes J_\nu xJ_\nu.\]

(ii) This follows from the relation
$l_\xi pJ_\nu\Lambda_\nu(x)=\xi\otimes J_\nu xJ_\nu\Lambda_\nu(p)=\gamma(x^*)(\xi\otimes\Lambda_\nu(p))$.

(iii) For any $\xi'\in H$, $\Xi'\in H\otimes H_\nu$, $p'\in\gN_\nu$, 
 \begin{align*}
 (\Xi\underset{\nu^{\op}}{_\ga\otimes_{\gamma}}(\xi\otimes\Lambda_\nu(p))|\Xi'\underset{\nu^{\op}}{_\ga\otimes_{\gamma}}(\xi'\otimes\Lambda_\nu(p'))
 &=
 (\ga(\langle \xi\otimes\Lambda_\nu(p), \xi'\otimes\Lambda_\nu(p)\rangle _{{\gamma}, \nu^{\op}})\Xi|\Xi')\\
 &=
 (\ga(p'^*l_{\xi'}^*l_\xi p)\Xi|\Xi')\\
 &=
 (\xi\otimes\ga(p)\Xi|\xi'\otimes\ga(p')\Xi'),
 \end{align*}
 from which we get the existence of $V_3$ as an isometry. As it is
 trivially surjective, we get it is a unitary. The last formula of
 (iii) is trivial.

(iv) Using \ref{propI}(ii) and \ref{V3}(i), we get the existence of an
 anti-linear bijective isometry
 $I\underset{N}{_\beta\otimes_\ga}J_{\tilde{\nu}}$ from $(H\otimes
 H_\nu)\underset{\nu}{_\beta\otimes_\ga}(H\otimes H_\nu)$ onto
 $(H\otimes H_\nu)\underset{\nu^{\op}}{_\ga\otimes_\gamma}(H\otimes
 H_\nu)$ with trivial values on elementary tensors. Moreover, for any
 $\Xi\in H\otimes H_\nu$, $\xi\in H$, $p\in \gN_\nu$, analytic with
 respect to $\nu$, we have, using successively \ref{action}, (iii),
 and \ref{propV}(i),
 \begin{align*}
 V_3(I\underset{N}{_\beta\otimes_\ga}J_{\tilde{\nu}})(\Xi\underset{\nu}{_\beta\otimes_\ga}U^\ga_\nu((\xi\otimes\Lambda_\nu(p)))
 &=
 V_3[I\Xi\underset{\nu}{_\ga\otimes_\gamma}(\widehat{J}\xi\otimes J_\nu\Lambda_\nu(p))]\\
 &=
 \widehat{J}\xi\otimes\ga(\sigma_{-i/2}^\nu(p^*)))I\Xi\\
 &=
 \widehat{J}\xi\otimes I\beta(\sigma_{i/2}^\nu(p))\Xi\\
 &=
 (\widehat{J}\otimes
 I)V_1[\Xi\underset{\nu}{_\beta\otimes_\ga}U^\ga_\nu(\xi\otimes\Lambda_\nu(p))]. \qedhere
 \end{align*}
 \end{proof}

\begin{theorem}[\cite{Ti2}]
\label{W}
 Let $\bfG$ be a locally compact quantum group,  $(N, \ga,
  \widehat{\ga})$ a braided-commutative $\bfG$-Yetter-Drinfel'd algebra, $\nu$ a normal faithful semi-finite weight on
  $N$, and let $\mathfrak{G}(N, \ga, \widehat{\ga}, \nu)$ be the associated Hopf-bimodule, equipped with a co-inverse, a left-invariant operator-valued weight and a right-invariant valued weight by \ref{th1}(ii), \ref{th2} and \ref{th3}. Let $\widetilde{W}$ be the pseudo-mutiplicative unitary associated by \ref{defW}. Then
  \[\widetilde{W}=V_3^*(W^*\otimes 1_{H_\nu})V_1,\]
where $V_1$ had been defined in \ref{propV} and $V_3$ in \ref{V3}. Moreover, for any $\xi$, $\eta$ in $H$, $p$, $q$ in $\gN_\nu$,
\[(\id*\omega_{U^\ga_{\nu}(\eta\otimes J_\nu\Lambda_\nu(p)), \xi\otimes\Lambda_\nu(q)})(\widetilde{W})=\ga(q^*)[(\omega_{\eta, \xi}\otimes \id)(W^*)\otimes 1_{H_\nu}]\beta(p^*).\]
\end{theorem}
\begin{proof}
Let $x$, $x_1$, $x_2$ in $\gN_\nu$ and $y$, $y_1$, $y_2$ in
$\gN_{\widehat{\varphi}}$. Then $(y\otimes 1)\ga(x)$, $(y_1\otimes
1_{H_{\nu}})\ga(x_1)$, $(y_2\otimes 1_{H_\nu})\ga(x_2)$ belong to
$\gN_{\tilde{\nu}}\cap\gN_{T_{\tilde{\ga}}}$, and by
(\ref{action}),  $\Lambda_{\tilde{\nu}}[(y\otimes
1_{H_\nu})\ga(x)]=\Lambda_{\widehat{\varphi}}(y)\otimes\Lambda_\nu(x)$
and
\begin{align*}
  J_{\tilde{\nu}}\Lambda_{\tilde{\nu}}[(y\otimes
  1_{H_\nu})\ga(x)]&=U^\ga_\nu(\widehat{J}\Lambda_{\widehat{\varphi}}(y)\otimes
  J_\nu\Lambda_\nu(x)) \\
J_{\tilde{\nu}}\Lambda_{\tilde{\nu}}[\ga(x_1^*)(y_1^*y_2\otimes 1_{H_\nu})\ga(x_2)]&=(1_H\otimes J_\nu x_1^*J_\nu)U^\ga_\nu[\widehat{J}\Lambda_{\widehat{\varphi}}(y_1^*y_2)\otimes J_\nu\Lambda_\nu(x_2)].
\end{align*}
By definition of $\widetilde{W}$ (\ref{defW}), we find that for any
$\Xi_1$, $\Xi_2$ in $H\otimes H_\nu$, the scalar product
\[(\widetilde{W}[\Xi_2\underset{\nu}{_\beta\otimes_\ga}J_{\tilde{\nu}}\Lambda_{\tilde{\nu}}(\ga(x_1^*)(y_1^*y_2\otimes 1_{H_\nu})\ga(x_2))]|\Xi_1\underset{\nu^{\op}}{_\ga\otimes_{\gamma}}(\Lambda_{\widehat{\varphi}}(y)\otimes \Lambda_\nu(x)))\]
is equal to
\[(\widetilde{\Gamma}[(y\otimes 1)\ga(x)]^*(\Xi_2\underset{\nu}{_\beta\otimes_\ga}U^\ga_\nu(\widehat{J}\Lambda_{\widehat{\varphi}}(y_2)\otimes J_\nu\Lambda_\nu(x_2))|\Xi_1\underset{\nu}{_\beta\otimes_\ga}U^\ga_\nu(\widehat{J}\Lambda_{\widehat{\varphi}}(y_1)\otimes J_\nu\Lambda_\nu(x_1)))).\]
 Using  \ref{th1}, we get that this  is equal to
\[((\widehat{\Gamma}^{\op}(y^*)\otimes 1_{H_\nu})V_1[\Xi_2\underset{\nu}{_\beta\otimes_\ga}U^\ga_\nu(\widehat{J}\Lambda_{\widehat{\varphi}}(y_2)\otimes J_\nu\Lambda_\nu(x_2))]|V_1[\ga(x)\Xi_1\underset{\nu}{_\beta\otimes_\ga}(U^\ga_\nu(\widehat{J}\Lambda_{\widehat{\varphi}}(y_1)\otimes J_\nu\Lambda_\nu(x_1))),\]
which, thanks to \ref{propV}(i), is equal to
\[((\widehat{\Gamma}^{\op}(y^*)\otimes 1_{H_\nu})(\widehat{J}\Lambda_{\widehat{\varphi}}(y_2)\otimes\beta(x_2^*)\Xi_2)|\widehat{J}\Lambda_{\widehat{\varphi}}(y_1)\otimes\beta(x_1^*)\ga(x)\Xi_1)\]
and to
\begin{multline*}
(\beta(x_1)((\omega_{\widehat{J}\Lambda_{\widehat{\varphi}}(y_2), \widehat{J}\Lambda_{\widehat{\varphi}}(y_1)}\otimes \id)(\widehat{\Gamma}^{\op}(y^*))\otimes 1_{H_\nu})\beta(x_2^*)\Xi_2|\ga(x)\Xi_1)=\\
=(\beta(x_1)((\id\otimes \omega_{\widehat{J}\Lambda_{\widehat{\varphi}}(y_2), \widehat{J}\Lambda_{\widehat{\varphi}}(y_1)})(\widehat{\Gamma}(y^*))\otimes 1_{H_\nu})\beta(x_2^*)\Xi_2|\ga(x)\Xi_1),
\end{multline*}
which, by \ref{lcqg}, is equal to
\begin{multline*}
(\beta(x_1)((\id\otimes \omega_{\widehat{J}\Lambda_{\widehat{\varphi}}(y_1^*y_2), \Lambda_{\widehat{\varphi}}(y)})(\widehat{W})\otimes 1_{H_\nu})\beta(x_2^*)\Xi_2|\ga(x)\Xi_1)=\\
=(\beta(x_1)((\omega_{\widehat{J}\Lambda_{\widehat{\varphi}}(y_1^*y_2), \Lambda_{\widehat{\varphi}}(y)}\otimes \id)(W^*)\otimes 1_{H_\nu})\beta(x_2^*)\Xi_2|\ga(x)\Xi_1),
\end{multline*}
which is, using \ref{propV}(i) and \ref{V3}, equal to
\[((1_H\otimes\beta(x_1))(W^*\otimes 1_{H_\nu})V_1(\Xi_2\underset{\nu}{_\beta\otimes_\ga}U^\ga_\nu(\widehat{J}\Lambda_{\widehat{\varphi}}(y_1^*y_2)\otimes J_\nu\Lambda_\nu(x_2)))|V_3(\Xi_1\underset{\nu^{\op}}{_\ga\otimes_{\gamma}}(\Lambda_{\widehat{\varphi}}(y)\otimes\Lambda_\nu(x)))),\]
which, using \ref{corintegrable3}(iv), is the scalar product of the vector
\[W^*_{12} (\id \otimes \beta^{\dag})(\ga^{\op}(x^{\op}_{1}))V_1(\Xi_2\underset{\nu}{_\beta\otimes_\ga}U^\ga_\nu(\widehat{J}\Lambda_{\widehat{\varphi}}(y_1^*y_2)\otimes J_\nu\Lambda_\nu(x_2)))\]
with $V_3(\Xi_1\underset{\nu^{\op}}{_\ga\otimes_{\gamma}}(\Lambda_{\widehat{\varphi}}(y)\otimes\Lambda_\nu(x)))$.
But, using \ref{propV}(i), this vector is equal to
\[(W^*\otimes 1_{H_\nu})V_1[(1_H\otimes 1_{H_\nu})\underset{N}{_\beta\otimes_\ga}(1_H\otimes J_\nu x^*_{1}J_\nu)](\Xi_2\underset{\nu}{_\beta\otimes_\ga}U^\ga_\nu(\widehat{J}\Lambda_{\widehat{\varphi}}(y_1^*y_2)\otimes J_\nu\Lambda_\nu(x_2))).\]
Finally, we get that the initial scalar product 
\[(\widetilde{W}[\Xi_2\underset{\nu}{_\beta\otimes_\ga}J_{\tilde{\nu}}\Lambda_{\tilde{\nu}}(\ga(x_1^*)(y_1^*y_2\otimes 1)\ga(x_2))]|\Xi_1\underset{\nu^{\op}}{_\ga\otimes_{\gamma}}(\Lambda_{\widehat{\varphi}}(y)\otimes \Lambda_\nu(x)))\]
 is equal to
\[((W^*\otimes 1_{H_\nu})V_1[\Xi_2\underset{\nu}{_\beta\otimes_\ga}J_{\tilde{\nu}}\Lambda_{\tilde{\nu}}(\ga(x_1^*)(y_1^*y_2\otimes 1)\ga(x_2))]|
V_3(\Xi_1\underset{\nu^{\op}}{_\ga\otimes_{\gamma}}(\Lambda_{\widehat{\varphi}}(y)\otimes\Lambda_\nu(x)))).\]
By density of linear combinations of elements of the form $\Lambda_{\widehat{\varphi}}(y)\otimes \Lambda_\nu(x)$ in $D((H\otimes H_\nu)_{\gamma}, \nu^{\op})$, and then of linear combinations of elements of the form $\Xi_1\underset{\nu^{\op}}{_\ga\otimes_{\gamma}}(\Lambda_{\widehat{\varphi}}(y)\otimes \Lambda_\nu(x))$ in $(H\otimes H_\nu)\underset{\nu^{\op}}{_\ga\otimes_{\gamma}}(H\otimes H_\nu)$, we get that
\[\widetilde{W}[\Xi_2\underset{\nu}{_\beta\otimes_\ga}J_{\tilde{\nu}}\Lambda_{\tilde{\nu}}(\ga(x_1^*)(y_1^*y_2\otimes 1)\ga(x_2))]=
V_3^*(W^*\otimes 1_{H_\nu})V_1[\Xi_2\underset{\nu}{_\beta\otimes_\ga}J_{\tilde{\nu}}\Lambda_{\tilde{\nu}}(\ga(x_1^*)(y_1^*y_2\otimes 1)\ga(x_2))],\]
and, with the same density arguments, we get that $\widetilde{W}=V_3^*(W^*\otimes 1_{H_\nu})V_1$.  
Therefore, using again \ref{propV}(i) and \ref{V3}, we get that
\begin{multline*}
(\widetilde{W}[\Xi_2\underset{\nu}{_\beta\otimes_\ga}U^\ga_\nu(\eta\otimes J_\nu\Lambda_\nu(p))]|\Xi_1\underset{\nu^{\op}}{_\ga\otimes_\gamma}(\xi\otimes\Lambda_\nu(q)))=\\
=((W^*\otimes 1_{H_\nu})V_1[\Xi_2\underset{\nu}{_\beta\otimes_\ga}U^\ga_\nu(\eta\otimes J_\nu\Lambda_\nu(p))]|V_3[\Xi_1\underset{\nu^{\op}}{_\ga\otimes_\gamma}(\xi\otimes\Lambda_\nu(q))])
\end{multline*}
is equal to
\[((W^*\otimes 1_{H_\nu})(\eta\otimes\beta(p^*)\Xi_2)|\xi\otimes\ga(p)\Xi_1)
=
(((\omega_{\eta, \xi}\otimes \id)(W^*)\otimes 1_{H_\nu})\beta(p^*)\Xi_2|\ga(p)\Xi_1),\]
 which  finishes the proof. 
\end{proof}

\begin{theorem}
\label{dual}
 Let $\bfG$ be a locally compact quantum group,  $(N, \ga,
   \widehat{\ga})$ a braided-commu\-tative $\bfG$-Yetter-Drinfel'd
   algebra,  and $\nu$ a normal faithful semi-finite weight on $N$
   such that $\mathfrak{G}(N, \ga, \widehat{\ga}, \nu)$ is a measured
   quantum groupoid in the sense of \ref{defMQG}. Let $\widehat{
     \mathfrak{G}(N, \ga, \widehat{\ga}, \nu)}$ be its dual measured
   quantum groupoid in the sense of \ref{data}, and  for all $X\in \widehat{\bfG}\ltimes_{\widehat{\ga}}N$, let 
\[\mathcal I(X)=U^\ga_\nu(U^{\widehat{\ga}}_\nu)^*XU^{\widehat{\ga}}_\nu(U^\ga_\nu)^*.\]
Then $\mathcal I$ is an isomorphism of Hopf bimodule structures from
$\mathfrak{G}(N, \widehat{\ga}, \ga, \nu)$ onto $\widehat{
  \mathfrak{G}(N, \ga, \widehat{\ga}, \nu)}$. 
\end{theorem}
\begin{proof}
To prove this result, we calculate the pseudo-multplicative $\widetilde{\widehat{W}}$ of $\mathfrak{G}(N, \widehat{\ga}, \ga, \nu)$, using \ref{W} applied to $(N, \widehat{\ga}, \ga, \nu)$. 
We first define, as in \ref{propV}(i) and \ref{V3}, a unitary $\widehat{V_1}$ from $(H\otimes H_\nu)\underset{\nu}{_{\widehat{\beta}}\otimes_{\widehat{\ga}}}(H\otimes H_\nu)$ onto $H\otimes H\otimes H_\nu$, and a unitary $\widehat{V_3}$ from $(H\otimes H_\nu)\underset{\nu^{\op}}{_{\widehat{\ga}}\otimes_{\widehat{\gamma}}}(H\otimes H_\nu)$ onto $H\otimes H\otimes H_\nu$, where, for all $x\in N$, 
\begin{align*}
  \widehat{\beta}(x)&=U^{\widehat{\ga}}_\nu(U^\ga_\nu)^*(1_H\otimes
  J_\nu x^*J_\nu)U^\ga_\nu(U^{\widehat{\ga}}_\nu)^*, \\
  \widehat{\gamma}(x)&=J_{\widetilde{\hat{\nu}}}\widehat{\ga}(x^*)J_{\widetilde{\hat{\nu}}}=1_H\otimes
  J_\nu x^*J_\nu=\gamma(x),
\end{align*}
 $\widetilde{\hat{\nu}}$ denoting the dual weight on
  $\widehat{\bfG} \ltimes_{\widehat{\ga}} N$ as before.
More precisely, applying \ref{propV}(i) to $(N, \widehat{\ga}, \ga,
\nu)$, we get that for any $\xi$, $\eta$ in $H$ and $p$, $q$ in $\gN_\nu$,
\[\widehat{V_1}(U^{\widehat{\ga}}_\nu(U^\ga_\nu)^*\underset{N}{_{\widehat{\beta}}\otimes_{\widehat{\ga}}}U^{\widehat{\ga}}_\nu(U^\ga_\nu)^*)\sigma_{\nu^{\op}}
[U^\ga_\nu(\eta\otimes J_\nu\Lambda_\nu(q))\underset{\nu^{\op}}{_\ga\otimes_\gamma}(\xi\otimes\Lambda_\nu(p))]\]
is equal to
\begin{align*}
  \widehat{V_1}[U^{\widehat{\ga}}_\nu(U^\ga_\nu)^*(\xi\otimes\Lambda_\nu(p))\underset{\nu}{_{\widehat{\beta}}\otimes_{\widehat{\ga}}}U^{\widehat{\ga}}_\nu(\eta\otimes J_\nu\Lambda_\nu(q))]
&=
\eta\otimes\widehat{\beta}(q^*)U^{\widehat{\ga}}_\nu(U^\ga_\nu)^*(\xi\otimes\Lambda_\nu(p))\\
&=
\eta\otimes U^{\widehat{\ga}}_\nu(U^\ga_\nu)^*(\xi\otimes J_\nu qJ_\nu \Lambda_\nu(p))\\
&=
(1_H\otimes U^{\widehat{\ga}}_\nu(U^\ga_\nu)^*)(\eta\otimes\xi\otimes pJ_\nu\Lambda_\nu(q))
\end{align*}
On the other hand, using \ref{V3}, we get that
\begin{align*}
V_3[U^\ga_\nu(\eta\otimes J_\nu\Lambda_\nu(q))\underset{\nu^{\op}}{_\ga\otimes_\gamma}(\xi\otimes\Lambda_\nu(p))] &=
\xi\otimes\ga(p)U^\ga_\nu(\eta\otimes J_\nu\Lambda_{\nu} (q))\\
&=
\xi\otimes U^\ga_\nu(\eta\otimes pJ_\nu\Lambda_\nu(q))\\
&=
(1_H\otimes U^\ga_\nu)(\xi\otimes\eta\otimes pJ_\nu\Lambda_\nu(q)),
\end{align*}
from which we get that
\[\widehat{V_1}(U^{\widehat{\ga}}_\nu(U^\ga_\nu)^*\underset{N}{_{\widehat{\beta}}\otimes_{\widehat{\ga}}}U^{\widehat{\ga}}_\nu(U^\ga_\nu)^*)\sigma_{\nu^{\op}}=(1_H\otimes U^{\widehat{\ga}}_\nu(U^\ga_\nu)^*)(\sigma\otimes 1_{H_\nu})(1_H\otimes (U^\ga_\nu)^*)V_3.\]
 Applying this result to $(N, \widehat{\ga}, \ga, \nu)$ and taking the
 adjoints, we find that
\[\widehat{V_3}(U^{\widehat{\ga}}_\nu(U^\ga_\nu)^*\underset{N^{\op}}{_\ga\otimes_\beta}U^{\widehat{\ga}}_\nu(U^\ga_\nu)^*)\sigma_\nu=
(1_H\otimes U^{\widehat{\ga}}_\nu)(\sigma\otimes 1_{H\nu})(1_H\otimes U^{\widehat{\ga}}_\nu(U^\ga_\nu)^*)V_1.\]
Applying \ref{W} to $(N, \widehat{\ga}, \ga, \nu)$, we get that $\widetilde{\widehat{W}}=\widehat{V_3}^*(\sigma\otimes 1_{H_\nu})(W\otimes 1_{H_\nu})(\sigma\otimes 1_{H_\nu})\widehat{V_1}$ and, therefore, that $\sigma_{\nu^{\op}}[U^\ga_\nu(U^{\widehat{\ga}}_\nu)^*\underset{N^{\op}}{_{\widehat{\ga}}\otimes_\gamma}U^\ga_\nu(U^{\widehat{\ga}}_\nu)^*]\widetilde{\widehat{W}}[U^{\widehat{\ga}}_\nu(U^\ga_\nu)^*\underset{N}{_\gamma\otimes_\ga}U^{\widehat{\ga}}_\nu(U^\ga_\nu)^*]\sigma_{\nu^{\op}}$ is equal to:
\[V_1^*(U^\ga_\nu)_{23}(U^{\widehat{\ga}}_\nu)^*_{23}(U^{\widehat{\ga}}_\nu)^*_{13}W_{12}(U^{\widehat{\ga}}_{\nu})_{13}(U^\ga_\nu)^*_{13}(U^\ga_\nu)^*_{23}V_3\]
But, as $(\widehat{\Gamma}\otimes \id)(U^{\widehat{\ga}}_\nu)=(U^{\widehat{\ga}}_{\nu})_{23}(U^{\widehat{\ga}}_{\nu})_{13}$, we get that $(U^{\widehat{\ga}}_{\nu})^*_{23}(U^{\widehat{\ga}}_{\nu})^*_{13}=W_{12}(U^{\widehat{\ga}}_\nu)^*_{13}W_{12}^*$, and therefore that $(U^{\widehat{\ga}}_\nu)^*_{23}(U^{\widehat{\ga}}_\nu)^*_{13}W_{12}(U^{\widehat{\ga}}_{\nu})_{13}=W_{12}$. 
On the other hand, by the same argument,
$(U^\ga_\nu)^*_{13}(U^\ga_\nu)^*_{23}=W_{12}^*(U^\ga_\nu)^*_{23}W_{12}$. Finally,
we get  that
\[\sigma_{\nu^{\op}}[U^\ga_\nu(U^{\widehat{\ga}}_\nu)^*\underset{N^{\op}}{_{\widehat{\ga}}\otimes_\gamma}U^\ga_\nu(U^{\widehat{\ga}}_\nu)^*]\widetilde{\widehat{W}}[U^{\widehat{\ga}}_\nu(U^\ga_\nu)^*\underset{N}{_\gamma\otimes_\ga}U^{\widehat{\ga}}_\nu(U^\ga_\nu)^*]\sigma_{\nu^{\op}}=V_1^*W_{12}V_3=\widetilde{W}^*,\]
and therefore
\[[U^\ga_\nu(U^{\widehat{\ga}}_\nu)^*\underset{N^{\op}}{_{\widehat{\ga}}\otimes_\gamma}U^\ga_\nu(U^{\widehat{\ga}}_\nu)^*]\widetilde{\widehat{W}}[U^{\widehat{\ga}}_\nu(U^\ga_\nu)^*\underset{N}{_\gamma\otimes_\ga}U^{\widehat{\ga}}_\nu(U^\ga_\nu)^*]=\sigma_\nu\widetilde{W}^*\sigma_\nu.\]
So, up to the isomorphism, the pseudo-multiplicative unitary $\widetilde{\widehat{W}}$ of $\mathfrak{G}(N, \widehat{\ga}, \ga, \nu)$ is equal to the dual pseudo-muliplicative untary $\widehat{\widetilde{W}}$, which finishes the proof. 
\end{proof}

\begin{proposition}
\label{coR}
 Let $\bfG$ be a locally compact quantum group,  $(N, \ga,
  \widehat{\ga})$ a braided-commu\-tative $\bfG$-Yetter-Drinfel'd
  algebra, and $\nu$ a normal faithful semi-finite weight on $N$. Suppose that $\mathfrak{G}(N, \ga, \widehat{\ga}, \nu)$ is a measured quantum groupoid in the sense of \ref{defMQG}, and let $\widehat{ \mathfrak{G}}(N, \ga, \widehat{\ga}, \nu)$ be its dual measured quantum groupoid in the sense of \ref{data}.
  \begin{enumerate}
  \item  The co-inverse $\widetilde{R}$ constructed in
    \ref{th2}(ii) is the canonical co-inverse of the measured quantum
    groupoid $\mathfrak{G}(N, \ga, \widehat{\ga}, \nu)$.
  \item  The isomorphim of Hopf bimodules from $\mathfrak{G}(N,
    \widehat{\ga}, \ga, \nu)$ onto $\widehat{ \mathfrak{G}}(N, \ga,
      \widehat{\ga}, \nu)$ constructed in \ref{dual} exchanges the
    canonical co-inverses of these Hopf-bimodules.
  \end{enumerate}
\end{proposition}
\begin{proof}
(i) By \ref{V3}(iv),
$V_3(I\underset{N}{_\beta\otimes_\ga}J_{\tilde{\nu}})=(\widehat{J}\otimes
I)V_1$. Taking adjoints, we also get
$V_1(I\underset{N^{\op}}{_\ga\otimes_\gamma}J_{\tilde{\nu}})=(\widehat{J}\otimes
I)V_3$. Therefore, we get, using \ref{W} and \ref{propI}(iii),
\begin{align*}
(I\underset{N^{\op}}{_\ga\otimes_\gamma}J_{\tilde{\nu}})\widetilde{W}(I\underset{N^{\op}}{_\ga\otimes_\gamma}J_{\tilde{\nu}})
&=
(I\underset{N^{\op}}{_\ga\otimes_\gamma}J_{\tilde{\nu}})V_3^*(W^*\otimes 1_{H_\nu})V_1(I\underset{N^{\op}}{_\ga\otimes_\gamma}J_{\tilde{\nu}})\\
&=
V_1^*(\widehat{J}\otimes I)(W^*\otimes 1_{H_\nu})(\widehat{J}\otimes I)V_3\\
&=
V_1^*(W\otimes 1_{H_\nu})V_3\\
&=
\widetilde{W}^*.
\end{align*}
 For all $\Xi\in D(_\ga(H\otimes H_\nu), \nu)$ and $\Xi'\in
 D((H\otimes H_\nu)_\gamma, \nu^{\op})$,  we therefore have
\[I(\id*\omega_{\Xi, \Xi'})(\widetilde{W})^*I=(\id*\omega_{J_{\tilde{\nu}}\Xi', J_{\tilde{\nu}}\Xi})(\widetilde{W}),\]
which proves that the canonical co-inverse is given  by $\widetilde{R}(X)=IX^*I$ for all $X\in\bfG\ltimes_\ga N$. 

(ii) By \ref{data},  the canonical co-inverse of
$\widehat{ \mathfrak{G}}(N, \ga, \widehat{\ga}, \nu)$ is implemented
by $J_{\tilde{\nu}}$. Using (ii) applied to $\gG(N,
\widehat{\ga}, \ga, \nu)$, we therefore get that the canonical co-inverse of
$\gG(N, \widehat{\ga}, \ga, \nu)$ is implemented by
$\widehat{I}=U^{\widehat{\ga}}_\nu(U^\ga_\nu)^*J_{\tilde{\nu}}U^\ga_\nu(U^{\widehat{\ga}})^*$.
 \end{proof}
 
 \begin{theorem}
 \label{ovw}
 Let $\bfG$ be a locally compact quantum group,  $(N, \ga,
   \widehat{\ga})$ a braided-commu\-tative $\bfG$-Yetter-Drinfel'd
   algebra, and $\nu$ a normal faithful semi-finite weight on $N$.  Suppose that $\mathfrak{G}(N, \ga,
   \widehat{\ga}, \nu)$ is a measured quantum groupoid in the sense of
   \ref{defMQG},  let $\widehat{ \mathfrak{G}}(N, \ga, \widehat{\ga},
     \nu)$ be its dual measured quantum groupoid in the sense of
   \ref{data}, and let $\mathcal I$ be the isomorphism of Hopf bimodue
   structures constructed in \ref{dual}. Then $\mathcal I$ exchanges
   the left-invariant and the right-invariant operator-valued weights
   on $\gG(N, \widehat{\ga}, \ga, \nu)$ and $\widehat{\gG}(N, \ga,
     \widehat{\ga},\nu)$. Therefore, $\gG(N, \widehat{\ga}, \ga,
   \nu)$ is also a measured quantum groupoid. 
\end{theorem}
 \begin{proof}
 Using \ref{coR}(ii), it suffices to verify that $\mathcal I$ exchanges the left-invariant operator valued weights, of  $\gG(N, \widehat{\ga}, \ga, \nu)$ and $\widehat{\gG}(N, \ga, \widehat{\ga}, \nu)$. The left-invariant weight of $\gG(N, \widehat{\ga}, \ga, \nu)$ is the dual weight $\widetilde{\hat{\nu}}$ on the crossed product $\widehat{\bfG}\ltimes_{\widehat{\ga}}N$. Let us denote by $\widehat{\Phi}$ the left-invariant weight of $\widehat{\gG}(N, \ga, \widehat{\ga}, \nu)$. 
 
We apply \ref{W} to $\gG(N, \widehat{\ga}, \ga, \nu)$ and get that, for any $\xi$ in $H$, $z\in\gN_{\widehat{\varphi}}$, $p$, $q$ in $\gN_\nu$,
 \[(\id*\omega_{U^{\widehat{\ga}}_\nu(\widehat{J}\Lambda_{\widehat{\varphi}}(z)\otimes J_\nu\Lambda_\nu(p)), \xi\otimes\Lambda_\nu(q)})(\widetilde{\widehat{W}})
 (\id*\omega_{U^{\widehat{\ga}}_\nu(\widehat{J}\Lambda_{\widehat{\varphi}}(z)\otimes J_\nu\Lambda_\nu(p)), \xi\otimes\Lambda_\nu(q)})(\widetilde{\widehat{W}})^*\]
 is equal to
 \begin{align*}
  \widehat{\ga}(q^*)[(\id\otimes\omega_{\widehat{J}\Lambda_{\widehat{\varphi}}(z), \xi})(W)\otimes 1]\widehat{\beta}(pp^*)[(\id\otimes\omega_{\widehat{J}\Lambda_{\widehat{\varphi}}(z), \xi})(W)^*\otimes 1]\widehat{\ga}(q), 
 \end{align*}
 where, as in \ref{dual}, $\widetilde{\widehat{W}}$ denotes the
 pseudo-multiplicative unitary associated to $\gG(N, \widehat{\ga},
 \ga, \nu)$, and $\widehat{\beta}$ is defined, for $x\in N$, by
 $\widehat{\beta}(x)=U^{\widehat{\ga}}_\nu(U^\ga_\nu)^*(1_H\otimes
 J_\nu x^*J_\nu)U^\ga_\nu(U^{\widehat{\ga}}_\nu)^*$.  Let us take now
 a family $(p_i)_{i\in I}$ in $\gM_\nu^+$, increasing to $1$. Then, we
 get that
 \begin{align*}
  \widehat{\ga}(q^*)[(\id\otimes\omega_{\widehat{J}\Lambda_{\widehat{\varphi}}(z),
   \xi})(W)(\id\otimes\omega_{\widehat{J}\Lambda_{\widehat{\varphi}}(z),
   \xi})(W)^*\otimes 1]\widehat{\ga}(q) 
 \end{align*}
 is the increasing limit of
\[(\id*\omega_{U^{\widehat{\ga}}_\nu(\widehat{J}\Lambda_{\widehat{\varphi}}(z)\otimes J_\nu\Lambda_\nu(p_i^{1/2})), \xi\otimes\Lambda_\nu(q)})(\widetilde{\widehat{W}})
 (\id*\omega_{U^{\widehat{\ga}}_\nu(\widehat{J}\Lambda_{\widehat{\varphi}}(z)\otimes J_\nu\Lambda_\nu(p_i^{1/2})), \xi\otimes\Lambda_\nu(q)})(\widetilde{\widehat{W}})^*.\] 
But, using \ref{dual}, we get that $(\id*\omega_{U^{\widehat{\ga}}_\nu(\widehat{J}\Lambda_{\widehat{\varphi}}(z)\otimes J_\nu\Lambda_\nu(p_i^{1/2})), \xi\otimes\Lambda_\nu(q)})(\widetilde{\widehat{W}})$ is equal to
\begin{multline*}
\mathcal I^{-1}[(\id*\omega_{U^\ga_\nu(\widehat{J}\Lambda_{\widehat{\varphi}}(z)\otimes J_\nu\Lambda_\nu(p_i^{1/2})), U^\ga_\nu(U^{\widehat{\ga}}_\nu)^*(\xi\otimes\Lambda_\nu(q))})(\sigma_{\nu^{\op}}\widetilde{W}^*\sigma_{\nu^{\op}})]=\\
=
\mathcal I^{-1}[(\omega_{U^\ga_\nu(U^{\widehat{\ga}}_\nu)^*(\xi\otimes \Lambda_\nu(q)), U^\ga_\nu(\widehat{J}\Lambda_{\widehat{\varphi}}(z)\otimes J_\nu\Lambda_\nu(p_i^{1/2}))}*\id)(\widetilde{W})^*].
\end{multline*}
Therefore, we get that $\widehat{\Phi}\circ\mathcal I[\widehat{\ga}(q^*)[(\id\otimes\omega_{\widehat{J}\Lambda_{\widehat{\varphi}}(z), \xi})(W)(\id\otimes\omega_{\widehat{J}\Lambda_{\widehat{\varphi}}(z), \xi})(W)^*\otimes 1]\widehat{\ga}(q)]$ is the increasing limit of
\[\widehat{\Phi}[(\omega_{U^\ga_\nu(U^{\widehat{\ga}}_\nu)^*(\xi\otimes \Lambda_\nu(q)), U^\ga_\nu(\widehat{J}\Lambda_{\widehat{\varphi}}(z)\otimes J_\nu\Lambda_\nu(p_i^{1/2}))}*\id)(\widetilde{W})^*(\omega_{U^\ga_\nu(U^{\widehat{\ga}}_\nu)^*(\xi\otimes \Lambda_\nu(q)), U^\ga_\nu(\widehat{J}\Lambda_{\widehat{\varphi}}(z)\otimes J_\nu\Lambda_\nu(p_i^{1/2}))}*\id)(\widetilde{W})],\]
which, using \ref{data}, is equal, by definition, to the increasing limit of 
\[\|\omega_{U^\ga_\nu(U^{\widehat{\ga}}_\nu)^*(\xi\otimes \Lambda_\nu(q)), U^\ga_\nu(\widehat{J}\Lambda_{\widehat{\varphi}}(z)\otimes J_\nu\Lambda_\nu(p_i^{1/2}))}\|^2_{\tilde{\nu}}.\] 
For  $X\in\gN_{\tilde{\nu}}$, the scalar  $\omega_{U^\ga_\nu(U^{\widehat{\ga}}_\nu)^*(\xi\otimes \Lambda_\nu(q)), U^\ga_\nu(\widehat{J}\Lambda_{\widehat{\varphi}}(z)\otimes J_\nu\Lambda_\nu(p_i^{1/2}))}(X^{*})$
is equal to
\begin{multline*}
(X^*U^\ga_\nu(U^{\widehat{\ga}}_\nu)^*(\xi\otimes \Lambda_\nu(q))|U^\ga_\nu(\widehat{J}\Lambda_{\widehat{\varphi}}(z)\otimes J_\nu\Lambda_\nu(p_i^{1/2})))=\\
=(U^\ga_\nu(U^{\widehat{\ga}}_\nu)^*(\xi\otimes \Lambda_\nu(q))|XJ_{\tilde{\nu}}\Lambda_{\tilde{\nu}}[(z\otimes 1)\ga(p_i^{1/2})])=\\
=
(U^\ga_\nu(U^{\widehat{\ga}}_\nu)^*(\xi\otimes \Lambda_\nu(q))|J_{\tilde{\nu}}(z\otimes 1)\ga(p_i^{1/2})J_{\tilde{\nu}}\Lambda_{\tilde{\nu}}(X))
\end{multline*}
and, therefore, 
\[\|\omega_{U^\ga_\nu(U^{\widehat{\ga}}_\nu)^*(\xi\otimes \Lambda_\nu(q)), U^\ga_\nu(\widehat{J}\Lambda_{\widehat{\varphi}}(z)\otimes J_\nu\Lambda_\nu(p_i^{1/2}))}\|^2_{\tilde{\nu}}=\|J_{\tilde{\nu}}\ga(p_i^{1/2})(z^*\otimes 1)J_{\tilde{\nu}}U^\ga_\nu(U^{\widehat{\ga}}_\nu)^*(\xi\otimes\Lambda_\nu(q))\|^2.\]
The  limit when $p_i$ goes to $1$ is equal to 
\begin{align*}
\|(\widehat{J}z^*\widehat{J}\otimes 1)(U^{\widehat{\ga}}_\nu)^{*}(\xi\otimes\Lambda_\nu(q))\|^2
&=
\|(\widehat{J}z^*\widehat{J}\otimes 1)(\xi\otimes\Lambda_\nu(q))\|^2\\
&=
\|\widehat{J}z^*\widehat{J}\xi\|^2\|\Lambda_\nu(q)\|^2\\
&=
\|\omega_{\xi, \widehat{J}\Lambda_{\widehat{\varphi}}(z)}\|^2_{\widehat{\varphi}}\|\Lambda_\nu(q)\|^2\\
&=
\|\Lambda_{\varphi}[(\id\otimes\omega_{\xi, \widehat{J}\Lambda_{\widehat{\varphi}}(z))}(W^*)]\otimes\Lambda_\nu(q)\|^2\\
&=
\|\Lambda_{\widetilde{\hat{\nu}}} [((\id\otimes\omega_{\xi, \widehat{J}\Lambda_{\widehat{\varphi}}(z)})(W^*)\otimes 1_{H_\nu})\widehat{\ga}(q)]\|^2,
\end{align*}
from which we get that
\[\|\Lambda_{\widehat{\Phi}\circ\mathcal I}([((\id\otimes\omega_{\xi, \widehat{J}\Lambda_{\widehat{\varphi}}(z)})(W^*)\otimes 1_{H_\nu})\widehat{\ga}(q)]\|^2
=\|\Lambda_{\widetilde{\hat{\nu}}} [((\id\otimes\omega_{\xi, \widehat{J}\Lambda_{\widehat{\varphi}}(z)})(W^*)\otimes 1_{H_\nu})\widehat{\ga}(q)]\|^2,\]
which proves that the left-invariant weight
$\widehat{\Phi}\circ\mathcal I+\widetilde{\hat{\nu}}$ is
semi-finite. Using now (\cite{L}5.2.2), we get that there exists an invertible $p\in N^+$, $p\leq 1$, such that
\[(D\widetilde{\hat{\nu}}:D(\widehat{\Phi}\circ\mathcal
I+\widetilde{\hat{\nu}}))_t=\beta(p)^{it}\]
 for all $t\in\R$.
So, $\beta(p)$ is invariant under the modular group $\sigma^{\widetilde{\hat{\nu}}}$ (i.e.\ $p$ is invariant under $\gamma$) and we get that
\begin{multline*}
2\|\Lambda_{\widetilde{\hat{\nu}}} [((\id\otimes\omega_{\xi, \widehat{J}\Lambda_{\widehat{\varphi}}(z)})(W^*)\otimes 1_{H_\nu})\widehat{\ga}(q)]\|^2
=\\
=\|\Lambda_{\widehat{\Phi}\circ\mathcal I+\widetilde{\hat{\nu}}}[((\id\otimes\omega_{\xi, \widehat{J}\Lambda_{\widehat{\varphi}}(z)})(W^*)\otimes 1_{H_\nu})\widehat{\ga}(q)]\|^2=\\
=
\|J_{\widetilde{\hat{\nu}}}\beta(p^{-1})J_{\widetilde{\hat{\nu}}}\Lambda_{\widetilde{\hat{\nu}}} [((\id\otimes\omega_{\xi, \widehat{J}\Lambda_{\widehat{\varphi}}(z)})(W^*)\otimes 1_{H_\nu})\widehat{\ga}(q)]\|^2,
\end{multline*}
from which we get that $p=1/2$, and
$\widetilde{\hat{\nu}}=1/2(\widehat{\Phi}\circ\mathcal
I+\widetilde{\hat{\nu}})$. Thus,
$\widetilde{\hat{\nu}}=\widehat{\Phi}\circ\mathcal I$.
 \end{proof}

\begin{theorem}
\label{thduality}
 Let $\bfG$ be a locally compact quantum group,  $(N, \ga,
  \widehat{\ga})$ a braided-commu\-tative $\bfG$-Yetter-Drinfel'd
  algebra, $\nu$ a normal faithful semi-finite weight on $N$.  Let
  $D_t$ be the Radon-Nikodym derivative of the weight $\nu$ with
  respect to the action $\ga$ and $\widehat{D}_t$ be the Radon-Nikodym
  derivative of the weight $\nu$ with respect to the action
  $\widehat{\ga}$. Then the following conditions are equivalent :
  \begin{enumerate}
  \item  $\mathfrak{G}(N, \ga, \widehat{\ga}, \nu)$ is a measured
    quantum groupoid;
  \item  $\mathfrak{G}(N, \widehat{\ga}, \ga, \nu)$ is a measured
    quantum groupoid;
  \item  $(\tau_t Ad(\delta^{-it})\otimes\gamma_t)(D_s)=D_s$ for all
    $s,t\in \R$;
  \item  $(\widehat{\tau}_t
    Ad(\widehat{\delta}^{-it})\otimes\gamma_{-t})(\widehat{D}_s)=D_s$
    for all $s,t\in \R$.
  \end{enumerate}
\end{theorem}
\begin{proof}
By \ref{ovw}, we know that (i) implies (ii), and is therefore equivalent to (ii). Moreover, by \ref{thmqg}, we know that (i) is equivalent to (iii). Applying \ref{thmqg} to $\mathfrak{G}(N, \widehat{\ga}, \ga, \nu)$, we obtain (iv), because the one-parameter group $\widehat{\gamma}_t$ is equal to $\gamma_{-t}$. The proof that (iv) implies (ii) is the same as in \ref{thmqg}, where we use again that the one-parameter group $\widehat{\gamma}_t$ of $N$ constructed from the dual measured quantum groupoid is equal to $\gamma_{-t}$ (\ref{data}). \end{proof}

\begin{corollary}
\label{corduality}
 Let $\bfG$ be a locally compact quantum group,  $(N, \ga,
  \widehat{\ga})$ a braided-commu\-tative $\bfG$-Yetter-Drinfel'd
  algebra, and $\nu$ a normal faithful semi-finite weight on $N$.  If
  the weight $\nu$ is $\widehat{k}$-invariant with respect to
  $\widehat{\ga}$, for $\widehat{k}$ affiliated to the center
  $Z(\widehat{M})$ or $\widehat{k}=\widehat{\delta}^{-1}$, then
  $\mathfrak{G}(N, \ga, \widehat{\ga}, \nu)$ is a measured quantum
  groupoid and its dual is isomorphic to $\mathfrak{G}(N,
  \widehat{\ga}, \ga, \nu)$.
\end{corollary}
\begin{proof} We verify easily property (iv) of \ref{thduality}, and then obtain the result by \ref{thduality} and \ref{ovw}. \end{proof}

\section{Examples}
\label{examples}
In this chapter, we give several examples of measured quantum groupoids constructed from a braided-commutative Yetter-Drinfel'd algebra. First, in \ref{groupoid}, we show that usual transformation groupoids are indeed a particular case of this construction, which justifies the terminology. Other examples are constructed from quotient type co-ideals of compact quantum groups, in particular one is constructed from the Podle\' s sphere $S_q^2$ (\ref{Sq2}). Another example (\ref{HG}) is constructed from a normal closed subgroup $H$ of a locally compact group $G$. 

\subsection{Transformation Groupoid}
\label{groupoid}
Let us consider a locally compact group $G$ right acting on a locally compact space $X$; let us denote $\ga$ this action. It is well known that this leads to a locally compact groupoid $X\underset{\ga}{\curvearrowleft} G$, usually called a \emph{transformation groupoid}. This groupoid is the set $X\times G$, with $X$ as set of units, and range and source applications given by $r(x,g)=x$ and $s(x,g)=\ga_g(x)$, the product being $(x,g)(\ga_g(x),h)=(x, gh)$, and the inverse $(x,g)^{-1}=(\ga_g(x), g^{-1})$ (\cite{R} 1.2.a). This locally compact groupoid has a left Haar system (\cite{R} 2.5a), and for any measure $\nu$ on $X$, the lifted measure on $X\times G$ is $\nu\otimes\lambda$, where $\lambda$ is the left Haar measure on $G$. 

The measure $\nu$ is then quasi-invariant in the sense of \cite{R} and
\ref{defHopf} if and only if $\nu\otimes\lambda$ is equivalent to its
image under the inversion $(x,g)\rightarrow (x,g)^{-1}$. This is
equivalent (\cite{R}, 3.21) to asking that, for all $g\in G$, the
measure $\nu\circ\ga_g$ is equivalent to $\nu$, which leads to a
Radon-Nikodym $\Delta(x,g)=\frac{d\nu\circ\ga_{g^{-1}}}{d\nu}(x)$.
Then, the Radon-Nikodym derivative between $\nu\otimes\lambda$ and its image under the inversion $(x,g)\rightarrow (x,g)^{-1}$ is $\Delta(x,g)\Delta_G(g)$, where $\Delta_G$ is the modulus of $G$. 

Let us consider the trivial action of the dual locally compact quantum
group $\widehat{G}$, defined by $\iota(f)=1\otimes f$ for all $f\in
L^\infty (X)$. It is straightforward to verify that $(L^\infty(X),
\ga, \iota)$ is a $G$-Yetter-Drinfel'd algebra which is
braided-commutative. The measure $\nu$, regarded as a normal
semi-finite faithful weight on $L^\infty (X)$, is evidently invariant
under $\iota$. So, by \ref{corduality}, we obtain measured quantum
groupoid structures on the crossed products $G\ltimes_\ga L^\infty
(X)$ and $\widehat{G}\ltimes_{\iota}L^\infty(X)$. 

The von Neumann algebra $\widehat{G}\ltimes_{\iota}L^\infty(X)$ is $L^\infty (G)\otimes L^\infty (X)$, or $L^\infty (X\underset{\ga}{\curvearrowleft} G)$, and the structure of measured quantum groupoid is nothing but the structure given by the groupoid structure of $X\underset{\ga}{\curvearrowleft} G$. 

The dual measured quantum groupoid
$\widehat{X\underset{\ga}{\curvearrowleft} G}$ is the von Neumann
algebra generated by the left regular representation of
$X\underset{\ga}{\curvearrowleft }G$, which is the crossed product
$G\ltimes_\ga L^\infty (X)$. 
Let us note that this measured quantum
grouped is co-commutative, in particular, $\beta=\ga$ and
$\gamma_t=\sigma_t^\nu=\id_{L^\infty (X, \nu)}$ for all $t\in\R$. As $\tau_t=\Ad(\Delta_G^{it})=\id_{L^\infty (G)}$, we see that $D_t=\Delta(x,g)^{it}$ satisfies the condition of \ref{thduality}. Moreover,  $\widehat{D}_t=1$ for all $t\in\R$. 

Therefore, we get that any transformation groupoid gives a very
particular case of our ``measured quantum transformation groupoids'', which explains the terminology. 

\subsection{Basic example}
\label{basicex}
Let $\bfG=(M, \Gamma, \varphi, \varphi\circ R)$ be a locally compact
quantum group,  $\DG$ its quantum double, and let us use the notation introduced in \ref{basic}. There exists an action ${\ga_D}$ of $\DG$ on $M$ such that
\[{\ga_D}(x)\otimes 1=\Gamma_D(x\otimes 1).\]
The Yetter-Drinfel'd algebra associated to this action is given by the restrictions of the applications $\gb$ and $\widehat{\gb}$ to $M$, which are, respectively,  the coproduct $\Gamma$ (when considered as a left action of $\bfG$ on $M$), and the adjoint action $\ad$ of $\widehat{\bfG}$ on $M$ given by
\begin{align} \label{eq:ad}
\ad(x)=\sigma W (x\otimes 1)W^*\sigma = \widehat {W}^{*}(1 \otimes
x)\widehat{W},
 \end{align}
and we get this way the Yetter-Drinfel'd algebra $(M, \Gamma, \ad)$, which is the basic example given in \cite{NV}. Moreover, as 
\begin{align}
  \varsigma\Gamma(x)=((R\otimes R)\circ\Gamma\circ
  R)(x)=(\widehat{J}\otimes\widehat{J})W^*(\widehat{J}\otimes\widehat{J})(1\otimes
  x)(\widehat{J}\otimes\widehat{J})W(\widehat{J}\otimes\widehat{J}),\label{eq:coprod}
  \end{align}
we get that that
\[\varsigma \alpha^{\op}(x^{\op})=(J\widehat{J}\otimes 1)W^*(1\otimes \widehat{J}x\widehat{J})W(\widehat{J}J\otimes 1)=(J\otimes J)W(1\otimes J\widehat{J}x\widehat{J}J)W^*(J\otimes J)\]
(where we prefer to note $\alpha$ the left action $\Gamma$ to avoid confusion between $\alpha^{\op}$ defined in \ref{BCdef1} and the coproduct $\Gamma^{\op}$ of the locally compact quantum group $\bfG^{\op}$). But
\[\varsigma \ad^{\op}(x^{\op})=(J\otimes J)W(x\otimes 1)W^*(J\otimes J)\]
from which we get that this Yetter-Drinfel'd algebra is braided-commutative.

As $\varphi$ is invariant under $\Gamma$, using \ref{thmqg}, we can
equip the crossed products $\bfG\ltimes_\Gamma M$ and
$\widehat\bfG\ltimes_{\ad} M$ with structures of measured quantum
groupoids. 

Let us describe $\widehat\bfG\ltimes_{\ad} M$ in more detail.
We claim that the map $\Phi:=\Ad((J\widehat J \otimes 1)\widehat W)$ identifies  $\widehat\bfG\ltimes_{\ad} M$ with
$M' \otimes M$.  Indeed, the first algebra is generated by elements of
the form $(z\otimes
1)\ad(x)$ and $x,z \in M$, and
\begin{align*}
  \Ad(\widehat W)[(z \otimes 1)\ad(x)] &= \Gamma^{\op}(z)(1 \otimes x)
  = \Ad(\sigma)(\Gamma(z)(x \otimes 1)).
\end{align*}
But elements of the form $\Gamma(z)(x\otimes 1)$ generate $M\otimes
M$,  and as $\Ad(J\widehat J)(M)=M'$, the assertion follows. We
just saw that $\Phi(\ad(x))=1 \otimes x$, and we claim that
$\Phi(\beta(x)) = x^{\op} \otimes 1$. Using \eqref{eq:ad} and the fact
that $\widehat W^{*}$ is a cocycle for the trivial action of
$\widehat \bfG$ on $M$,  we get (\cite{V} 4.2)
\begin{align*}
  U^{\ad}_{\phi} &= \widehat W^{*}(J \otimes J)\widehat W (J \otimes J)
\end{align*}
 and therefore, using the relations $(J \otimes
 \widehat{J})\widehat{W}^{*}(J \otimes \widehat{J})=\widehat{W}$ and
 $\Gamma \circ R= (R \otimes R) \circ \Gamma^{\op}$ (\ref{lcqg}),
 \begin{align*}
   \Phi(\beta(x)) &= \Ad((J\widehat J\otimes 1)\widehat{W} U^{\ad}_{\phi}(\widehat J \otimes J))[\Gamma(x)] \\
&= \Ad((J\widehat J\otimes 1)(J \otimes
J)\widehat W (J \otimes J)(\widehat J \otimes J))[\Gamma(x)]  \\
&= \Ad((\widehat{J} \otimes J)\widehat W(J \widehat J \otimes \widehat
J \widehat J))[\Gamma(x)] \\
&= \Ad((\widehat J J \otimes J \widehat J)
\widehat{W}^{*})[\Gamma^{\op}(R(x))] \\
&= \Ad((\widehat J J \otimes J \widehat J))[R(x) \otimes 1] \\
&= x^{\op}
\otimes 1.
 \end{align*}
Therefore, $\Phi$ defines an isomorphism between
$\mathfrak{G}(M,\ad,\Gamma,\phi)$ and the pair quantum groupoid $M'
\otimes M$ of Lesieur (\cite{L} 15), and induces an isomorphism
between the respective  duals, which are (isomorphic to)
$\mathfrak{G}(M,\Gamma,\ad,\phi)$ and  the dual pair quantum
groupoid $B(H)$ constructed in (\cite{L} 15.3.7), respectively.

\subsection{Quantum Measured groupoid associated to an action}
\label{QGaction}
Let us apply \ref{basicex} to $\bf{\widehat{G}^o}$. We obtain that $(\widehat{M}, \widehat{\Gamma}^o, \ad)$ is a $\bf{\widehat{G}^o}$-Yetter-Drinfel'd algebra, where $\ad$ means here $\ad(x)=W^{c*}(1\otimes x)W^c$. As noticed by (\cite{NV}, 3.1), we can extend this example to any crossed-product $\bf G$$\ltimes_\ga N$, where $\ga$ is a left action of $\bf G$ on a von Neumann algebra $N$. Let us recall this construction. For any $X\in\bf G\ltimes_\ga N$, the dual action $\widetilde{\ga}$ is given by 
\[\widetilde{\ga}(X)=(\widehat{W}^{o*}\otimes 1)(1\otimes X)(\widehat{W}^o\otimes 1).\]
 Let us also write 
\[\underline{\ad}(X)=(W^{c*}\otimes 1)(1\otimes X)(W^c\otimes 1).\]
We first show that this formula defines an action $\underline{\ad}$ of $\bf G^o$ on $\bf G\ltimes_\ga N$. If $X=y\otimes 1$, with $y\in \widehat{M}$, we get that $\underline{\ad}(1\otimes y)=\ad(y)\otimes 1$, which belongs to $M'\otimes \bf G\ltimes_\ga N$. If $X=\ga(x)$, with $x\in N$, we get that $\underline{\ad}(\ga(x))=(W^{c*}\otimes 1)(1\otimes\ga(x))(W^c\otimes 1)$, which belongs to $M'\otimes \bfG\ltimes_\ga N$; moreover, the properties of $W^{c*}$ give then that $\underline{\ad}$ is an action. 

To prove that $(\bf G$$\ltimes_\ga N, \widetilde{\ga}, \underline{\ad})$ is a $\bf{\widehat{G}^o}$-Yetter-Drinfel'd algebra, we have to check that, for any $X\in \bf{\widehat{G}^o}$,
\[\Ad (\sigma_{12}\widehat{W}^o_{12})(\id\otimes\underline{\ad})\widetilde{\ga}(X)=(\id\otimes\widetilde{\ga})\underline{\ad}(X).\]
To check that, it suffices to prove that $\sigma_{12}\widehat{W}^o_{12}W^{c*}_{23}\widehat{W}^{o*}_{13}=\widehat{W}^{o*}_{23}W^c_{13}$, which follows from $\widehat{W}^o=\sigma W^{c*}\sigma$ and the pentagonal relation for $W^c$. 

\subsubsection{\bf{Proposition}}
\label{propB}
{\it Let $\ga$ an action of a locally compact quantum group $\bf G$ on a von Neumann algebra $N$ and let $B=\bf G\ltimes_\ga N\cap \ga(N)'$. Then the formulas 
\[\gb(X)=(\widehat{W}^{o*}\otimes 1)(1\otimes X)(\widehat{W}^o\otimes 1),\]
\[\widehat{\gb}(X)=(W^{c*}\otimes 1)(1\otimes X)(W^c\otimes 1)\]
define actions $\gb$ and $\widehat{\gb}$ of $\widehat{\bf G}^o$ and $\bf G^c$, respectively, on $B$ and $(B, \gb, \widehat{\gb})$ is a braided-commutative Yetter-Drinfel'd algebra. }
\begin{proof}
As $\widetilde{\ga}(\ga(x))=1\otimes\ga(x)$, for all $x\in N$, we get that $\gb$ is an action of $\widehat{\bf G}^o$ on $B=\bf G$$\ltimes_\ga N\cap \ga(N)'$. 

To prove a similar result for $\widehat{\gb}$, we need to make a detour via the inclusion $\ga(N)\subset G\ltimes_\ga N$ which is depth 2 (\cite{V} 5.10). Let $\nu$ be a normal faithful semi-finite weight on $N$, and $\widetilde{\nu}$ its dual weight on $\bf G$$\ltimes_\ga N$. Then, we have 
\begin{align*}
J_{\widehat{\nu}}\ga(N)'J_{\widehat{\nu}}&=(\widehat{J}\otimes J_\nu)(U^\ga_\nu)^*\ga(N)'U_\nu^\ga (\widehat{J}\otimes J_\nu) \\ &=(\widehat{J}\otimes J_\nu)(B(H)\otimes N')(\widehat{J}\otimes J_\nu)=B(H)\otimes N
\end{align*}
and therefore $B(H)\otimes N\cap (\bf G$$\ltimes_\ga N)'=J_{\widetilde{\nu}}BJ_{\widetilde{\nu}}$. 

Moreover (\cite{V}, 2.6 (ii)), we have an isomorphism $\Phi$ from $B(H)\otimes N$ with $\bf G^o$$\ltimes_{\widetilde{\ga}}\bf G$$\ltimes_\ga N$ which sends $\bf G$$\ltimes_\ga N$ onto $\widetilde{\ga}(\bf G$$\ltimes_\ga N)$. Via this isomorphism, the bidual action $\widetilde{\widetilde{\ga}}$ of $\bf G^{oc}$ on $\bf G^o$$\ltimes_{\widetilde{\ga}}\bf G$$\ltimes_\ga N$ gives an action $\gamma$ of $\bf G$ on $B(H_\nu)\otimes N$. As $\widetilde{\widetilde{\ga}}$ is invariant on $\widetilde{\ga}(\bf G$$\ltimes_\ga N)$, $\gamma$ is invariant on $\bf G$$\ltimes_\ga N$, and its restriction to $J_{\widetilde{\nu}}BJ_{\widetilde{\nu}}$ defines an action of $\bf G$ on $J_{\widetilde{\nu}}BJ_{\widetilde{\nu}}$, and, thanks to this restriction, we can define an action of $\bf G^c$ on $B$. 
Let's have a closer look at this last action: $\gamma$ is given, for any $X\in B(H)\otimes N$, by (\cite{V}, 2.6 (iii))
\[\gamma (X)=W^o_{12}(\varsigma\otimes \id)(\id\otimes\ga)(X)W^{o*}_{12}=\Ad [W_{12}^o(U_\nu^\ga)_{13}](X_{23}).\]

So, the opposite action of its restriction to $J_{\widetilde{\nu}}BJ_{\widetilde{\nu}}$ will be implemented by 
\begin{eqnarray*}
(J\otimes J_{\widetilde{\nu}})W^o_{12}(U^\ga_\nu)_{13}(\widehat{J}\otimes J_{\widetilde{\nu}})
&=&
(U_\nu^\ga)_{23}(J\otimes\widehat{J}\otimes J_\nu)W^o_{12}(U_\nu^\ga)_{13}(\widehat{J}\otimes \widehat{J}\otimes J_\nu)(U^\ga_\nu)^*_{23}\\
&=&
(U_\nu^\ga)_{23}(J\otimes\widehat{J}\otimes J_\nu)W^o_{12}(\widehat{J}\otimes \widehat{J}\otimes J_\nu)(U^\ga_\nu)^*_{13}(U^\ga_\nu)^*_{23}\\
&=&
(J\widehat{J})_1(U_\nu^\ga)_{23}W_{12}(U^\ga_\nu)^*_{13}(U^\ga_\nu)^*_{23}\\
&=&
(J\widehat{J})_1W_{12}
\end{eqnarray*}
So, we get an action of $\bf G^c$ on $B$ given by $z\mapsto \Ad((J\widehat{J})_1W_{12})(1\otimes z)=W^{c*}(1\otimes z)W^c$, which is $\widehat{\gb}$. Thus, $\widehat{\gb}$ is an action of $\bf G$$^c$ on $B$, and, by restriction of $(\bf G$$\ltimes_\ga N, \widetilde{\ga}, \underline{\ad})$, we have obtained that $(B, \gb, \widehat{\gb})$ is a $\bf \widehat{G}^o$-Yetter-Drinfel'd algebra. Let's now prove that it is braided-commutative. Let us write $\mathcal J(x)=J\widehat{J}x\widehat{J}J$ for any $x\in M'$. We get that $(\mathcal J\otimes id)\widehat{\gb}(B)$ is included in $M\otimes B$, and, therefore, commutes with $1\otimes\ga(N)$. On the other hand, we get that $(\mathcal J\otimes id)(\widehat{\gb}(B))=(W\otimes 1)(1\otimes B)(W^*\otimes 1)$ commutes with $(W^*\otimes 1)(\widehat{M}\otimes 1\otimes 1)(W\otimes 1)=\widehat{\Gamma}^o(\widehat{M})\otimes 1$. Therefore, we get that $(\mathcal J\otimes id)(\widehat{\gb}(B))$ commutes with $\widetilde{\ga}(\bf G$$\ltimes_\ga N)$, and, therefore, with $\gb(B)$. This finishes the proof. \end{proof}

Applying now \ref{th1} to this braided-commutative Yetter-Drinfel'd algebra, we recover the Hopf-bimodule introduced in (\cite{E2}, 14.1)

\subsubsection{\bf{Theorem}}
\label{thB}
{\it Let $\ga$ an action of a locally compact quantum group $\bf G$ on a von Neumann algebra $N$, let $B=\bf G$$\ltimes_\ga N\cap \ga(N)'$, let $\gb$ (resp. $\widehat{\gb}$) be the action of $\widehat{\bf G}^o$ (resp. $\bf G^c$) on $B$ introduced in \ref{propB}, and suppose that there exists a normal semi-finite faithful weight $\chi$ on $B$, invariant under the modular group $\sigma^{T_{\widehat{\ga}}}$. Then, $\gG(B, \gb, \widehat{\gb}, \chi)$ is a measured quantum groupoid, which is equal to the measured quantum groupoid $\gG(\ga)$ introduced in (\cite{E2}, 14.2)}

\begin{proof}
With the hypotheses, the measured quantum groupoid $\gG(\ga)$ is constructed in (\cite{E2}, 14.2); so, we get that the Hopf-bimodule constructed in \ref{propB} is a measured quantum groupoid. So, we may apply \ref{thmqg} to get that $\gG(B, \gb, \widehat{\gb}, \chi)$ is measured quantum groupoid equal to $\gG(\ga)$. \end{proof}

\subsubsection{{\bf Theorem}}
\label{thquantumgroupoidaction}
{\it Let $(N, \ga, \widehat{\ga})$ be a ${\bf G}$-Yetter-Drinfel'd algebra with a norm faithful semi-finite weight $\nu$ on $N$ satisfying the conditions of \ref{thmqg}, which allow us to construct the measured quantum groupoid $\gG(N, \ga, \widehat{\ga}, \nu)$. Suppose that $\beta(N)= {\bf G}\ltimes_\ga N\cap \ga(N)'$. Then, the weight $\nu^o\circ\beta^{-1}$ on $\beta(N)$ allows us to define the measured quantum groupoid $\gG (\ga)$, which is canonically isomorphic to $\gG (N^o, \widehat{\ga}^o, \ga^o, \nu^o)$. }

\begin{proof}
We have, for all $x\in N$ and $t\in\bf R$, $\sigma_t^{T_{\tilde{\ga}}}(\beta(x))=\beta(\gamma_t(x))$. As $\nu\circ\gamma_t=\nu$, we get that the weight $\nu^o\circ\beta^{-1}$ on $\beta(N)$ allows us to define the measured quantum groupoid $\gG(\ga)$. Moreover, the dual action $\widetilde{\ga}$ of $\widehat{{\bf G}}^o$ on ${\bf G}\ltimes_\ga N$ satisfies, for all $x\in N$, by \ref{th1}(iii),
\[\widetilde{\ga}(\beta(n))=(id\otimes\beta^{\dag})(\widehat{\ga}^o(x^o)),\]
which gives that $\beta^\dag$ is an isomorphism between $\widetilde{\ga}_{|\beta(N)}=\gb$ and $\widehat{\ga}^o$. So, the result follows. \end{proof}

We are indebted to the referee who suggested us to look at the relation betwen the construction made in (\cite{E2}, 14.2) and the measured quantum transformation groupoids  considered in this article. 

\subsection{Quotient type co-ideals}
\label{QTCI}
\subsubsection{\bf Definitions}
\label{defQTCI}
Let $\bfG=(M, \Gamma, \varphi, \varphi\circ R)$ and $\bfG_1=(M_1,
\Gamma_1, \varphi_1, \varphi_1\circ R_1)$ be two locally compact
quantum groups. Following \cite{K}, a \emph{morphism} from $\bfG$ on
$\bfG_1$ is a non-degenerate strict $*$-homomorphism $\Phi$ from
$\Cu(\bfG)$ on the multipliers $M(\Cu(\mathbf{G_1}))$ (which means that $\Phi$ extends to a unital $*$-homomorphism on
$M(\Cu(\bfG))$) such that
$\Gamma_{1,u}\circ\Phi=(\Phi\otimes\Phi)\Gamma_u$, where
$\Gamma_{1,u}$ denotes the coproduct of $\Cu(\bfG_1)$. In
(\cite{K}, 10.3 and 10.8), it was shown that a morphism is
equivalently given by a right action $\Gamma_r$ of $\bfG_1$ on $M$
satisfying, in addition to the action condition
$(\id\otimes\Gamma)\Gamma_r=(\Gamma_r\otimes \id)\Gamma_r$, also the
relation $(\Gamma\otimes \id)\Gamma_r=(\id\otimes
\Gamma_r)\Gamma$. The morphism $\Phi$ and the action $\Gamma_{r}$ are
related by the formula
\[\Gamma_r(\pi_{\bfG}(x))=(\pi_{\bfG}\otimes\pi_{\bfG_1}\circ\Phi)\Gamma_u(x)
\quad \text{for all } x\in \Cu(\bfG).\]
We get as well a left action $\Gamma_l$ of $\bfG_1$ on $M$ such that $(\id\otimes\Gamma_l)\Gamma_l=(\Gamma_1\otimes \id)\Gamma_l$ and $(\id\otimes\Gamma)\Gamma_l=(\Gamma_l\otimes \id)\Gamma$. 

Following (\cite{DKSS}, th. 3.6), we shall say that $\bfG_1$ is a \emph{closed quantum subgroup} of $\bfG$ \emph{in the sense of Woronowicz}, if, in the situation above, the $*$-homomorphism $\Phi$ is surjective. In (\cite{DKSS}, 3.3), $\bfG_1$ is called a \emph{closed quantum subgroup} of $\bfG$ \emph{in the sense of Vaes} if there exists an injective $*$-momomorphism $\gamma$ from $\widehat{M_1}$ into $\widehat{M}$ such that $\widehat{\Gamma}\circ\gamma=(\gamma\otimes\gamma)\circ\widehat{\Gamma}_{1}$. Moreover, any closed quantum subgroup of $\bfG$ in the sense of Vaes is a closed quantum subgroup in the sense of Woronowicz (\cite{DKSS}, 3.5), and if $\bfG_1$ is (the von Neumann version of) a compact quantum group, then the two notions are equivalent (\cite{DKSS}, 6.1). It is also remarked that if $\bfG$ is (the von Neumann version of) a compact quantum group, then any closed quantum subgroup of $\bfG$ is also (the von Neumann version of) a compact quantum group.

\begin{subproposition}
\label{propQTCI}
 Let $\bfG=(M, \Gamma, \varphi, \varphi\circ R)$ and  $\bfG_1=(M_1,
  \Gamma_1, \varphi_1, \varphi_1\circ R_1)$ be two locally compact
  quantum groups  and $\Phi$ a surjective morphism from $\bfG$ to
  $\bfG_1$ in the sense of \ref{defQTCI}. Let $\Gamma_r$ be the right
  action of $\bfG_1$ on $M$ defined in \ref{defQTCI}, and let  $N=M^{\Gamma_r}=\{x\in M : \Gamma_r(x)=x\otimes 1\}$. Then:
  \begin{enumerate}
  \item  $\Gamma_{|N}$ is a left action of $\bfG$ on $N$.
  \item  $\ad_{|N}$ is a left action of $\widehat{\bfG}$ on $N$.
  \item  $(N, \Gamma_{|N}, \ad_{|N})$ is a braided-commutative
    $\bfG$-Yetter-Drinfel'd algebra.
  \item Let $\Gamma_l$ be the left action of $\bfG_1$ on $M$ defined
    in \ref{defQTCI}. Then its invariant algebra $M^{\Gamma_l}$ is
    equal to $R(N)$, which is a right co-ideal of $\bfG$.
  \end{enumerate}
\end{subproposition}

In the situation above, we call $N$ a \emph{quotient type left
  co-ideal} of $\bfG$.
\begin{proof}
  (i) Since $(\id\otimes\Gamma_r)\Gamma=(\Gamma\otimes \id)\Gamma_r$
  by construction, we get that for every $x$ in $N=M^{\Gamma_r}$, the
  coproduct $\Gamma(x)$ belongs to $M\otimes N$.

  (ii) By (\cite{K} 6.6), there exists a unique unitary  $U \in M(\Cu(\bfG)
  \otimes \Cr(\mathbf{\widehat G}))$ such that $(\Gamma_{u} \otimes
  \id)(U)=U_{13}U_{23}$ and $(\pi_{\bfG} \otimes \id)(U)=W$, where
  $\Gamma_{u}$ denotes the comultiplication on $\Cu(\bfG)$. Let
  $\widehat U=\varsigma(U^{*}) \in M(\Cr(\widehat\bfG) \otimes
  \Cu(\bfG))$ and $x\in \Cu(\bfG)$. Then
  $\ad(\pi_{\bfG}(x)) = (\id \otimes \pi_{\bfG})(\widehat U^{*}(1
  \otimes x)\widehat U)$, and using the relation
 $(\id \otimes \Gamma_{u})(\widehat U^{*}) = \widehat
  U^{*}_{12}\widehat U^{*}_{13}$, we find
  \begin{align*}
    (\id \otimes \Gamma_{r})(\ad(\pi_{\bfG}(x))) &= 
    (\id \otimes \pi_{\bfG} \otimes \pi_{\bfG_{1}}\Phi)((\id
    \otimes  \Gamma_{u})(\widehat U^{*}(1 \otimes x)\widehat U)) \\
    &=
    (\id \otimes \pi_{\bfG} \otimes \pi_{\bfG_{1}}\Phi)(\widehat
    U^{*}_{12}\widehat U^{*}_{13}(1 \otimes \Gamma_{u}(x))\widehat
    U_{13}\widehat U_{12}) \\
    &= \widehat W^{*}_{12}\widetilde U^{*}_{13}(1 \otimes
    \Gamma_{r}(\pi_{\bfG}(x))) \widetilde U_{13}\widehat W_{12},
  \end{align*}
where $\widetilde U = (\id \otimes \pi_{\bfG_{1}}\Phi)(V)$. By continuity, we
get that for any $y\in N$,
  \begin{align*}
    (\id \otimes \Gamma_{r})(\ad(y)) = \widehat W^{*}_{12}\widetilde
    U^{*}_{13}(1 \otimes y \otimes 1) \widetilde U_{13}\widehat W_{12}
    = \ad(y) \otimes 1,
  \end{align*}
showing that $\ad(y) \in \widehat{M} \otimes N$.

 (iii) This follows immediately from \ref{YD}.

(iv) This follows easily from the fact that the unitary antipode
reverses the comultiplication.
\end{proof}

\begin{subtheorem}
\label{compactsubgroup}
 Let $\bfG=(M, \Gamma, \varphi, \varphi\circ R)$ be a locally
  compact quantum group and $(A_1, \Gamma_1)$  a compact quantum group
  which is a closed quantum subgroup in the sense of
  \ref{QTCI}, and denote by $N$  the quotient type co-ideal defined by
  this closed subgroup, as defined in \ref{propQTCI}. Then, the
  restriction of the weight $\varphi\circ R$ to $N$ is semi-finite and
  $\delta^{-1}$-invariant with respect to the action
  $\Gamma_{|N}$. Therefore,  $\mathfrak{G}(N, \Gamma_{|N}, \ad_{|N},
  \varphi\circ R_{|N})$ and $\mathfrak{G}(N, \ad_{|N}, \Gamma_{|N},
  \varphi\circ R_{|N})$ are measured quantum groupoids, dual to each
  other. 
\end{subtheorem}
\begin{proof}
The formula $E=(\id\otimes\omega_1)\circ\Gamma_r$, where $\omega_1$ is the Haar state of $(A_1, \Gamma_1)$, and $\Gamma_r$ is the right action of $(A_1, \Gamma_1)$ on $M$ defined in \ref{QTCI}, defines a normal faithful conditional expectation from $M$ onto $N=M^{\Gamma_r}$. 

By definition of $\Gamma_r$ (\ref{defQTCI}), and using the
right-invariance of $\varphi\circ R\circ\pi_{\bfG}$ with respect to
the coproduct $\Gamma_u$ of $\Cu(\bfG)$, we get that for any
$y\in\Cu(\bfG)$, with the notations of \ref{defQTCI},
\begin{align*}
\varphi\circ R\circ E(\pi_{\bfG}(y))
&=
(\varphi\circ R\otimes\omega_1)\Gamma_r(\pi(y))\\
&=
(\varphi\circ R\circ\pi_{\bfG}\otimes\omega_1\circ\pi_{\bfG_{1}}\circ\Phi)\Gamma_u(y)\\
&=
(\varphi\circ R\circ\pi_{\bfG})(y)(\omega_1\circ\pi_{\bfG}\circ\Phi)(1)\\
&=
(\varphi\circ R\circ\pi_{\bfG})(y).
\end{align*}
Therefore, $\varphi\circ R\circ E(x)=\varphi\circ R(x)$ for all
$x\in\Cr(\bfG)$, and, by continuity, for all $x\in M$, which
gives that this conditional expectation $E$ is invariant under
$\varphi\circ R$. Moreover, we get that $\varphi\circ R_{|N}$ is
semi-finite and $\sigma_t^{\varphi\circ R}\circ
E=E\circ\sigma_t^{\varphi\circ R}$. 

This weight $\varphi\circ R_{|N}$ is clearly $\delta^{-1}$-invariant with respect to $\Gamma_{|N}$. The result comes then from \ref{thmqg} and \ref{ovw}. 
\end{proof}

\begin{subcorollary}
\label{NY}
 Let $(A, \Gamma)$ be a compact quantum group, $\omega$ its Haar state (which we can suppose to be faithful) and let $\bfG=(\pi_\omega(A)'', \Gamma, \omega, \omega)$ be the von Neumann version of $(A, \Gamma)$ (\ref{lcqg}). Let $N$ be a sub-von Neumann algebra $N$ of $\pi_\omega(A)''$. Then the following conditions are equivalent:
 \begin{enumerate}
 \item  $\Gamma_{|N}$ is a left action of $\bfG$ on $N$ and
   $\ad_{|N}$ is a left action of $\widehat{\bfG}$ on $N$.
 \item  There exists a quantum compact subgroup of $(A, \Gamma)$ such
   that $N$ is the quotient type co-ideal of $\bfG$ constructed from
   this quantum compact subgroup.
 \end{enumerate}
If (i) and (ii) hold, then  the crossed products $\bfG\ltimes_{\Gamma_{|N}}N$ and
$\widehat{\bfG}\ltimes_{\ad_{|N}}N$ carry mutually dual structures of
measured quantum groupoids $\mathfrak{G}(N, \Gamma_{|N}, \ad_{|N},
\omega_{|N})$ and $\mathfrak{G}(N, \ad_{|N}, \Gamma_{|N},
\omega_{|N})$, respectively.
\end{subcorollary}
\begin{proof}
The fact that (ii) implies (i) is given by \ref{QTCI}.  Suppose
(i). Then $N$ is, by \ref{propQTCI}, a quotient type co-ideal of $\bfG$, which is defined as the invariants by a right action $\Gamma_r$ of a closed quantum subgroup of $\bfG$, which is (\ref{defQTCI}) a compact quantum group $(A_1, \Gamma_1)$. Denote its Haar state by $\omega_1$. Then $\Gamma_r(A)\subset A\otimes A_1$, and the conditional expectation $E=(id\otimes\omega_1)\Gamma_r$ which sends $\pi_\omega(A)''$ onto $N$, sends $A$ onto $A\cap N$. From this it is easy to get that $A\cap N$ is weakly dense in $N$. But $N\cap A$ is a sub-$C^*$-algebra of $A$ which is invariant
under $\Gamma$ and $\ad$; therefore, using (\cite{NY}, Th. 3.1), we
get (ii). 
If these conditions hold, we can apply \ref{compactsubgroup}.
\end{proof}

\subsubsection{\bf Example of a measured quantum groupoid constructed from a quotient type coideal of a compact quantum group} 
\label{Sq2}

Let us take the compact quantum group $\SUq$ (\cite{W2}), which is the $C^*$-algebra generated by elements $\alpha$ and $\gamma$ satisfying the relations 
\[\alpha^*\alpha +\gamma^*\gamma=1, \qquad \alpha\alpha^*+q^2\gamma\gamma^*=1,\]
\[\gamma\gamma^*=\gamma^*\gamma, \qquad q\gamma\alpha=\alpha\gamma,
\qquad q\gamma^*\alpha=\alpha\gamma^*.\]
The circle group $\mathds{T}$ appears as a closed quantum subgroup via
the  morphism $\Phi$ from $\Cu(\SUq)$ to
$\Cu(\T)=C_{0}(\T)$ given by $\Phi(\alpha)=0$ and
$\Phi(\gamma)=\id$.  Then we obtain the Podle\'{s} sphere $S_q^2$ as a
quotient type coideal from this map (\cite{P}),  and  mutually dual structures of measured quantum groupoids
$\mathfrak{G}(S_q^2, \Gamma_{|S_q^2}, \ad_{|S_q^2}, \omega_{|S_q^2})$
on $\SUq\ltimes_{\Gamma_{|S_q^2}}S_q^2$ and $\mathfrak{G}(S_q^2,
\ad_{|S_q^2}, \Gamma_{|S_q^2}, \omega_{|S_q^2})$ on
$\widehat{\SUq}\ltimes_{\ad_{|S_q^2}}S_q^2$, respectively.

\subsubsection{\bf Further examples}
\label{ex}

Here we quickly give examples of situations in which the hypothesis of \ref{compactsubgroup} are fulfilled. 

Let us consider the (non-compact) quantum group $E_q(2)$ constructed
by Woronowicz in \cite{W3}. In (\cite{J}, 2.8.36) is proved that the
circle group $\mathds{T}$ is a closed quantum subgroup of $E_q(2)$. 

In \cite{VV2} is constructed the cocycle bicrossed product of two
locally compact quantum groups $(M_1, \Gamma_1)$ and $(M_2,
\Gamma_2)$, and it is proved (\cite{VV2}, 3.5) that $(\widehat{M_1},
\widehat{\Gamma_1})$ is a closed subgroup (in the sense of Vaes) of
$(M, \Gamma)$. So, if $(M_1, \Gamma_1)$ is a discrete quantum group,
then $(\widehat{M_1}, \widehat{\Gamma_1})$ is the von Neumann version of a compact quantum group which is a closed quantum subgroup of $(M, \Gamma)$.

\subsection{Another example}

\begin{subtheorem}
\label{HG}
 Let $G$ be a locally compact group and $H$ a closed normal subgroup of $G$. Then:

\begin{enumerate}
\item The von Neumann algebra $\mathcal L(H)$, which can be considered as a sub-von Neumann algebra of $\mathcal L(G)$, is invariant under the coproduct $\Gamma_G$ of $\mathcal L(G)$, considered as a right action of the locally compact quantum group $\widehat{G}$ on $\mathcal L(G)$, and under the adjoint action $\ad$ of $G$ on $\mathcal L(G)$. Therefore, $(\mathcal L(H), \Gamma_{G|\mathcal L(H)}, \ad_{|\mathcal L(H)})$ is a braided-commutative $\widehat{G}$-Yetter-Drinfel'd algebra, which is a subalgebra of the canonical example $(\mathcal L(G), \Gamma_G, \ad)$ described in \ref{basicex}. 
\item  The Plancherel weight $\varphi_H$ on $\mathcal L(H)$
  satisfies the conditions of \ref{thmqg}, and the crossed product
  $\widehat{G}\ltimes_{\Gamma_{G|\mathcal L(H)}}\mathcal L(H)$ (which
  is isomorphic to $(\mathcal L(H)\cup L^\infty(G))''$) carries a
  structure of measured quantum groupoid $\mathfrak{G}(\mathcal L(H),
  \Gamma_{G|\mathcal L (H)}, \ad_{|\mathcal L(H)}, \varphi_H)$ over
  the basis $\mathcal L(H)$.
\end{enumerate}
\end{subtheorem}
\begin{proof}
(i) Let $\lambda_G$ (resp. $\lambda_H$) be the left regular representation of $G$ (resp.{} $H$). It is well known that the application which sends $\lambda_H(s)$ to $\lambda_G(s)$, where  $s\in H$, extends to an injection from $\mathcal L(H)$ into $\mathcal L(G)$, which will send the coproduct $\Gamma_H$ of $\mathcal L(H)$ on the coproduct $\Gamma_G$ of $\mathcal L(G)$. Let us identify $\mathcal L(H)$ with this sub-von Neumann algebra of $\mathcal L(G)$. Then for all $x\in\mathcal L(H)$,
\[\Gamma_G(x)=\Gamma_H(x)\in\mathcal L(H)\otimes\mathcal L(H)\subset \mathcal L(G)\otimes\mathcal L(H),\]
so that the coproduct, considered as a right action of $\widehat{G}$ on $\mathcal L(G)$, gives also a right action of $\widehat{G}$ on $\mathcal L(H)$. 

Let $W_G$ be the fundamental unitary of $G$, which belongs to
$L^\infty(G)\otimes\mathcal L(G)$. The adjoint action of $G$ on
$\mathcal L(G)$ is given, for $x\in\mathcal L(G)$ by
$\ad(x)=W_G^*(1\otimes x)W_G$, and is therefore the function on $G$
given by $s\mapsto \lambda_G(s)x\lambda_G(s)^*$. Hence, if $t\in H$, we get that $\ad(\lambda_H(s))$ is the function $s\mapsto \lambda_G(sts^{-1})$. As $H$ is normal, $sts^{-1}$ belongs to $H$, and this function takes its values in $\mathcal L(H)$. By density, we get that for any $x\in\mathcal L(H)$, $\ad(x)$ belongs to $L^\infty(G)\otimes\mathcal L(H)$, and, therefore, the restriction of the adjoint action of $G$ to $\mathcal L(H)$ is an action of $G$ on $\mathcal L(H)$.

(ii) The Haar weight $\varphi_H$ is invariant under
$\Gamma_{G|\mathcal L(H)}$ because
$(\id\otimes\varphi_H)(\Gamma_G(x))=(\id\otimes\varphi_H)(\Gamma_H(x))=\varphi_H(x)1$
for all $x\in\mathcal L(H)^+$.
 We can therefore apply \ref{thmqg} to that braided-commutative Yetter-Drinfel'd algebra, equipped with this relatively invariant weight, and get (ii). Let us remark that $\widehat{G}\ltimes_{\Gamma_{G|\mathcal L(H)}}\mathcal L(H)$ is equal to $(\Gamma_G(\mathcal L(H))\cup L^\infty (G)\otimes 1_{L^2(G)})''$ which we can write:
\[((J\otimes J)W_G^*(J\otimes J)(\mathcal L(H)\otimes 1_{L^2(G)})(J\otimes J)W_G(J\otimes J)\cup L^\infty (G)\otimes 1_{L^2(G)})''\]
which is clearly isomorphic to $(\mathcal L(H)\cup L^\infty (G))''$. 
\end{proof}
\begin{subremark}
\label{remHG}
Let us take again the hypotheses of \ref{HG}, in the particular case
where $G$ is abelian. Then $\widehat{G}$ (resp. $\widehat{H}$) is a
commutative locally compact group, and we have constructed a right
action of $\widehat{G}$ on the set $\widehat{H}$, which leads to a
transformation groupoid $\widehat{H}\curvearrowleft
\widehat{G}$. Then, the measured quantum groupoid constructed in
\ref{HG}(ii) is just the dual of this transformation groupoid. 
\end{subremark}

\section{Quotient type co-ideals and Morita equivalence}
\label{QM}
In this chapter, we show that, in the case of a quotient type co-ideal $N$ of a compact quantum group $\bfG$, the measured quantum groupoid $\widehat{\bfG}\ltimes_{\ad|N}N$ is Morita equivalent to the quantum subgroup $\bfG_1$ (\ref{ThMorita}).

\subsection{Definitions of actions of a measured quantum groupoid and Morita equivalence}
\subsubsection{{\bf Definition}} ([E5] 2.4)
\label{right}
Let $\gG=(N, M, \alpha, \beta, \Gamma, T, T', \nu)$ be a measured quantum groupoid, and let $A$ be a von Neumann algebra. 

A \emph{right action} of $\gG$ on $A$ is a couple $(b, \underline{\mathfrak a})$, where:

(i) $b$ is an injective anti-$*$-homomorphism from $N$ into $A$; 

(ii) $\underline{\mathfrak a}$ is an injective $*$-homomorphism from $A$ into $A\underset{N}{_b*_\alpha}M$; 

(iii) $b$ and $\underline{\mathfrak a}$ satisfy
\[\underline{\mathfrak a}
(b(n))=1\underset{N}{_b\otimes_\alpha}\beta(n) \quad \text{for all }
n\in N,\]
which allow us to define $\underline{\mathfrak a}\underset{N}{_b*_\alpha}\id$ from $A\underset{N}{_b*_\alpha}M$ into $A\underset{N}{_b*_\alpha}M\underset{N}{_\beta*_\alpha}M$,
and
\[(\underline{\mathfrak a}\underset{N}{_b*_\alpha}\id)\underline{\mathfrak a}=(\id\underset{N}{_b*_\alpha}\Gamma)\underline{\mathfrak a}.\]
If there is no ambiguity, we shall say that $\underline{\mathfrak a}$ is the right action. 

So, a measured quantum groupoid $\gG$ can act only on a von Neumann algebra $A$ which is a right module over the basis $N$.

Moreover, if $M$ is abelian, then $\underline{\mathfrak a}
(b(n))=1\underset{N}{_b\otimes_\alpha}\beta(n)$ commutes with
$\underline{\mathfrak a} (x)$ for all $n  \in N$ and $x\in A$, so that $b(N)$ is in the center of $A$. As in that case (\ref{defMQG}) the measured quantum groupoid comes from a measured groupoid $\mathcal G$, we have $N=L^\infty(\mathcal G^{(0)}, \nu)$, and $A$ can be decomposed as $A=\int_{\mathcal G^{(0)}}A^xd\nu (x)$. 

The invariant subalgebra $A^{\underline{\ga}}$ is defined by
\[A^{\underline{\ga}}=\{x\in A\cap b(N)' :
\underline{\ga}(x)=x\underset{N}{_b\otimes_\alpha}1\}.\] As
$A^{\underline{\ga}}\subset b(N)'$, $A$ (and $L^2(A)$) is a
$A^\ga$-$N^{\op}$-bimodule. If $A^{\underline{\ga}}=\mathbb{C}$, the
action $(b, \underline{\ga})$ (or, simply $\underline{\ga}$) is called
\emph{ergodic}.

Let us write, for any $x\in A^+$,
$T_{\underline{\ga}}(x)=(\id\underset{\nu}{_b*_\alpha}\Phi){\underline{\ga}}(x)$. This
formula defines a normal faithful operator-valued weight from $A$ onto
$A^{\underline{\ga}}$, and the action $\underline{\ga}$ will be called \emph{integrable} if $T_{\underline{\ga}}$ is semi-finite (\cite{E4}, 6.11, 12, 13 and 14). 

The \emph{crossed product} of $A$ by $\mathfrak {G}$ via the action
$\underline{\ga}$ is the von Neumann algebra generated by
$\underline{\ga}(A)$ and $1\underset{N}{_b\otimes_\alpha}\widehat{M}'$
(\cite{E2}, 9.1) and is denoted by
$A\rtimes_{\underline{\ga}}\mathfrak {G}$. There exists (\cite{E2},
9.3) an integrable dual action $(1\underset{N}{_b\otimes_\alpha}\widehat{\alpha}, \tilde{\underline{\ga}})$ of $(\widehat{\mathfrak {G}})^{\com}$ on $A\rtimes_{\underline{\ga}}\mathfrak {G}$. 

We have
$(A\rtimes_{\underline{\ga}}\gG)^{\tilde{\underline{\ga}}}=\ga(A)$
(\cite{E2} 11.5), and, therefore, the normal faithful semi-finite
operator-valued weight $T_{\tilde{\underline{\ga}}}$ sends
$A\rtimes_\mathfrak a\mathfrak {G}$ onto $\ga(A)$. Starting with a
normal semi-finite weight $\psi$ on $A$, we can thus  construct a \emph{dual weight} $\tilde{\psi}$ on $A\rtimes_{\underline{\ga}}\mathfrak {G}$ by the formula $\tilde{\psi}=\psi\circ\underline{\ga}^{-1}\circ T_{\tilde{\underline{\ga}}}$ (\cite{E4} 13.2). 

Moreover (\cite{E2} 13.3), the linear set generated by all the
elements $(1\underset{N}{_b\otimes_\alpha}a)\underline{\ga}(x)$, where
$x\in\gN_\psi$ and
$a\in\gN_{\widehat{\Phi}^{\com}}\cap\gN_{\widehat{T}^{\com}}$, is a
core for $\Lambda_{\tilde{\psi}}$, and one can identify the GNS representation of $A\rtimes_{\underline{\ga}}\gG$ associated to the weight $\tilde{\psi}$ with the natural representation on $H_\psi\underset{\nu}{_b\otimes_\alpha}H$ by writing
\[\Lambda_{\tilde{\psi}}[(1\underset{N}{_b\otimes_\alpha}a)\underline{\ga}(x)]=\Lambda_\psi(x)\underset{\nu}{_b\otimes_\alpha}\Lambda_{\widehat{\Phi}^{\com}}(a),\]
which leads to the identification of $H_{\tilde{\psi}}$ with $H_\psi\underset{\nu}{_b\otimes_\alpha}H$. 

Let us suppose now that the action $\underline{\ga}$ is integrable.
Let $\psi_0$ be a normal semi-finite weight on $A^{\underline{\ga}}$,
and let us write $\psi_1=\psi_0\circ T_{\underline{\ga}}$.  If we
write
$V=J_{\tilde{\psi_1}}(J_{\psi_1}\underset{N^{\op}}{_a\otimes_\beta}J_{\widehat{\Phi}})$,
we get a representation of $\gG$ which implements $\ga$ and which we
shall call the \emph{standard implementation} of $\ga$ (\cite{E5}, 3.2
and \cite{E4} 8.6).

Moreover, there exists then a canonical isometry $G$ from $H_{\psi_1}\underset{\psi_0}{_s\otimes_{r}}H_{\psi_1}$ into $H_{\psi_1}\underset{\nu}{_b\otimes_\alpha}H$ such that, for any $x\in \gN_{T_{\ga}}\cap\gN_{\psi_1}$, $\zeta\in D((H_{\psi_1})_b, \nu^{\op})$ and $e$ in $\gN_\Phi$, 
\[(1\underset{N}{_b\otimes_\alpha}J_\Phi eJ_\Phi)G(\Lambda_{\psi_1}(x)\underset{\psi_0}{_s\otimes_r}\zeta)=\ga(x)(\zeta\underset{\nu}{_b\otimes_\alpha}J_\Phi\Lambda_\Phi(e)),\]
where $r$ is the canonical injection of $A^{\underline{\ga}}$ into
$A$, and $s(x)=J_{\psi_1}x^*J_{\psi_1}$  for all $x\in A^{\underline{\ga}}$. There exists a surjective $*$-homomorphism $\pi_{\underline{\ga}}$ from the crossed product $(A\rtimes_{\underline{\ga}}\gG)$ onto $s(A^{\underline{\ga}})'$, defined, for all $X$ in $A\rtimes_\ga\gG$ by $\pi_{\underline{\ga}}(X)\underset{A^{\ga}}{_s\otimes_r}1=G^*XG$. It should be noted that this algebra $s(A^{\underline{\ga}})'$ is the basic construction for the inclusion $A^{\underline{\ga}}\subseteq A$. (\cite{E5}, 3.6)
If the operator $G$ is
 unitary (or, equivalently,  the $*$-homomorphism
 $\pi_{\underline{\ga}}$ is an isomorphism), then the action
 $\underline{\ga}$   is called a  \emph{Galois action} (\cite{E5}
 3.11) and the unitary $\widetilde{G}=\sigma_\nu G$  its \emph{Galois unitary}.

\subsubsection{{\bf Definition}}([E4] 6.1)
\label{left}
A \emph{left action} of $\gG$ on a von Neumann algebra $A$ is a couple
$(a, \underline{\gb})$, where

(i) $a$ is an injective *-homomorphism from $N$ into $A$;

(ii) $\underline{\gb}$ is an injective $*$-homomorphism from $A$ into $M\underset{N}{_\beta*_a}A$; 

(iii) $\underline{\gb}(a(n))=\alpha(n)\underset{N}{_\beta\otimes_a}1$
for all $n\in N$, and
$(\id\underset{N}{_\beta*_a}\underline{\gb})\underline{\gb} =
(\Gamma\underset{N}{_\beta*_a}\id)\underline{\gb}$.

Then, it is clear that $(a, \varsigma_N\underline{\gb})$ is a right
action of $\gG^{\op}$ on $A$. Conversely, if $(b, \underline{\ga})$ is a left action of $\gG$ on $A$, then, $(b, \varsigma_{N}\underline{\ga})$ is a left action of $\gG^{\op}$ on $A$. 

The invariant subalgebra $A^{\underline{\gb}}$ is  defined by
\[A^{\underline{\gb}}=\{x\in A\cap a(N)'
:\underline{\gb}(x)=1\underset{N}{_\beta\otimes_a}x\},\] and
$T_{\underline{\gb}}=(\Phi\circ
R\underset{\nu}{_\beta*_a}\id)\underline{\gb}$ is a normal faithful
operator-valued weight from $A$ onto $A^{\underline{\gb}}$. The action
$\underline{\gb}$ will be called \emph{integrable} if
$T_{\underline{\gb}}$ is semi-finite. It is clear that
$\underline{\gb}$ is integrable if and only if
$\varsigma_N{\underline{\gb}}$ is integrable, and \emph{Galois} if and only
if $\varsigma_N{\underline{\gb}}$ is Galois.
\subsubsection{{\bf Definition}} ([E5] 2.4)
\label{commuting}
Let $(b, \underline{\ga})$ be a right action of $\gG_1=(N_1, M_1,
\alpha_1, \beta_1, \Gamma_1, T_1, T'_1, \nu_1)$ on a von Neumann
algebra $A$ and $(a, \underline{\gb})$ a left action of $\gG_2=(N_2,
M_2, \alpha_2, \beta_2, \Gamma_2, T_2, T'_2, \nu_2)$ on $A$ such that
$a(N_2)\subset b(N_1)'$ We shall say that the actions
$\underline{\ga}$ and $\underline{\gb}$ \emph{commute} if
\begin{align*}
b(N_1)&\subseteq A^{\underline{\gb}}, &
a(N_2)&\subseteq A^{\underline{\ga}}, &
(\underline{\gb}\underset{N_1}{_b*_{\alpha_1}}\id)\underline{\ga}&=(\id\underset{N_2}{_{\beta_2}*_a}\underline{\ga})\underline{\gb}.  
\end{align*}
Let us remark that the  first two properties allow us to write the fiber products $\underline{\gb}\underset{N_1}{_b*_{\alpha_1}}\id$ and $\id\underset{N_2}{_{\beta_2}*_a}\underline{\ga}$. 
\subsubsection{{\bf Definition}} ([E5] 6.5)
\label{Morita}
For $i=1,2$, let $\gG_i=(N_i, M_i, \alpha_i, \beta_i, T_i, T'_i,
\nu_i)$ be a measured quantum groupoid. We shall say that $\gG_1$ is
\emph{Morita equivalent} to $\gG_2$ if there exists a von Neumann
algebra $A$, a Galois right action $(b, \ga)$ of $\gG_1$ on $A$, and a Galois left action $(a, \gb)$ of $\gG_2$ on $A$ such that

(i) $A^\ga=a(N_2)$, $A^\gb=b(N_1)$, and the actions $(b, \ga)$ and $(a, \gb)$ commute; 

(ii) the modular automorphism groups of the normal semi-finite faithful weights $\nu_1\circ b^{-1}\circ T_\gb$ and $\nu_2\circ a^{-1}\circ T_\ga$ commute. 

Then $A$ (or, more precisely, $(A, b, \ga, a, \gb)$) will be called an
\emph{imprimitivity bi-comodule} for $\gG_1$ and $\gG_2$. 

\begin{proposition}
\label{qgaction}
 Let $N$ be a quotient type co-ideal of (the von Neumann version
  of a) compact quantum group $\bfG=(M, \Gamma, \omega, \omega)$, and
  let us consider the measured quantum groupoid $\gG (N, \ad_{|N},
  \Gamma_{|N}, \omega_{|N})$ constructed in \ref{NY}. 
  \begin{enumerate}
  \item  There exists a unitary $V_4$ from
    $H\underset{\omega_{|N}}{_{R_{|N}}\otimes_{\ad_{|N}}}(H\otimes
    H_{\omega_{|N}})$ onto $H\otimes H$ such that
    \[V_4(\xi\underset{\omega_{|N}}{_{R_{|N}}\otimes_{\ad_{|N}}}
    U^{\ad_{|N}}_{\omega_{|N}} (\eta\otimes J_{\omega_{|N}}
    \Lambda_{\omega_{|N}}(x^{*})))=R(x)\xi\otimes\eta\] for all $x\in
    N$ and $\xi$, $\eta$ in $H$.  Moreover, for all $z\in R(N)'$, $x
    \in N$ and $y\in B(H)$,
    \begin{align*}
      V_4(z\underset{N}{_{R_{|N}}\otimes_{\ad_{|N}}}1) &= (z\otimes
      1_H)V_4, \\
      V_4(1_H\underset{N}{_{R_{|N}}\otimes_{\ad_{|N}}}U^{\ad_{|N}}_{\omega_{|N}}(y\otimes
      x^{\op})(U^{\ad_{|N}}_{\omega_{|N}})^*)&= (R(x) \otimes y)V_{4}.
    \end{align*}
  \item  Let $y\in M$ and $\underline{\ga}(y)=V_4^*\Gamma(y)V_4$. Then
    $\underline{\ga}(y)$ belongs to
    $M\underset{N}{_{R_{|N}}*_{\ad_{|N}}}
    (\widehat{\bfG}\ltimes_{\ad_{|N}}N)$.
  \item  Let $x\in N$. Then
    $\underline{\ga}(R(x))=1\underset{N}{_{R_{|N}}\otimes_{\ad_{|N}}}\widehat{\beta}(x)$,
    where $\widehat{\beta}$ is the canonical anti-representation of
    the basis $N$ into $\widehat{\bfG}\ltimes_{\ad_{|N}}N$.
  \item  $(R_{|N}, \underline{\ga})$ is a right action of $\gG (N,
    \ad_{|N}, \Gamma_{|N}, \omega_{|N})$ on $M$.
  \item  The action $\underline{\ga}$ is ergodic, and integrable. More
    precisely, the canonical operator-valued weight
    $T_{\underline{\ga}}$ is equal to the Haar state $\omega$.
  \item  The action $\underline{\ga}$ is Galois and its Galois unitary
    is $V_4^*W^*\sigma$.
  \end{enumerate}
\end{proposition}
\begin{proof}
(i) By \ref{propV}(i) applied to the  braided-commutative $\widehat{\bfG}$-Yetter-Drinfel'd algebra $(N, \ad_{|N}, \Gamma_{|N})$, we get that $U^{\ad_{|N}}_{\omega_{|N}} (\eta\otimes J_{\omega_{|N}}\Lambda_{\omega_{|N}}(x^{*}))$ belongs to $D((H\otimes H_{\omega_{|N}})_{\ad_{|N}}, \omega_{|N})$ and that \[R^{\ad_{|N}, \omega_{|N}}(U^{\ad_{|N}}_{\omega_{|N}} (\eta\otimes J_{\omega_{|N}} \Lambda_{\omega_{|N}}(x^{*})))=U^{\ad_{|N}}_{\omega_{|N}}l_\eta J_{\omega_{|N}}x^*J_{\omega_{|N}}.\]
Therefore, using standard arguments, we get an isometry $V_4$ given by
the formula above. As its image is trivially dense in $H\otimes H$, we
get that $V_4$ is unitary. The commutation relations are straightforward.

(ii) Thanks to the commutation property in (i), $\underline{\ga}(y)$
belongs to $M\underset{N}{_{R_{|N}}*_{\ad_{|N}}}B(H\otimes
H_{\omega_{|N}})$.

By \ref{BC}(i), 
  \[(\widehat{\bfG}\ltimes_{\ad_{|N}} N)'= U^{\ad_{|N}}_{\omega_{|N}}(\widehat{\bfG}^{\op}
    \ltimes_{\ad_{|N}^{\op}} N^{\op})(U^{\ad_{|N}}_{\omega_{|N}})^{*} 
    =  U^{\ad_{|N}}_{\omega_{|N}} (M' \otimes
1 \cup \ad_{|N}^{\op}(N^{\op}))''(U^{\ad_{|N}}_{\omega_{|N}})^{*}.\]
 On one hand, the commutation relations in (i) imply
\[  1 \underset{\omega_{|N}}{_{R_{|N}}\otimes_{\ad_{|N}}}  U^{\ad_{|N}}_{\omega_{|N}}(M' \otimes
1)(U^{\ad_{|N}}_{\omega_{|N}})^{*}=V_4^*(1_H\otimes M')V_4,\]
which evidently commutes with $\underline{\ga}(M)=V_4^{*}\Gamma(M)V_4$. 
On the other hand, if 
$z \in \widehat M$ and $x\in N$, then
\[V_4(1 \underset{\omega_{|N}}{_{R_{|N}}\otimes_{\ad_{|N}}}  U^{\ad_{|N}}_{\omega_{|N}}(z \otimes
x^{\op})( U^{\ad_{|N}}_{\omega_{|N}})^{*})V_4^{*} = \widehat Jx^{*}\widehat J \otimes z=
(\widehat J J \otimes 1)\sigma (z \otimes x^{\op})\sigma (J\widehat J \otimes 1)\]
and hence
\begin{align*}
V_4(1 \underset{\omega_{|N}}{_{R_{|N}}\otimes_{\ad_{|N}}}  U^{\ad_{|N}}_{\omega_{|N}}\ad_{|N}^{\op}(N^{\op})( U^{\ad_{|N}}_{\omega_{|N}})^{*})V_4^{*}
&=
(\widehat J J \otimes 1)\sigma \ad_{|N}^{\op}(N^{\op})\sigma (J\widehat J \otimes 1)\\
&=
(\widehat{J}\otimes J)\sigma \ad_{|N}(N)\sigma (\widehat{J}\otimes J)\\
&=
(\widehat{J}\otimes J)W(N\otimes 1_H)W^*(\widehat{J}\otimes J)\\
&=
W^*(R(N)\otimes 1_H)W
\end{align*}
which commutes with $\Gamma(M)=W^*(1_H\otimes M)W$. 

Therefore,  $\underline{\ga}(y)$ commutes with $1\underset{N}{_{R_{|N}}\otimes_{\ad_{|N}}}
(\widehat{\bfG}\ltimes_{\ad_{|N}} N)'$.

(iii)
 Using \ref{ThBC} applied to $(\widehat{\bfG}, \ad_{|N},
 \Gamma_{|N})$, we get that
 $\widehat{\beta}(x)=U^{\ad_{|N}}_{\omega_{|N}}\alpha^{\op}(x^{\op})(U^{\ad_{|N}}_{\omega_{|N}})^*$,
 where we write $\alpha=\Gamma_{|N}$ and
 $\alpha^{\op}(x^{\op})=(R\otimes .^{\op})\Gamma (x)\in M\otimes
 N^{\op}$ to avoid confusion with $\Gamma^{\op}$.  Then the
 commutation relations in (i)  imply that
  \begin{align*}
    V_4(1_H\underset{N}{_{R_{|N}}\otimes_{\ad_{|N}}}\widehat{\beta}(x))V_4^*&=
    V_4(1_H\underset{N}{_{R_{|N}}\otimes_{\ad_{|N}}}U^{\ad_{|N}}_{\omega_{|N}}\alpha^{\op}(x^{\op})(U^{\ad_{|N}}_{\omega_{|N}})^*)V_4^* 
  \end{align*}
is equal to $
\varsigma(R \otimes R)(\Gamma(x)) = \Gamma(R(x))
= V_4\underline{\ga}(R(x))V_4^*$.

(iv) Let us first fix notation. We denote by
\begin{align*}
 \varsigma \circ \widetilde{\ad_{|N}} \colon \widehat
 {\bfG}\ltimes_{\ad_{|N}} N \to  (\widehat
{\bfG}\ltimes_{\ad_{|N}} N) \otimes M
\end{align*}
the dual action followed by the flip. Standard arguments show that
there exists a unitary
\begin{align*}
  V_5 \colon (H \otimes H_{\omega_{|N}}) \rtensor{\widehat\beta}{\omega_{|N} }{\ad_{|N}} (H
  \otimes H_{\omega_{|N}}) \to   H \otimes H_{\omega_{|N}} \otimes H
\end{align*}
such that
\begin{align*}
  V_5(\Xi \rtensor{ \widehat\beta}{\omega_{|N}}{\ad_{|N}}  U^{\ad_{|N}}_{\omega_{|N}}(\eta \otimes
  \Lambda_{\omega_{|N}}(x^{*}))) &=  \widehat{\beta}(x)\Xi \otimes \eta
\end{align*}
for all $\Xi \in H\otimes H_{\omega_{|N}}$, $\eta \in H$, $x \in N$.

We need to prove commutativity of the following diagram,
\begin{align}  \tag{*}
  \xymatrix@C=25pt@R=25pt{
    M \ar[r]^(0.4){\Gamma} \ar@{}[rd]|{(1)} \ar[d]_{\Gamma} & M\otimes M
    \ar[r]^(0.4){\ad_{V_4^{*}}}  \ar[d]^{\Gamma \otimes \id} \ar@{}[rd]|{(2)} & M \rast{R_{}}{N}{\ad_{}} (\widehat {\bfG}
   \ltimes N)  \ar[d]^{\Gamma \ast \id}\\
   M \otimes M \ar[r]_(0.4){\id \otimes \Gamma} \ar[d]_{\ad_{V_4^{*}}} \ar@{}[rd]|{(3)}& M \otimes M \otimes M
   \ar[r]_(0.4){\id \otimes \ad_{V_4^{*}}} \ar[d]^{\ad_{V_4^{*}}\otimes \id}
   \ar@{}[rd]|{(4)} &
   (M\otimes M) \rast{(\Gamma\circ R_{})}{N}{\ad_{}} (\widehat{ \bfG} \ltimes N) \ar[d]^{\ad_{V_4^{*}} \ast \id}
 \\
 M \rast{R_{}}{N}{\ad_{}} (\widehat{ \bfG} \ltimes_{\ad_{}} N) \ar[r]_(0.4){\id \ast
   \varsigma \widetilde{\ad_{}}} & M \rast{R_{}}{N}{(\ad_{} \otimes 1)} ((\widehat
{ \bfG}\ltimes_{\ad_{}} N) \otimes M) \ar[r]_(0.48){\id \ast \ad_{V_5^{*}}} & M \rast{R_{}}{N}{\ad_{}} (\widehat {\bfG}\ltimes_{\ad_{}}
 N) \rast{\widehat{\beta}}{N}{\ad_{}} (\widehat \bfG \ltimes N),
}
\end{align}
where we dropped the subscripts
from $R$ and $\ad$. 

Commutativity of cells (1) and (2) is evident or easy.  

Let us show that cell
(3) commutes.  By definition,
\begin{align*}
  (\varsigma \circ \widetilde{\ad_{|N}})(X) = \widehat W_{13}^{\com}(X \otimes 1)(\widehat W_{13}^{\com})^{*}
\end{align*}
for all $X \in \widehat{\bfG} \ltimes_{\ad_{|N}} N$, where
\begin{align*}
  \widehat W^{\com} = (\widehat J \otimes \widehat J)\widehat W(\widehat J \otimes \widehat J)
  \in \widehat M' \otimes M,
\end{align*}
and $\Gamma(x) = \widehat W^{\com}(x \otimes 1)\widehat W^{\com}$ for all
$x\in M$. Therefore, 
\begin{align} \label{eq:cell3-1}
  (\ad_{V_4^{*}} \otimes \id)((\id \otimes \Gamma)(Y))&= \ad_{(V_4^{*}
  \otimes 1_{H})(1_{H} \otimes \widehat W^{\com})}(Y \otimes 1), \\
 \label{eq:cell3-2}
(\id \ast \varsigma \circ \widetilde{\ad_{|N}})(\ad_{V_4^{*}}(Y)) &=
 \ad_{(1 \rtensor{R_{|N}}{\omega_{|N}}{(\ad_{|N}\otimes 1)}\widehat W^{\com}_{13})(V_4^{*}\otimes
 1_{H})}(Y\otimes 1)
\end{align}
for all $Y\in M\otimes M$.
 To prove that the two expressions coincide,
it  suffices to show that the following diagram (**) commutes:
\begin{align*} \tag{**}
  \xymatrix@C=80pt{
    H \rtensor{R_{|N}}{\omega_{|N}}{\ad_{|N}}(H \otimes H_{\omega_{|N}}) \otimes H
    \ar[r]^{1 \rtensor{R_{|N}}{\omega_{|N}}{(\ad_{|N}\otimes 1)} \widehat W^{\com}_{13}}
    \ar[d]_{V_4\otimes 1_{H}}& 
    H \rtensor{R_{|N}}{\omega_{|N}}{\ad_{|N}}(H \otimes H_{\omega_{|N}}) \otimes H
         \ar[d]^{V_4\otimes 1_{H}} \\
         H\otimes H\otimes H \ar[r]_{\widehat W_{23}^{\com}} & H\otimes
         H\otimes H
  }
\end{align*}
 But since the first legs of $U^{\ad_{|N}}_{\omega_{|N}} \in \widehat M \otimes
 B(H_{\omega_{|N}})$ and $\widehat W^{\com} \in (\widehat M)' \otimes M$ commute, 
\begin{multline*}
  (V_4 \otimes 1_{H})(1 \rtensor{R_{|N}}{\omega_{|N}}{(\ad_{|N}\otimes 1)} \widehat W^{\com}_{13})(\xi
  \rtensor{R_{|N}}{\omega_{|N}}{\ad_{|N}} U^{\ad_{|N}}_{\omega_{|N}}(\eta \otimes
x^{\op}\Lambda_{\omega_{|N}}(1)) \otimes \vartheta)=\\
=(V_4 \otimes 1_{H})(\xi \rtensor{R_{|N}}{\omega_{|N}}{(\ad_{|N}\otimes 1)}
(U^{\ad_{|N}}_{\omega_{|N}})_{12}  \widehat W^{\com}_{13}(\eta \otimes x^{\op}\Lambda_{\omega_{|N}}(1)
\otimes \vartheta))= \\
= R(x)\xi \otimes \widehat W^{\com}(\eta \otimes \vartheta)
= \widehat W^{\com}_{23}(V_4 \otimes 1_{H})(\xi
  \rtensor{R_{|N}}{\omega_{|N}}{\ad_{|N}} U^{\ad_{|N}}_{\omega_{|N}}(\eta \otimes
x^{\op}\Lambda_{\omega_{|N}}(1)) \otimes \vartheta).
\end{multline*}
for all $\vartheta\in H$. Therefore,  diagram (**) commutes, the expressions \eqref{eq:cell3-1} and \eqref{eq:cell3-2}
coincide, and
cell (3) commutes.

To see that cell (4) commutes as well, consider the following diagram:
\begin{align*}
  \xymatrix@C=80pt{
    H \rtensor{R_{|N}}{\omega_{|N}}{\ad_{|N}} (H \otimes H_{\omega_{|N}})
    \rtensor{\widehat{\beta}}{\omega_{|N}}{\ad_{|N}} (H \otimes H_{\omega_{|N}}) \ar[r]^(0.55){1
        \otimes V_{5}} \ar[d]_{V_{4} \otimes 1} &
    H \rtensor{R_{|N}}{\omega_{|N}}{\ad_{|N}} (H \otimes H_{\omega_{|N}}) \otimes H
    \ar[d]^{V_{4} \otimes 1} \\
    (H \otimes H) \rtensor{(\Gamma \circ R_{|N})}{\omega_{|N}}{\ad_{|N}} (H\otimes
    H_{\omega_{|N}}) \ar[r]_(0.55){1 \otimes V_4} & H\otimes  H\otimes H
  }
\end{align*}
We show that this diagram commutes, and then cell (4)  commutes as well. We
first compute $(V_4 \otimes 1)(1 \otimes V_5)(\xi \rtensor{R_{|N}}{\omega_{|N}}{\ad_{|N}}
  U^{\ad_{|N}}_{\omega_{|N}}(\eta \otimes x^{\op}\Lambda_{\omega_{|N}}(1))
  \rtensor{\widehat{\beta}}{\omega_{|N}}{\ad_{|N}} U^{\ad_{|N}}_{\omega_{|N}}(\vartheta \otimes
  y^{\op}\Lambda_{\omega_{|N}}(1)))$. 

We use (iii) and find that this vector is equal to
\[ (V_4\otimes 1) (\xi \underset{\omega_{|N}}{_{R_{|N}}\otimes_{\ad_{|N}} } \widehat{\beta}(y)U^{\ad_{|N}}_{\omega_{|N}}(\eta
  \otimes x^{\op}\Lambda_{\omega_{|N}}(1)) \otimes \vartheta) )\]
and therefore
\begin{multline*}
(\Gamma(R(y)) \otimes 1)  (V_{5} \otimes 1)(\xi \rtensor{R_{|N}}{\omega|_{N}}{\ad_{|N}}
U^{\ad_{|N}}_{\omega_{|N}}(\eta \otimes x^{\op}\Lambda_{\omega_{|N}}(1)) \otimes
  \vartheta) 
  =\\=
(\Gamma(R(y)) \otimes 1) (R(x)\xi \otimes \eta \otimes \vartheta).
\end{multline*}
On the other hand,
\begin{align*}
  (1 \otimes V_{4})(V_{4} \otimes 1)(\xi \rtensor{R_{|N}}{\omega_{|N}}{\ad_{|N}}
  U^{\ad_{|N}}_{\omega_{|N}}(\eta \otimes x^{\op}\Lambda_{\omega_{|N}(1)})
  \rtensor{\widehat{\beta}}{\omega_{|N}}{\ad_{|N}} U^{\ad_{|N}}_{\omega_{|N}}(\vartheta \otimes
  y^{\op}\Lambda_{\omega_{|N}}(1))) 
\end{align*}
is equal to
\begin{align*}
  (1 \otimes V_4)((R(x)\xi \otimes \eta) \rtensor{(\Gamma\circ
    R_{|N})}{\omega_{|N}}{\ad_{|N}} U^{\ad_{|N}}_{\omega_{|N}}(\vartheta \otimes y^{\op}\Lambda_\omega(1)))
  &=
\Gamma(R(y))( R(x)\xi\otimes \eta) \otimes \vartheta
\end{align*}
as well, which finishes the proof of (iv). 

(v)
Let $y \in M\cap R(N)'$ and assume
  $\underline{\ga}(y) = y \rtensor{R_{|N}}{N}{\ad_{|N}} 1$. Then by (i), 
  $\Gamma(y)V_{4} = V_{4}(y \rtensor{R_{|N}}{N}{\ad_{|N}} 1)=(y \otimes
  1_{H})V_4$ and hence $\Gamma(y)=y\otimes 1_{H}$, whence $y$ is a scalar
and $\underline{\ga}$ is ergodic. 

The canonical operator-valued weight $T_{\underline{\ga}}$ is equal to $(\id \rast{R_{|N}}{N}{\ad_{|N}} \widehat\Phi) \circ \underline{\ga}$, where $\widehat \Phi=\omega \circ \ad^{-1} \circ T_{\widetilde{\ad_{|N}}}$, and $T_{\widetilde{\ad_{|N}}}$ is the left-invariant weight from  $\widehat {\bfG}
\ltimes_{\ad_{|N}} N$ to $\ad(N)$, i.e.\ the operator-valued weight 
arising from the dual action on $\widehat{ \bfG} \ltimes_{\ad_{|N}} N$, that is,
$(\omega \otimes \id)\circ \widetilde{\ad_{|N}}$. In fact, these operator-valued weights are conditional expectations. 

We write $T_{\widetilde{\ad_{|N}}} = (\id \otimes \omega) \circ
  \varsigma \widetilde{\ad_{|N}}$ and use commutativity of the cells (1) and
  (3) in diagram (*), and find that for any $x\in M^+$,
\begin{align*}
  (\id \rast{R_{|N}}{\omega_{|N}}{\ad_{|N}} T_{\widetilde{\ad_{|N}}}) \circ \underline{\ga}(x)
  &= (\id \rast{R_{|N}}{\omega_{|N}}{\ad_{|N}} T_{\widetilde{\ad_{|N}}}) \circ \ad_{V_4^*} \circ \Gamma (x)\\
  &= ((\id \rast{R_{|N}}{\omega_{|N}}{\ad_{|N}} \id) \otimes \omega) \circ (\id
  \rast{R_{|N}}{\omega_{|N}}{\ad_{|N}} \varsigma\widetilde{\ad_{|N}}) \circ \ad_{V_4^*} \circ
  \Gamma(x)
  \\
  &= ((\id \rast{R_{|N}}{\omega_{|N}}{\ad_{|N}} \id) \otimes \omega) \circ
  (\ad_{V_{4}^{*}} \otimes \id) \circ \Gamma^{(2)} (x)\\
  &= \ad_{V_{4}^{*}} \circ (\id \otimes \id \otimes \omega) \circ
  \Gamma^{(2)}(x)
\\ &  = \ad_{V_{4}^{*}} \circ (1_{M\otimes M} \cdot \omega)(x)\\
&=1_{(M\rast {R_{|N}}{\omega_{|N}}{\ad_{|N}} \widehat{\bfG}\ltimes_{\ad_{|N}} N)} \cdot \omega(x),
  \end{align*}
  where $\Gamma^{(2)}=(\Gamma \otimes \id)\circ \Gamma$ and, for any von Neumann algebra $P$, 
  $1_{P} \cdot \omega$ denotes the positive application
  $x\mapsto \omega(x)1_{P}$. Therefore, we get (v). 
  
 As $\ga$ is integrable and ergodic, by (\cite{E5}, 3.8) or \ref{left} , there exists an isometry $G$ from
  $H\otimes H$ to $H \rtensor{R_{|N}}{\omega_{|N}}{\ad_{|N}} H_{\omega_{|N}}$ such that, for all $\zeta\in D(H_{R|N}, (\omega_{|N})^{\op})$, $x\in M$ and $e\in \widehat{\bfG}\ltimes_{\ad_{|N}}N$,
\[(1 \rtensor{R_{|N}}{N}{\ad_{|N}} J_{\widehat{ \Phi}} e J_{\widehat{ \Phi}})G(x\Lambda_\omega(1)
  \otimes \zeta) = \underline{\ga}(x)( \zeta \rtensor{R_{|N}}{\omega_{|N}}{\ad_{|N}} J_{\widehat{ \Phi}}\Lambda_{\widehat{ \Phi}}(e)).\]
 Let $y^{*} \in M$ and let us take $e=y^*\otimes 1\in \widehat{\bfG}\ltimes_{\ad_{|N}}N$. 
The relation $J_{\widehat{ \Phi}}=U^{\ad_{|N}}_{\omega_{|N}}(J\otimes J_{\omega_{|N}})$ implies
 $J_{\widehat{ \Phi}}eJ_{\widehat{ \Phi}} = U^{\ad_{|N}}_{\omega_{|N}}(y^{\op} \otimes
1)(U^{\ad_{|N}}_{\omega_{|N}})^{*}$ and
\[U^{\ad_{|N}}_{\omega_{|N}}(y^{\op}\Lambda_\omega(1) \otimes \Lambda_{\omega_{|N}}(1))
 =
U^{\ad_{|N}}_{\omega_{|N}}(Jy^{*}\Lambda_\omega(1) \otimes \Lambda_{\omega_{|N}}(1)) =  U^{\ad_{|N}}_{\omega_{|N}}(J\otimes J_{\omega_{|N}})\Lambda_{\widehat{\Phi}}(e) =
 J_{\widehat{\Phi}}\Lambda_{\widehat {\Phi}}(e).\]
We then get that for all $\xi\in H$, $z\in M$, the vector $ (1 \underset{N}{_{R_{|N}}\otimes_{\ad_{|N}}} J_{\widehat{ \Phi}}eJ_{\widehat{ \Phi}})V_4^*(\xi \otimes z\Lambda_\omega(1))$ is equal to
 \begin{multline*}
 (1 \rtensor{R_{|N}}{N}{\ad_{|N}} U^{\ad_{|N}}_{\omega_{|N}}(y^{\op} \otimes
  1)(U^{\ad_{|N}}_{\omega_{|N}})^{*}) (\xi \rtensor{R_{|N}}{\omega_{|N}}{\ad_{|N}}
  U^{\ad_{|N}}_{\omega_{|N}}(z\Lambda_\omega(1)  \otimes \Lambda_{\omega_{|N}}(1))) =\\
 \xi \rtensor{R_{|N}}{\omega_{|N}}{\ad_{|N}} U^{\ad_{|N}}_{\omega_{|N}}(y^{\op}
z\Lambda_\omega(1)  \otimes \Lambda_{\omega_{|N}}(1)))=
 V_4^*(\xi \otimes y^{\op}z\Lambda_\omega(1)).
\end{multline*}
Therefore, 
\begin{align*}
 (1 \rtensor{R_{|N}}{N}{\ad_{|N}} J_{\widehat{ \Phi}}eJ_{\widehat{ \Phi}})V_4^*W^{*}\sigma (x\Lambda_\omega(1) \otimes \zeta) 
 &=
 V_4^{*}(1 \otimes y^{\op})W^{*}(\zeta\otimes x\Lambda_\omega(1)) \\
  &=
   V_4^{*}(1 \otimes y^{\op})\Gamma(x)(\zeta \otimes \Lambda_\omega(1)) \\
 &=
 V_4^{*}\Gamma(x)(\zeta \otimes y^{\op}\Lambda_\omega(1)) \\
 &=
 \underline{\ga}(x)V_4^{*}(\zeta \otimes y^{\op}\Lambda_\omega(1)) \\
 &=
\underline{\ga}(x)(\zeta \rtensor{R_{|N}}{\omega_{|N}}{\ad_{|N}}
U^{\ad_{|N}}_{\omega_{|N}}(y^{\op}\Lambda_\omega(1)\otimes \Lambda_{\omega_{|N}}(1)))\\
&=
 \underline{\ga}(x)(\zeta \rtensor{R_{|N}}{\omega_{|N}}{\ad_{|N}}J_{\widehat{\Phi}}\Lambda_{\widehat {\Phi}}(e)).
\end{align*}
Thus, we get that $(1 \rtensor{R_{|N}}{N}{\ad_{|N}} J_{\widehat{ \Phi}}eJ_{\widehat{ \Phi}})V_4^*W^{*}\sigma=
(1 \rtensor{R_{|N}}{N}{\ad_{|N}} J_{\widehat{ \Phi}}eJ_{\widehat{ \Phi}})G$
for all $e=y^*\otimes 1$, and so $G=V_4^*W^{*}\sigma$.
\end{proof}
\begin{theorem}
\label{ThMorita}
 Let $\bfG=(M, \Gamma, \omega, \omega)$ be a (von Neumann version
  of a) compact quantum group, $\bfG_1$ a compact quantum subgroup,
  and $N$ the quotient type co-ideal. Then the von Neumann algebra
  $M$, equipped with the right Galois action $(R_{|N},
  \underline{\ga})$ of $\widehat{\bfG}\ltimes_{\ad_{|N}}N$ constructed
  in \ref{qgaction} and the left Galois action $\Gamma_l$ of $\bfG_1$
  defined in \ref{QTCI}, is an imprimitivity bimodule which is a
  Morita equivalence between the compact quantum group $\bfG_1$ and
  the measured quantum groupoid $\mathfrak{G}(N,\ad_{|N},\Gamma_{|N},
\omega_{|N})$.
\end{theorem}
\begin{proof}
Let $x\in M$. Commutativity of the cells (1) and (2) in diagram (*) 
implies that
\[(\Gamma\underset{N}{_{R_{|N}}*_{\ad_{|N}}}\id)\underline{\ga}(x)
=(\id\otimes\underline{\ga})\Gamma(x)\]
and applying $(\pi\otimes \id)\underset{N}{_{R_{|N}}*_{\ad_{|N}}}\id$ to this relation, we get:
\[(\Gamma_l\underset{N}{_{R_{|N}}*_{\ad_{|N}}}\id)\underline{\ga}(x)=(\id\otimes \underline{\ga})\Gamma_l(x),\]
which is the commutativity of the right Galois action $(R_{|N}, \underline{\ga})$ of $\widehat{\bfG}\ltimes_{\ad_{|N}}N$ and the left Galois action $\Gamma_l$ of $\bfG_1$. 

Moreover, we had got in \ref{qgaction} that the canonical
operator-valued weight $T_{\underline{\ga}}$ was the Haar state
$\omega$. Let $\omega_1$ be the Haar state of $\bfG_1$. Then the
canonical operator-valued weight $T_{\Gamma_l}$ is equal to
$(\omega_1\circ\pi\otimes \id)\Gamma$, which is, in fact, a
conditional expectation from $M$ into $M^{\Gamma_l}=R(N)$. Composed
with the state $\omega_{|N}\circ R=\omega_{|R(N)}$, we get
$(\omega_1\circ\pi\otimes\omega)\Gamma=\omega_1(\pi(1))\omega=\omega$. Therefore,
using \ref{Morita}, we get the result.
\end{proof}
\begin{corollary}
\label{exSq2}
 The measured quantum groupoid
 $\widehat{\SUq}\ltimes_{\ad_{|S_q^2}}S_q^2$ constructed in \ref{Sq2}
 is Morita equivalent to $\mathbb{T}$. 
\end{corollary}
\begin{proof}
 Apply \ref{ThMorita} to \ref{Sq2}. \end{proof}
\begin{corollary}[\cite{Ri}]
\label{MR}
 Let $G$ be a compact group and $G_1$ a compact subgroup of
  $G$. The the right action of $G$ on $G/G_1$ defines a transformation
  groupoid $(G/G_1)\curvearrowleft G$ and this groupoid is Morita
  equivalent to $G_1$.
\end{corollary}
\begin{proof}
The canonical surjective $*$-homomorphism from $L^\infty (G)$ onto $L^\infty (G_1)$ gives to $L^\infty(G/G_1)$ a structure of a quotient type co-ideal. The restriction of the coproduct $\Gamma_{L^\infty (G)}$ to $L^\infty(G/G_1)$ is just the right action  of $G$ on $G/G_1$, and the measured quantum groupoid $G\ltimes_{\Gamma}L^\infty (G/G_1)$ is the dual of the groupoid $(G/G_1)\curvearrowleft G$. Therefore, by \ref{groupoid}, its dual is just the abelian von Neumann algebra $L^\infty ((G/G_1)\curvearrowleft G)$, and, by \ref{ThMorita}, we get the result. 
\end{proof}

\end{document}